\documentclass{article}
\usepackage[utf8]{inputenc}
\usepackage[T1]{fontenc}
\usepackage{lmodern}
\usepackage{microtype}

\usepackage{amsmath,amssymb,amsfonts,amsthm}
\usepackage{mathtools}
\usepackage{mathrsfs}
\usepackage{dsfont}
\usepackage{bm}
\usepackage{systeme}
\usepackage{derivative}

\usepackage[margin=1in]{geometry}
\usepackage{graphicx}
\usepackage{float}
\usepackage[export]{adjustbox}

\usepackage{enumitem}

\usepackage[colorlinks=true,
            linkcolor=blue,
            citecolor=blue,
            urlcolor=blue]{hyperref}
\usepackage[nameinlink,capitalise]{cleveref}

\numberwithin{equation}{section}

\theoremstyle{plain}
\newtheorem{theorem}{Theorem}[section]
\newtheorem{proposition}[theorem]{Proposition}
\newtheorem{lemma}[theorem]{Lemma}
\newtheorem{corollary}[theorem]{Corollary}

\theoremstyle{definition}
\newtheorem{definition}[theorem]{Definition}

\theoremstyle{remark}
\newtheorem{remark}[theorem]{Remark}

\title{Resonant Geometry and Well-Posedness for the 2D Boussinesq Wave Kinetic Equation}

\author{Haoling Xiang}

\date{}

\begin{document}
\maketitle

\begin{abstract}
We study the wave kinetic equation (WKE) derived by Shavit--B\"uhler--Shatah
from the two-dimensional Boussinesq system of internal waves.  The equation is
anisotropic and vector-valued, with a two-branch sign-changing dispersion
relation, coupled propagation branches, and a sign-indefinite pseudo-momentum
invariant.  The resonant manifold has angular degeneracies that obstruct the
standard analytic treatment of the collision operator.  We introduce a natural
angular cut-off adapted to these degeneracies, prove boundedness of the cut-off
collision operator, and establish local well-posedness in weighted
\(L^\infty\) spaces.  We also show that, even with the zero-frequency cut-off
retained, removing the cut-off near
\(|\cos\theta|=\tfrac12\) makes the collision operator unbounded on these
weighted spaces.  This provides, to our knowledge, the first rigorous analytic
framework for an anisotropic vector-valued WKE arising from Boussinesq
dynamics.
\end{abstract}

\tableofcontents

\section{Introduction}
\subsection{Background on internal waves}\label{section:Background on internal waves}
Wave kinetic equations (WKEs) arise as asymptotic models for weakly nonlinear,
dispersive wave fields whose long-time dynamics is governed by resonant wave
interactions~\cite{hasselmann1962non,nazarenko2011wave}.  A central example
comes from internal gravity waves in the ocean.  In this setting, weak
nonlinearity leads to resonant three-wave interactions in stratified fluids,
and kinetic theory provides a framework for describing the redistribution of
wave energy across scales and directions.

A major observational reference point is the Garrett--Munk spectrum, the
classical empirical spectrum for the distribution of internal-wave energy in
the ocean~\cite{garrett1972space,garrett1975space,munk1981internal}.  One of
the main motivations for internal-wave kinetic theory is to understand whether
nonlinear resonant interactions can help explain, maintain, or modify such
observed energy distributions~\cite{LvovTabak2001,lvov2010oceanic,
nazarenko2011wave}.

At the primitive level, internal gravity waves in a stably stratified fluid are
described by the Boussinesq equations.  Classical internal-wave kinetic theory
has been developed primarily for the \emph{three-dimensional} Boussinesq
system, often in regimes where the hydrostatic approximation is imposed or used
as a leading-order simplification.  Early foundational works include the
studies of M\"uller--Olbers~\cite{muller1975dynamics}, Olbers
~\cite{olbers1976nonlinear}, McComas--Bretherton
~\cite{mccomas1977resonant}, and McComas--M\"uller
~\cite{mccomas1981dynamic,mccomas1981time}.  Later developments include the
kinetic formulation of Caillol--Zeitlin~\cite{caillol2000kinetic} and the
Hamiltonian and wave-turbulence formulations of Lvov--Tabak and collaborators
~\cite{LvovTabak2001,lvov2004hamiltonian,lvov2012resonant}.  More recently,
Labarre--Lanchon--Cortet--Krstulovic--Nazarenko
~\cite{labarre2024kinetics} derived a fully symmetric interaction coefficient
and a kinetic equation beyond the hydrostatic approximation.

Although these works differ in their precise modeling assumptions,
normalizations, and asymptotic regimes---most notably in whether the
hydrostatic approximation is imposed, which changes the dispersion relation
and hence the resonant geometry---they all belong to the classical
three-dimensional internal-wave kinetic theory.  In that setting, the kinetic
description is formulated after the usual modal reduction for the propagating
internal-wave field, and the collision integral is organized around resonant
surfaces in three-dimensional wave-vector space.

\medskip
\subsection{The two-dimensional internal-wave kinetic equation of Shavit--B\"uhler--Shatah}\label{section:The two-dimensional internal-wave kinetic equation of Shavit--B\"uhler--Shatah}
Although the physically relevant problem is three-dimensional, deriving a
closed WKE for the propagating waves is complicated by the additional
zero-frequency vortical branch of the linearized three-dimensional Boussinesq
system; moreover, the geometry of the corresponding wave-resonant manifold is
prohibitively intricate.  We therefore consider the two-dimensional
vertical-plane model, which removes the independent vortical branch and permits
a complete analysis of the resonant geometry and collision operator, while
retaining the essential mathematical and physical features of internal-wave
interactions.

The present paper concerns the two-dimensional Boussinesq wave kinetic equation
introduced by Shavit--B\"uhler--Shatah.  In two dimensions, the linear
internal-wave dynamics has two propagating branches, corresponding to right-
and left-traveling waves.  The kinetic unknown is therefore a coupled
two-component wave-action density
\[
    n(\sigma,k), \qquad \sigma\in\{\pm1\},\quad k=(k_1,k_2)\in\mathbb{R}^2,
\]
with dispersion relation
\[
    \omega_{\sigma,k}=\sigma\,\frac{k_1}{|k|}.
\]
This dispersion relation is homogeneous of degree zero, anisotropic, and
sign-changing.  The two branches are not merely two copies of the same scalar
equation: they interact through the collision operator and must be retained as
distinct components of the kinetic unknown.

A further distinctive feature of the two-dimensional model is the presence of a
second quadratic invariant, the pseudo-momentum.  Unlike the energy, this
invariant is sign-indefinite: the two propagation branches contribute with
opposite signs.  This sign-indefinite structure plays an important role in the
cascade scenarios proposed in~\cite{shavit2024}.

The wave kinetic equation has the form
\[
    \partial_t n = \mathcal{C}(n)(\alpha),
    \qquad \alpha=(\sigma_\alpha,k_\alpha),
\]
where the collision operator couples the two propagation branches
$\sigma=\pm1$ and is supported on the resonant set
\[
    k_\alpha+k_\beta+k_\gamma=0,
    \qquad
    \omega_\alpha+\omega_\beta+\omega_\gamma=0.
\]
For a fixed output mode $\alpha$, this resonant set is a one-dimensional curve
in the remaining wave-vector variables, rather than a resonant surface as in
the three-dimensional theory.  Moreover, because the dispersion relation is
degree-zero homogeneous, anisotropic, and sign-changing, this curve is governed
almost entirely by angular variables and develops strong angular degeneracies.

The physics works of Shavit--B\"uhler--Shatah also discuss several consequences
of this kinetic equation.  In~\cite{shavit2024}, the sign-indefinite
pseudo-momentum is used to motivate new cascade scenarios.  In~\cite{shavit2025},
the authors propose a scale-invariant Kolmogorov-type stationary spectrum for
two-dimensional internal gravity waves.  This spectrum is obtained through
formal constant-flux scaling arguments, together with a regularized treatment
of the angular collision integral and a numerical reduction of the directional
dynamics.  In particular, the proposed spectrum should be understood as a
formal and numerically supported stationary-cascade candidate, rather than as
a rigorous solution theory for the kinetic equation.

More recently, the same authors extended this framework to the
two-dimensional rotating--stratified Boussinesq system and derived the
corresponding inertia--gravity-wave kinetic equation after restricting the
dynamics to the wave manifold~\cite{shavit2026rotating}.

From the viewpoint of analysis, however, the full two-component kinetic
equation remains open.  The previous works do not provide a function space in
which the collision operator is shown to be well defined, bounded, or
continuous, nor do they establish well-posedness or stability for the proposed
spectra.  In particular, the full coupled equation for
\[
    (n(+,\cdot),n(-,\cdot))
\]
has not previously been studied in a rigorous analytic framework.  The purpose
of the present paper is to provide such a framework for the cut-off
two-dimensional Boussinesq wave kinetic equation.

\medskip

\subsection{Mathematical wave kinetic theory}\label{sec:Mathematical wave kinetic theory}
\subsubsection*{Derivation from microscopic dynamics}

The rigorous mathematical study of wave turbulence for dispersive Hamiltonian
PDEs was initiated by Buckmaster--Germain--Hani--Shatah, who identified the
first kinetic correction predicted by the WKE for random-phase cubic NLS on
long, but sub-kinetic, time scales
~\cite{BuckmasterGermainHaniShatah2021}.  Subsequent works of Deng--Hani and
Collot--Germain reached near-kinetic regimes, up to arbitrarily small
polynomial losses in favorable scalings, and clarified the roles of the
scaling law and the dispersion relation
~\cite{deng2021derivation,collot2025homogeneous,collot2025longer}.
Deng--Hani then achieved a full derivation at the kinetic time scale under a
distinguished scaling; in companion work they also proved propagation of chaos
and identified the limiting higher-order statistics
~\cite{deng2023full,deng2021propagation}.  Their later long-time
theorem extended the kinetic approximation throughout any interval on which
the corresponding WKE solution exists~\cite{deng2023long}.  Related
long-time combinatorial methods were subsequently adapted by Deng--Hani--Ma in
their derivation of the Boltzmann equation from hard-sphere dynamics and in
subsequent work on Hilbert's sixth problem
~\cite{deng2024long,deng2025hilbert}.

In the space-inhomogeneous setting, Ampatzoglou--Collot--Germain rigorously
established the validity of the three-wave kinetic description for a model
quadratic dispersive equation, up to an arbitrarily small polynomial loss from
the kinetic time scale
~\cite{ampatzoglou2025derivation}.

Several recent works address one-dimensional settings.  Vassilev justified
the kinetic prediction for the MMT model on sub-kinetic time scales, while Wu
treated a reduced \(\beta\)-FPUT evolution with the non-resonant terms removed.
Vassilev--Wu subsequently incorporated these terms into the diagrammatic
expansion for the full \(\beta\)-FPUT system, reaching times of order
\(T_{\mathrm{kin}}^{2/3}\)
~\cite{vassilev2025,wu2025,VassilevWu2026}.  For the two-dimensional gravity
water-wave system, with a one-dimensional interface,
Deng--Ionescu--Pusateri combined deterministic energy estimates, normal forms,
and propagation of randomness to construct random solutions on large tori up
to the sub-kinetic scale \(\varepsilon^{-8/3+}\)
~\cite{DengIonescuPusateriGravityI,DengIonescuPusateriGravityII}; these works
do not yet derive the water-wave WKE.  In these one-dimensional settings,
unfavorable lattice counting and the resulting growth of Feynman diagrams
prevent the current methods from reaching the full kinetic time; for water
waves this difficulty is compounded by quasilinear derivative loss.  A
derivation of the two-dimensional Boussinesq WKE would face additional
difficulties of a different kind, including its two-branch anisotropic
interaction, degree-zero dispersion, and singular three-wave resonance
geometry.  We therefore take the WKE as our starting point.

\medskip
\subsubsection*{Well-posedness and qualitative dynamics}

Once a WKE has been derived or formally proposed, a separate question is
whether its collision operator generates a meaningful nonlinear evolution.
For homogeneous four-wave equations with radial dispersion laws,
Germain--Ionescu--Tran established an optimal local well-posedness theory in
nearly critical weighted spaces~\cite{GermainIonescuTran2020}.  For the kinetic
MMT equation, Germain--La--Zhang developed a local well-posedness theory and
uncovered a nonlinear smoothing mechanism~\cite{GermainLaZhang2025}.

For the space-inhomogeneous WKE associated with cubic NLS, Ampatzoglou proved global well-posedness and stability for small mild solutions
near vacuum~\cite{ampatzoglou2022global}.  Ampatzoglou--Miller--Pavlovi\'c--
Taskovi\'c extended the small-data global theory to polynomially weighted
\(L^\infty\) spaces and established global well-posedness for the associated
hierarchy~\cite{AmpatzoglouMillerPavlovicTaskovic2025}, while
Ampatzoglou--L\'eger constructed small global strong dispersive solutions and
developed a scattering theory~\cite{ampatzoglou2024scattering}.

Beyond existence, Menegaki proved nonlinear \(L^2\)-stability near
Rayleigh--Jeans equilibria for a frequency-cutoff model and dispersions weakly
perturbed from the quadratic case~\cite{Menegaki2024}.  By contrast,
Escobedo--Menegaki proved nonlinear instability of the singular
Rayleigh--Jeans equilibrium through finite-time formation of a Dirac mass at
zero frequency, and analyzed the corresponding linearized concentration
dynamics~\cite{EscobedoMenegaki2024}.  Kolmogorov--Zakharov spectra instead
describe non-equilibrium cascades.  In the isotropic stationary setting,
Collot--Dietert--Germain proved nonlinear stability of the mass-cascade
Kolmogorov--Zakharov spectrum and constructed non-equilibrium steady states
exhibiting an inverse mass cascade and a direct energy cascade
~\cite{CollotDietertGermain2024}.

For the kinetic FPU equation with phonon wave number \(p\in\mathbb T\),
Lukkarinen--Spohn established anomalous decay of equilibrium energy-current
correlations through spectral analysis of the linearized collision operator, while Germain--La--Menegaki subsequently proved nonlinear stability of nonsingular Rayleigh--Jeans equilibria in the spatially homogeneous setting ~\cite{lukkarinen2008anomalous,germain2026stability}. Escobedo--Germain--La--Menegaki characterized entropy maximizers for kinetic wave equations, including condensation regimes ~\cite{escobedo2025entropy}. For the spatially inhomogeneous kinetic FPU equation, the author of the present
paper exploited dispersion of the transport flow to extend the lifespan of small solutions near the vacuum from the quadratic to the quartic time scale~\cite{xiang2025long}.

Very recently, Pan--Wu constructed local-in-time \(L^1\) strong solutions to
the four-wave gravity-water-wave WKE for initial data in suitably weighted
\(L^2\cap L^\infty\) spaces, using a sharp kernel bound in a highly nonlocal
regime and a refined structural decomposition of the collision operator
~\cite{PanWu2026}.  Although their radially nonlocal regime differs from our
angular zero-frequency endpoints, their \(L^1\)-based formulation suggests
exploring a non-cutoff zero-frequency evolution in \(L^1\), combined with the
energy structure identified in Section~\ref{sec:zero-frequency-cutoff}, rather
than a direct Banach fixed-point argument on \(L^\infty_m\).  This possible
route leaves the nonzero-frequency high--low--high cusp unaffected.

Ampatzoglou--L\'eger identified the sharp threshold \(\beta=\tfrac14\) between
local well-posedness and ill-posedness in weighted \(L^\infty\) spaces for a
family of scalar isotropic four-wave equations.  The gain-only and full
equations have the same threshold, but the full-equation counterexample uses
oscillatory initial data to overcome gain--loss cancellation
~\cite{AmpatzoglouLegerIllposed2025}.  Our mechanism is instead a three-wave,
anisotropic high--low--high interaction at a quadratic cusp of the resonance
curve at nonzero frequency.  Carefully chosen nonnegative, angularly localized
wave packets isolate a positive unbounded contribution in the full collision
operator, showing that the angular cut-off near
\(|\cos\theta|=\tfrac12\) is necessary for the collision map to define a
bounded quadratic vector field on the weighted \(L^\infty_m\) spaces.

We now turn to the two-dimensional Boussinesq WKE.  Its formal derivation from
the underlying Hamiltonian PDE is recorded in
Appendix~\ref{Appendix: Formal derivation}.  The next subsections define the
collision operator and the angular cut-off adapted to its zero-frequency
endpoints and nonzero-frequency high--low--high cusp, and then state the main
theorem; the position and contributions of the present work are discussed in
Section~\ref{subsec:organization-article}.
\subsection{The formal WKE}\label{section:WKE}

In this section we introduce the formal wave kinetic equation (WKE) associated
with two-dimensional internal waves.  The unknown is the wave-action density
\(n_\alpha(t)\), indexed by the multi-index
\[
    \alpha=(\sigma_\alpha,k_\alpha)
    =(\sigma_\alpha,k_\alpha^1,k_\alpha^2)
    \in\{\pm1\}\times\mathbb R^2,
\]
where \(\sigma_\alpha\in\{\pm1\}\) denotes the frequency branch and
\(k_\alpha\in\mathbb R^2\) is the wave vector.  We write
\[
    k_\alpha=|k_\alpha|(\cos\theta_\alpha,\sin\theta_\alpha),
    \qquad
    \theta_\alpha\in\mathbb R/2\pi\mathbb Z,
\]
and use the same convention for \(k_\beta\) and \(k_\gamma\).

Throughout the main text, we normalize the buoyancy frequency to one, which is
always possible after rescaling time.

Following Shavit--B\"uhler--Shatah~\cite{shavit2024}, the formal WKE is
\[
    \partial_t n_\alpha=\mathcal C(n)(\alpha),
\]
where the collision operator is
\begin{equation}\label{eq:WKE}
\begin{aligned}
\mathcal C(n)(\alpha)
    &=
    \int
    \Gamma_{\alpha\beta\gamma}^{\,2}\,\omega_\alpha
    \Bigl(
        \omega_\alpha n_\beta n_\gamma
        +\omega_\gamma n_\alpha n_\beta
        +\omega_\beta n_\alpha n_\gamma
    \Bigr)  \\
    &\qquad\qquad\times
    \delta(k_\alpha+k_\beta+k_\gamma)\,
    \delta(\omega_\alpha+\omega_\beta+\omega_\gamma)
    \,d\beta\,d\gamma .
\end{aligned}
\end{equation}
Here the integration runs over both frequency branches and all wave vectors:
\[
    \int d\alpha
    :=
    \sum_{\sigma_\alpha=\pm1}\int_{\mathbb R^2} dk_\alpha,
\]
The interaction coefficient is denoted by \(\Gamma_{\alpha\beta\gamma}\) and, with this normalization, is given by
\[
    \Gamma_{\alpha\beta\gamma}
    =
    \pi^2
    \bigl(
        \sigma_\alpha\sin\theta_\alpha
        +\sigma_\beta\sin\theta_\beta
        +\sigma_\gamma\sin\theta_\gamma
    \bigr)
    \bigl(
        \sigma_\alpha |k_\alpha|
        +\sigma_\beta |k_\beta|
        +\sigma_\gamma |k_\gamma|
    \bigr).
\]
When no confusion is possible, we write \(\Gamma\) for
\(\Gamma_{\alpha\beta\gamma}\).  The dispersion relation is
\[
    \omega_\alpha
    =
    \sigma_\alpha \frac{k_\alpha^1}{|k_\alpha|}
    =
    \sigma_\alpha \cos\theta_\alpha .
\]

We introduce the resonance functions
\[
    \Xi_2(\alpha,\beta,\gamma)
    :=
    k_\alpha+k_\beta+k_\gamma
\]
and
\[
    \Xi_1(\alpha,\beta,\gamma)
    :=\omega_\alpha+\omega_\beta+\omega_\gamma
   .
\]

Thus the delta functions in \eqref{eq:WKE} impose
\[
    \Xi_2=0,
    \qquad
    \Xi_1=0,
\]
corresponding respectively to momentum conservation and three-wave frequency
resonance.

At this point \eqref{eq:WKE} should be understood as a formal collision
integral.  The product of the two delta constraints and the possible
degeneracies of the resonant manifold have not yet been justified
analytically.  The angular cut-off formulation introduced below is designed to place
this collision operator in a regime where the co-area representation is
well defined and the resulting nonlinear operator can be estimated in weighted
\(L^\infty\) spaces.
%

\subsection{WKE with angular cut-off}
\label{cut-off}

We introduce the angular cut-off used throughout the paper.  The cut-off is
inserted symmetrically into the collision kernel, with one factor for each
member of a resonant triad, rather than imposed on the unknown \(n\).  This
preserves the permutation symmetries of the collision integrand and hence the
algebraic cancellations underlying the conservation laws and the
\(H\)-theorem.

We briefly describe the two geometric mechanisms removed by the cut-off.  The
sign configurations and the precise scaling of the resonant manifolds are
introduced in Sections~\ref{Symmetry of the resonance manifold}
and~\ref{Scaling and co–area representation}.

\paragraph{Collapsed endpoints.}

The resonance relations are invariant under a common dilation of the three
wave vectors.  Thus, along a resonant family with a large wave-number scale
\(R\), one may divide all three wave vectors by \(R\).  We call a limiting
point a \emph{collapsed endpoint} if one of the resulting normalized wave
vectors tends to zero.  Equivalently, one member of the triad is asymptotically
small relative to the other two.

Since the dispersion relation is homogeneous of degree zero and is not defined
at the zero wave vector, the angle of a collapsed mode denotes only a limiting
approach direction.  Different limiting angular values may therefore
correspond to the same collapsed Cartesian point.  Several such endpoints
appear in the explicit parametrizations below, but only the two mechanisms
described next require angular localization.

\paragraph{The zero-frequency localization.}

The first excluded set is
\[
    \cos\theta=0.
\]
Indeed, for every nonzero wave vector,
\[
    \omega_{\sigma,k}
    =
    \sigma\frac{k_1}{|k|}
    =
    \sigma\cos\theta,
\]
so every nonzero vertical wave vector has zero frequency.  Such modes carry no
fast oscillatory phase, and the averaging underlying the formal wave-kinetic
closure is not expected to remain uniform near this set.

Because the dispersion relation is homogeneous of degree zero, the
zero-frequency degeneration is angular rather than radial.  It occurs along
the entire vertical rays
\[
    \{(0,k_2):k_2\neq0\},
\]
and therefore persists at arbitrarily large wave-number magnitude; it cannot
be removed by radial decay alone.

The same degeneration has a more precise geometric manifestation in the
co--area representation.  At a representative non-collapsed zero-frequency
critical configuration, the reduced resonance function has an indefinite
quadratic part.  Consequently, two one-dimensional resonant branches cross at
a point where the co--area Jacobian \(J\) vanishes.  If \(\varrho\) denotes
the distance to this critical point, then along either branch
\[
    J\sim\varrho,
    \qquad
    d\mathcal H^1\sim d\varrho,
\]
and hence the co--area factor degenerates logarithmically:
\[
    \frac{d\mathcal H^1}{J}
    \sim
    \frac{d\varrho}{\varrho}.
\]

Zero-frequency directions also occur as limiting directions of certain
collapsed modes.  In these endpoint configurations, the leading high-pair
exchange and the associated coefficient estimates are not uniformly
controlled by the weighted \(L^\infty_m\) vector-field argument used in this
paper.  The non-collapsed saddle-type co--area degeneration and the
collapsed-endpoint high-pair exchange are analyzed in
Section~\ref{sec:zero-frequency-cutoff}; together they motivate the
localization away from \(\cos\theta=0\).

\paragraph{The finite-frequency high--low--high cusp.}

The second excluded set is
\[
    |\cos\theta|=\frac12.
\]
This is a finite-frequency degeneration.  One of the resonant families
classified in Section~\ref{Symmetry of the resonance manifold} contains, in a
representative labeling, a collapsed low mode \(k_\beta\).  After division by
the high scale \(|k_\alpha|\), the corresponding normalized low wave vector
reaches zero precisely when
\[
    |\cos\theta_\alpha|\le\frac12.
\]

For
\[
    |\cos\theta_\alpha|<\frac12,
\]
two distinct limiting values of \(\theta_\beta\) correspond to the same
collapsed point, and the incident arcs have different one-sided tangents.  We
refer to this nonsmooth collapsed point as a
\emph{high--low--high cusp}.  The exchanged labeling gives the analogous cusp
with \(k_\gamma\) as the low mode.

The terminology refers to wave-number magnitude.  Restoring the physical scale
\(R=|k_\alpha|\), a family approaching the cusp satisfies, up to exchanging
\(\beta\) and \(\gamma\),
\[
    |k_\alpha|\sim|k_\gamma|\sim R,
    \qquad
    |k_\beta|=O(1),
    \qquad
    R\to\infty.
\]

At the critical boundary
\[
    |\cos\theta_\alpha|=\frac12,
\]
the two limiting angular directions coalesce.  In the representative case
\[
    \theta_\alpha=\frac{\pi}{3},
    \qquad
    \theta_\beta\to\pi,
\]
set
\[
    \rho:=\frac{|k_\beta|}{|k_\alpha|},
    \qquad
    \eta:=\theta_\beta-\pi.
\]
The local resonance relation gives
\[
    \rho\sim\eta^2.
\]
Thus the low mode collapses quadratically.  This critical quadratic cusp is
responsible for the failure of the weighted \(L^\infty_m\) vector-field bound
proved in Section~\ref{Section: Unboundedness}.

The same resonant family also contains another collapsed endpoint, at which
the other non-output mode vanishes.  That endpoint is controlled by the
additional vanishing factors in the interaction coefficient and does not
require the localization near \(|\cos\theta|=\frac12\).  The remaining
collapsed endpoints are likewise either removed by the zero-frequency
localization or controlled in the branchwise estimates.

\paragraph{Definition of the cut-off operator.}

Fix \(0<\varepsilon\ll1\).  In polar coordinates
\[
    k=|k|(\cos\theta,\sin\theta),
\]
choose a smooth \(2\pi\)-periodic function
\[
    \chi_\varepsilon\in C^\infty(\mathbb R/2\pi\mathbb Z),
    \qquad
    0\le\chi_\varepsilon\le1,
\]
such that
\[
    \chi_\varepsilon(\theta)=0
    \quad\text{if}\quad
    |\cos\theta|<\varepsilon
    \quad\text{or}\quad
    \bigl||\cos\theta|-\tfrac12\bigr|<\varepsilon,
\]
and
\[
    \chi_\varepsilon(\theta)=1
    \quad\text{if}\quad
    |\cos\theta|>2\varepsilon
    \quad\text{and}\quad
    \bigl||\cos\theta|-\tfrac12\bigr|>2\varepsilon.
\]
Between these regions, \(\chi_\varepsilon\) is chosen to interpolate smoothly.
We take it to depend only on \(|\cos\theta|\).  In particular,
\[
    \chi_\varepsilon(-\theta)=\chi_\varepsilon(\theta),
    \qquad
    \chi_\varepsilon(\theta+\pi)=\chi_\varepsilon(\theta),
\]
so that the cut-off preserves the reflection and half-turn symmetries of the
resonance geometry.

The cut-off collision operator is defined by
\begin{align}
\label{eq:WKE-cutoff-symmetric}
\mathcal C_\varepsilon(n)(\alpha)
&=
\int
\Gamma_{\alpha\beta\gamma}^{2}\,\omega_\alpha\,
\chi_\varepsilon(\theta_\alpha)
\chi_\varepsilon(\theta_\beta)
\chi_\varepsilon(\theta_\gamma)
\Bigl(
    \omega_\alpha n_\beta n_\gamma
    +\omega_\gamma n_\alpha n_\beta
    +\omega_\beta n_\alpha n_\gamma
\Bigr)
\notag\\
&\qquad\qquad\times
\delta(k_\alpha+k_\beta+k_\gamma)\,
\delta(\omega_\alpha+\omega_\beta+\omega_\gamma)
\,d\beta\,d\gamma .
\end{align}
The corresponding cut-off WKE is
\[
    \partial_t n_\alpha
    =
    \mathcal C_\varepsilon(n)(\alpha).
\]
When \(\varepsilon\) is fixed, we often write
\(\mathcal C\) instead of \(\mathcal C_\varepsilon\).

The two localizations have different roles: the first excludes the
zero-frequency sector, while the second excludes the critical quadratic
high--low--high cusp.

This formulation is analogous in spirit to the classical Grad angular cut-off
for the Boltzmann equation.  In the Boltzmann setting, the angular restriction
is imposed on the collision kernel and not on the distribution function itself;
this preserves the algebraic conservation structure while removing the
singular angular regime.  Grad's angular cut-off is a standard hypothesis in
much of the classical mathematical theory of the Boltzmann equation; see, for
example, Grad~\cite{Grad1958}, Cercignani~\cite{Cercignani1988},
DiPerna--Lions~\cite{DiPernaLions1989}, and
Villani~\cite{Villani2002}.  The complementary non-cutoff theory, in which the
angular singularity is retained, has a different analytic character; see, for
instance,
Alexandre--Desvillettes--Villani--Wennberg
~\cite{AlexandreDesvillettesVillaniWennberg2000}.

All estimates below are for fixed \(\varepsilon>0\).  The constants may depend
on \(\varepsilon\) and may diverge as \(\varepsilon\to0\); no uniform removal
of either angular localization is proved here.


\subsection{Main Results}\label{Main Results}

\begin{definition}
For $m\in\mathbb{R}$ we define the weighted space
\[
L^\infty_m
:= \Bigl\{ n(\sigma,k):\{\pm1\}\times\mathbb{R}^2\to\mathbb{R}_{\ge0} \;\Big|\;
      \|n\|_{L^\infty_m}<\infty \Bigr\},
\]
with norm
\[
\|n\|_{L^\infty_m}
:= \sup_{\sigma=\pm1}\sup_{k\in\mathbb{R}^2}
   \langle k\rangle^{m}\,|n(\sigma,k)|.
\]
\end{definition}

\begin{theorem}[Boundedness of the collision operator]
\label{thm:bounded-C}
Consider the cut-off WKE~\eqref{eq:WKE-cutoff-symmetric} and let 
$\mathcal{C}$ denote the associated collision operator.  
Fix $0<\varepsilon\ll1$. Then for every $m>4$ the operator $\mathcal{C}$ is 
bounded on $L^\infty_m$ in the sense that
\[
\sup_{\sigma=\pm1}\sup_{k\in\mathbb{R}^2}
\langle k\rangle^{m}\,
|\mathcal{C}(n)(\sigma,k)|
\;\lesssim_\varepsilon\;
\|n\|_{L^\infty_m}^{\,2},
\qquad n\in L^\infty_m.
\]
In particular, $\mathcal{C}:L^\infty_m\to L^\infty_m$ for all $m>4$.
\end{theorem}

\begin{corollary}[Local well-posedness in $L^\infty_m$]
\label{cor:LWP}
Let $m>4$ and $0<\varepsilon\ll1$ be fixed.  
Then the cut-off WKE~\eqref{eq:WKE-cutoff-symmetric}
\[
\partial_t n = \mathcal{C}(n)
\]
is locally well posed in $L^\infty_m$ in the following sense:  
for any initial data $n_0\in L^\infty_m$ there exists $T>0$ and a unique 
solution
\[
n \in C\bigl([0,T],L^\infty_m\bigr)
\]
with $n(0)=n_0$, and the map $n_0\mapsto n$ is continuous from $L^\infty_m$ to 
$C([0,T],L^\infty_m)$.  
Moreover, the lifespan $T$ can be chosen to depend only on 
$\varepsilon$, $m$ and $\|n_0\|_{L^\infty_m}$.
\end{corollary}

\paragraph{Contributions and scope.}

Theorem~\ref{thm:bounded-C} and Corollary~\ref{cor:LWP} place the
angularly cut-off Boussinesq WKE in a standard Banach-space ODE framework.
For each fixed cut-off scale and every \(m>4\), the collision operator is a
bounded quadratic map on \(L^\infty_m\), and the full two-component equation
is locally well posed there.  The proof is fundamentally geometric: it
requires output-uniform estimates for anisotropic resonance curves with
non-compact ends, collapsed endpoints, and a finite-frequency cusp.  The
symmetric placement of the cut-off also preserves the algebraic identities
behind the formal conservation laws, the \(H\)-theorem, and the
Rayleigh--Jeans equilibria; see
Section~\ref{sec:structural-properties-resonance-geometry}.

The paper also explains why the two angular localizations are not merely
technical.  Section~\ref{Section: Unboundedness} shows that removing the
finite-frequency localization near \(|\cos\theta|=\frac12\), while retaining
the zero-frequency cut-off, makes the collision map unbounded on the weighted
\(L^\infty_m\) spaces.  This obstruction comes from a critical quadratic
high--low--high cusp.  Section~\ref{sec:zero-frequency-cutoff} identifies a
different obstruction near \(\cos\theta=0\): the leading high--high exchange
is not a bounded \(L^\infty_m\) vector field, although its high-pair block has
a favorable unweighted \(L^2\) structure.  Treating this regime without a
cut-off would require a different, energy-level framework.

Finally, Section~\ref{sec:related-Hamiltonian-models} compares the analysis
with two related Hamiltonian Boussinesq models.  For the rotating--stratified
model, the frequency gap removes the zero-frequency wave degeneration, and
the remaining finite-frequency cusp suggests a natural rotating cut-off
theory for fixed \(0<f<\frac12\); the full rotating collision estimates are
left for future work.  For the hydrostatic Euler--Boussinesq model, we prove
the analogue of the cut-off \(L^\infty_m\) estimate and local well-posedness.
There the required angular localizations are near \(\cos\theta=0\) and
\(\sin\theta=0\), and no localization near
\(|\cos\theta|=\frac12\) is needed.  Thus the paper gives a first analytic
framework for the anisotropic two-branch Boussinesq WKE and shows how the
relevant angular obstructions change under related Hamiltonian variants.

\subsection{Organization of the article}
\label{subsec:organization-article}

Section~\ref{sec:structural-properties-resonance-geometry} records the
structural identities of the angularly cut-off WKE, including the conservation
laws, \(H\)-theorem, Rayleigh--Jeans equilibria, symmetry reductions, and the
scaled co--area representation.  Sections~\ref{Parametrization and Estimates
for the $(+,+,+)$ Resonance Manifold}--\ref{Parametrization and Estimates for
the $(+,-,-)$ Resonance Manifold} carry out the branch-by-branch resonance
analysis and prove the collision bounds for the three representative sign
configurations.  Section~\ref{Section: Unboundedness} proves the
high--low--high cusp obstruction when the cut-off near
\(|\cos\theta|=\frac12\) is removed.  Section~\ref{sec:zero-frequency-cutoff}
studies the distinct zero-frequency exchange mechanism.  Section
~\ref{sec:related-Hamiltonian-models} discusses the rotating and hydrostatic
Hamiltonian variants.  Appendix~\ref{Appendix: Formal derivation} gives the
formal derivations of the non-hydrostatic and hydrostatic WKEs, and
Appendix~\ref{appendix:additional-results} collects auxiliary calculus
estimates.

\section{Structural Properties and Resonance Geometry}

\label{sec:structural-properties-resonance-geometry}

\subsection{The two-dimensional Boussinesq system and its invariants}
\label{subsec:2d-boussinesq-invariants}

Let \(x=(x_1,x_2)\in\mathbb R^2\), where \(x_1\) is the horizontal
coordinate and \(x_2\) is the vertical coordinate.  We write the inviscid
Boussinesq equations in perturbation variables about a motionless, uniformly
stratified hydrostatic equilibrium.  After normalizing the buoyancy frequency
to one by rescaling time, the exact nonlinear dynamics in a vertical plane are
\begin{equation}
\label{eq:2d-boussinesq-system}
\begin{cases}
(\partial_t+u\cdot\nabla)u=-\nabla p+b\,e_2,\\[0.3em]
(\partial_t+u\cdot\nabla)b+u_2=0,\\[0.3em]
\nabla\cdot u=0,
\end{cases}
\end{equation}
where \(u=(u_1,u_2)\), \(b\) is the buoyancy perturbation, \(p\) is the
pressure perturbation, and \(e_2=(0,1)\).  Here ``perturbation'' refers only
to the subtraction of the background equilibrium: all nonlinear transport
terms are retained, and no linearization or small-amplitude approximation has
been made.  System~\eqref{eq:2d-boussinesq-system} is the vertical-plane
reduction of the usual three-dimensional stratified Boussinesq equations.

For smooth solutions with sufficient decay at spatial infinity,
\eqref{eq:2d-boussinesq-system} conserves the energy
\begin{equation}
\label{eq:2d-boussinesq-energy}
    \mathcal E_{\mathrm B}
    :=
    \frac12\int_{\mathbb R^2}
    \bigl(|u|^2+b^2\bigr)\,dx
\end{equation}
and, with our sign convention, the horizontal pseudo-momentum
\begin{equation}
\label{eq:2d-boussinesq-pseudomomentum}
    \mathcal P_{\mathrm B}
    :=
    \int_{\mathbb R^2}
    b\bigl(\partial_2u_1-\partial_1u_2\bigr)\,dx .
\end{equation}
The first is the vertical-plane restriction of the standard
three-dimensional Boussinesq energy, while the second is the corresponding
horizontal pseudo-momentum invariant for the two-dimensional directional-wave
dynamics.

The normal-mode calculation in
Appendix~\ref{app:boussinesq-modes} diagonalizes these two quadratic
invariants.  Away from the zero-frequency set \(k^1=0\), their modal weights
are, up to fixed normalization constants,
\[
    1
    \qquad\text{and}\qquad
    s_{\sigma,k}
    :=
    \frac{k^1}{\omega_{\sigma,k}}
    =
    \sigma|k|,
\]
respectively.  After ensemble averaging and passage to the continuum kinetic
description, the corresponding quantities are
\[
    \int n_\alpha\,d\alpha
    \qquad\text{and}\qquad
    \int s_\alpha n_\alpha\,d\alpha.
\]
Thus the collision invariants established below are precisely the kinetic
counterparts of the two quadratic invariants of the underlying nonlinear
Boussinesq system.

\subsection{Conservation laws, equilibria, and the \(H\)-theorem}
\label{Conservation Laws}

Having identified the PDE origin of the two quadratic invariants, we now
verify directly that the symmetric angular cutoff preserves their kinetic
counterparts, together with the \(H\)-theorem.

Set
\[
    \chi_\alpha:=\chi_\varepsilon(\theta_\alpha),
    \qquad
    \chi_\beta:=\chi_\varepsilon(\theta_\beta),
    \qquad
    \chi_\gamma:=\chi_\varepsilon(\theta_\gamma),
\]
and introduce the symmetric nonnegative resonant measure
\[
\begin{aligned}
d\mu_\varepsilon(\alpha,\beta,\gamma)
:={}&
\chi_\alpha\chi_\beta\chi_\gamma
\Gamma_{\alpha\beta\gamma}^{2}
\delta(\Xi_1)\delta(\Xi_2)
\,d\alpha\,d\beta\,d\gamma .
\end{aligned}
\]
Assume throughout this subsection that \(n\) is sufficiently regular,
strictly positive, and decays fast enough for the following computations to
be justified.

Since \(d\mu_\varepsilon\) is invariant under permutations of
\((\alpha,\beta,\gamma)\), symmetrization gives, for every sufficiently
regular test function \(\Psi_\alpha\),
\begin{equation}
\label{eq:CL-general-cutoff}
\begin{aligned}
\frac{d}{dt}\int \Psi_\alpha n_\alpha\,d\alpha
&=
\frac13
\iiint
\bigl(
    \omega_\alpha\Psi_\alpha
    +
    \omega_\beta\Psi_\beta
    +
    \omega_\gamma\Psi_\gamma
\bigr)
\\
&\qquad\qquad\times
\bigl(
    \omega_\alpha n_\beta n_\gamma
    +
    \omega_\beta n_\alpha n_\gamma
    +
    \omega_\gamma n_\alpha n_\beta
\bigr)
\,d\mu_\varepsilon(\alpha,\beta,\gamma).
\end{aligned}
\end{equation}
The factor \(1/3\) comes from averaging the three equivalent expressions
obtained by taking \(\alpha\), \(\beta\), or \(\gamma\) as the distinguished
output variable.

Consequently, every test function satisfying
\begin{equation}
\label{eq:collision-invariant-condition}
    \omega_\alpha\Psi_\alpha
    +
    \omega_\beta\Psi_\beta
    +
    \omega_\gamma\Psi_\gamma
    =
    0
\end{equation}
on the resonant set generates a conserved quantity.

\paragraph{Wave energy.}
Taking
\[
    \Psi_\alpha\equiv1,
\]
condition~\eqref{eq:collision-invariant-condition} is precisely the frequency
resonance relation
\[
    \omega_\alpha+\omega_\beta+\omega_\gamma=0.
\]
Hence the total wave energy
\begin{equation}
\label{eq:energy-conservation-cutoff}
    E(t):=\int n_\alpha(t)\,d\alpha
\end{equation}
is conserved:
\[
    \frac{dE}{dt}=0.
\]

\paragraph{Horizontal pseudo-momentum.}
Define the horizontal slowness by
\[
    s_\alpha:=\frac{k_\alpha^1}{\omega_\alpha}.
\]
Under the normalization \(N=1\), and away from the zero-frequency set,
\[
    s_\alpha
    =
    \frac{k_\alpha^1}
         {\sigma_\alpha k_\alpha^1/|k_\alpha|}
    =
    \sigma_\alpha|k_\alpha|.
\]
Since
\[
    \omega_\alpha s_\alpha=k_\alpha^1,
\]
the momentum resonance gives
\[
\begin{aligned}
    \omega_\alpha s_\alpha
    +
    \omega_\beta s_\beta
    +
    \omega_\gamma s_\gamma
    &=
    k_\alpha^1+k_\beta^1+k_\gamma^1
    \\
    &=0.
\end{aligned}
\]
Therefore the horizontal pseudo-momentum
\begin{equation}
\label{eq:pm-conservation-cutoff}
    PM(t):=\int s_\alpha n_\alpha(t)\,d\alpha
\end{equation}
is conserved:
\[
    \frac{d}{dt}PM(t)=0.
\]

\paragraph{\(H\)-theorem.}
For strictly positive \(n\), the same symmetrization gives
\begin{align}
\frac{d}{dt}\int \log n_\alpha\,d\alpha
&=
\frac13
\iiint
\left(
    \frac{\omega_\alpha}{n_\alpha}
    +
    \frac{\omega_\beta}{n_\beta}
    +
    \frac{\omega_\gamma}{n_\gamma}
\right)
\notag\\
&\qquad\qquad\times
\left(
    \omega_\alpha n_\beta n_\gamma
    +
    \omega_\beta n_\alpha n_\gamma
    +
    \omega_\gamma n_\alpha n_\beta
\right)
\,d\mu_\varepsilon .
\label{eq:entropy-general-cutoff}
\end{align}
Since
\[
\begin{aligned}
&\left(
    \frac{\omega_\alpha}{n_\alpha}
    +
    \frac{\omega_\beta}{n_\beta}
    +
    \frac{\omega_\gamma}{n_\gamma}
\right)
\left(
    \omega_\alpha n_\beta n_\gamma
    +
    \omega_\beta n_\alpha n_\gamma
    +
    \omega_\gamma n_\alpha n_\beta
\right)
\\
&\qquad
=
n_\alpha n_\beta n_\gamma
\left(
    \frac{\omega_\alpha}{n_\alpha}
    +
    \frac{\omega_\beta}{n_\beta}
    +
    \frac{\omega_\gamma}{n_\gamma}
\right)^2,
\end{aligned}
\]
we obtain
\begin{equation}
\label{eq:H-theorem-cutoff}
\begin{aligned}
\frac{d}{dt}\int \log n_\alpha\,d\alpha
&=
\frac13
\iiint
n_\alpha n_\beta n_\gamma
\left(
    \frac{\omega_\alpha}{n_\alpha}
    +
    \frac{\omega_\beta}{n_\beta}
    +
    \frac{\omega_\gamma}{n_\gamma}
\right)^2
\,d\mu_\varepsilon
\\
&\ge0.
\end{aligned}
\end{equation}
The nonnegative square structure relies on the symmetry of
\(\Gamma_{\alpha\beta\gamma}\) and on the symmetric placement of the cutoff
factor
\[
    \chi_\alpha\chi_\beta\chi_\gamma.
\]

\paragraph{Formal Rayleigh--Jeans equilibria.}
For \(A,B\in\mathbb R\), consider the algebraic ansatz
\begin{equation}
\label{eq:Rayleigh-Jeans-equilibrium}
    n_\alpha^{\mathrm{RJ}}
    =
    \frac{1}{A+B s_\alpha},
\end{equation}
wherever the denominator is positive.  Since
\[
    \frac{\omega_\alpha}{n_\alpha^{\mathrm{RJ}}}
    =
    A\omega_\alpha+B k_\alpha^1,
\]
the resonance relations imply
\[
\begin{aligned}
\frac{\omega_\alpha}{n_\alpha^{\mathrm{RJ}}}
+
\frac{\omega_\beta}{n_\beta^{\mathrm{RJ}}}
+
\frac{\omega_\gamma}{n_\gamma^{\mathrm{RJ}}}
&=
A(\omega_\alpha+\omega_\beta+\omega_\gamma)
\\
&\quad+
B(k_\alpha^1+k_\beta^1+k_\gamma^1)
\\
&=0.
\end{aligned}
\]
Thus the collision integrand vanishes formally.

For a bona fide nonnegative and nonsingular stationary profile on the full
unbounded two-branch spectral domain, however, one must take
\[
    B=0,
    \qquad
    A>0.
\]
Indeed,
\[
    s_{\sigma,k}=\sigma|k|
\]
takes arbitrarily large values of both signs, so \(A+B s_{\sigma,k}\) cannot
remain positive for every mode unless \(B=0\).  Consequently, the only
globally positive member of the Rayleigh--Jeans family is the
energy-equipartition profile
\[
    n_\alpha^{\mathrm{RJ}}=\frac1A.
\]
Profiles with \(B\neq0\) can represent equilibria only for a correspondingly
spectrally truncated model on which \(A+B s_\alpha>0\).

The constant \(B=0\) profile has infinite total wave energy and does not
belong to \(L^\infty_m\) for \(m>0\).  It therefore remains a formal
equilibrium relative to the solution class considered in this paper.

Thus the angular-cutoff WKE conserves total wave energy and horizontal
pseudo-momentum, satisfies the \(H\)-theorem
\eqref{eq:H-theorem-cutoff}, and admits the formal Rayleigh--Jeans family
\eqref{eq:Rayleigh-Jeans-equilibrium}, whose only globally positive member on
the full two-branch domain has \(B=0\).

We next record the geometric symmetries of the resonant set used in the
branch-by-branch parametrizations.

%


\subsection{Symmetry of the resonance manifold}
\label{Symmetry of the resonance manifold}

We first record the elementary symmetries of the resonance set.  The resonance
conditions are
\begin{equation}\label{manifold}
\begin{cases}
    k_\alpha+k_\beta+k_\gamma=0, \\[0.4ex]
    \displaystyle
    \sigma_\alpha\frac{k_\alpha^1}{|k_\alpha|}
    +\sigma_\beta\frac{k_\beta^1}{|k_\beta|}
    +\sigma_\gamma\frac{k_\gamma^1}{|k_\gamma|}
    =0 .
\end{cases}
\end{equation}
A priori there are \(2^3=8\) possible sign configurations
\[
    (\sigma_\alpha,\sigma_\beta,\sigma_\gamma)\in\{\pm1\}^3.
\]
However, the system is invariant under the simultaneous sign change
\[
    (\sigma_\alpha,\sigma_\beta,\sigma_\gamma)
    \longmapsto
    (-\sigma_\alpha,-\sigma_\beta,-\sigma_\gamma),
\]
since this only multiplies the frequency resonance equation by \(-1\).
Thus, for the geometric analysis, we may fix
\[
    \sigma_\alpha=+1.
\]
The system is also invariant under the exchange of the two integration
variables
\[
    \beta\leftrightarrow\gamma .
\]
Consequently, after fixing \(\sigma_\alpha=+1\), the four remaining choices
\[
    (+,+,+),\qquad (+,+,-),\qquad (+,-,+),\qquad (+,-,-)
\]
reduce, under the exchange \(\beta\leftrightarrow\gamma\), to the following
three non-equivalent representatives:
\[
    (+,+,-),\qquad (+,+,+),\qquad (+,-,-).
\]
The omitted case \((+,-,+)\) is equivalent to \((+,+,-)\) by exchanging
\(\beta\) and \(\gamma\).  The full collection of eight sign configurations is
therefore recovered from these three representatives by the global sign flip
and the \(\beta\leftrightarrow\gamma\) symmetry.

Throughout this section, the output mode
\[
    \alpha=(\sigma_\alpha,k_\alpha)
\]
is fixed.  Thus \(k_\alpha\) is regarded as a parameter, while the integration
variables are \(k_\beta\) and \(k_\gamma\).  Using momentum conservation, we
eliminate
\[
    k_\gamma=-k_\alpha-k_\beta .
\]
Consequently, for each fixed \(k_\alpha\) and each representative sign
configuration, the resonance set can be viewed as a subset of the
\(k_\beta\)-plane.  When the remaining frequency constraint is regular, this
subset is a one-dimensional resonance curve.

The three representative sign configurations give the following systems.

\paragraph{Manifold 1: representative \((+,+,-)\).}
\begin{equation}\label{manifold1}
\begin{cases}
k_\alpha^1+k_\beta^1+k_\gamma^1=0,\\
k_\alpha^2+k_\beta^2+k_\gamma^2=0,\\[0.4ex]
\displaystyle
\frac{k_\alpha^1}{|k_\alpha|}
+\frac{k_\beta^1}{|k_\beta|}
-\frac{k_\gamma^1}{|k_\gamma|}
=0.
\end{cases}
\end{equation}

\paragraph{Manifold 2: representative \((+,+,+)\).}
\begin{equation}\label{manifold2}
\begin{cases}
k_\alpha^1+k_\beta^1+k_\gamma^1=0,\\
k_\alpha^2+k_\beta^2+k_\gamma^2=0,\\[0.4ex]
\displaystyle
\frac{k_\alpha^1}{|k_\alpha|}
+\frac{k_\beta^1}{|k_\beta|}
+\frac{k_\gamma^1}{|k_\gamma|}
=0.
\end{cases}
\end{equation}

\paragraph{Manifold 3: representative \((+,-,-)\).}
\begin{equation}\label{manifold3}
\begin{cases}
k_\alpha^1+k_\beta^1+k_\gamma^1=0,\\
k_\alpha^2+k_\beta^2+k_\gamma^2=0,\\[0.4ex]
\displaystyle
\frac{k_\alpha^1}{|k_\alpha|}
-\frac{k_\beta^1}{|k_\beta|}
-\frac{k_\gamma^1}{|k_\gamma|}
=0.
\end{cases}
\end{equation}

The resonance conditions are homogeneous under simultaneous dilation: if
\((k_\alpha,k_\beta,k_\gamma)\) satisfies \eqref{manifold}, then so does
\[
    (\lambda k_\alpha,\lambda k_\beta,\lambda k_\gamma),
    \qquad \lambda>0.
\]
Thus the radial scale of the fixed vector \(k_\alpha\) may be normalized.  In
the geometric analysis we set
\[
    |k_\alpha|=1,
    \qquad
    k_\alpha=(\cos\theta_\alpha,\sin\theta_\alpha).
\]
The resonance set for general \(|k_\alpha|>0\) is recovered by the inverse
dilation.

For each representative sign configuration we define the corresponding
unit-scale resonance set
\[
    \mathcal R_j(\theta_\alpha)
    :=
    \left\{
        k_\beta\in\mathbb R^2:
        (k_\alpha,k_\beta,k_\gamma)
        \text{ satisfies the corresponding system above, with }
        k_\gamma=-k_\alpha-k_\beta
    \right\},
    \qquad j=1,2,3.
\]
Here \(\mathcal R_1\), \(\mathcal R_2\), and \(\mathcal R_3\) correspond
respectively to
\[
    (+,+,-),\qquad (+,+,+),\qquad (+,-,-).
\]

The systems are also compatible with mirror reflections of the spatial
variables.  Reflecting the \(k\)-variables maps the resonance set associated
with one direction of \(k_\alpha\) to the resonance set associated with the
reflected direction of \(k_\alpha\).  Therefore, for the purpose of estimating
the geometry of the resonance curves, it suffices to restrict the direction of
the fixed vector \(k_\alpha\) to the first quadrant:
\begin{equation}\label{eq:theta-first-quadrant}
    0\le \theta_\alpha\le \frac{\pi}{2}.
\end{equation}
The other angular sectors are obtained from this one by the corresponding
mirror symmetries.

Under this normalization, the first two representative sign configurations,
\[
    (+,+,-)
    \qquad\text{and}\qquad
    (+,+,+),
\]
give one-dimensional resonance curves, except at lower-dimensional degenerate
sets.  The third configuration,
\[
    (+,-,-),
\]
requires a further distinction.  For
\[
    \frac{\pi}{3}<\theta_\alpha\le \frac{\pi}{2},
\]
the set \(\mathcal R_3(\theta_\alpha)\) is a one-dimensional resonance curve.
For
\[
    0\le\theta_\alpha\le \frac{\pi}{3},
\]
the resonance set is degenerate; as shown by the explicit parametrization
below, it is either empty or consists of at most two points.  Therefore, in the
analysis of the third sign configuration, we restrict to
\[
    \frac{\pi}{3}<\theta_\alpha\le \frac{\pi}{2}.
\]
Having identified these symmetries, we now turn to the scaling and co--area
representation of the collision integral.


\subsection{Scaling and co--area representation}
\label{Scaling and co–area representation}

In order to give a precise meaning to the product of the two delta constraints,
we first appeal to the co--area formula.  The role of the co--area formula is
to interpret the formal factor
\[
    \delta(\Xi_1)\delta(\Xi_2)
\]
as integration over the regular part of the resonant manifold, where the
constraints are transversal.  The scaling factor \(|k_\alpha|^4\) used later
in the estimates is obtained separately, after using the momentum constraint to
eliminate \(k_\gamma\).

Fix the output mode \(\alpha=(\sigma_\alpha,k_\alpha)\).  Recall that
\[
    \Xi_2(\alpha,\beta,\gamma)
    :=
    k_\alpha+k_\beta+k_\gamma
    \in\mathbb R^2,
\]
and
\[
    \Xi_1(\alpha,\beta,\gamma)
    :=
    \sigma_\alpha\frac{k_\alpha^1}{|k_\alpha|}
    +\sigma_\beta\frac{k_\beta^1}{|k_\beta|}
    +\sigma_\gamma\frac{k_\gamma^1}{|k_\gamma|}
    \in\mathbb R .
\]
For fixed \(\alpha\), the map
\[
    (k_\beta,k_\gamma)
    \longmapsto
    (\Xi_1,\Xi_2)
\]
maps \(\mathbb R^2\times\mathbb R^2\) into \(\mathbb R^3\).  On the regular
set where this map has rank \(3\), the common zero set
\[
    \mathcal R_\alpha
    :=
    \{(k_\beta,k_\gamma):\Xi_1=0,\ \Xi_2=0\}
\]
is a one-dimensional resonant manifold.

Assume temporarily that \(n\) is smooth and compactly supported on a regular
portion of the resonant set.  By the co--area formula applied to
\[
    (k_\beta,k_\gamma)\mapsto(\Xi_1,\Xi_2),
\]
the cut-off collision operator can be represented as
\begin{align}
\mathcal C_\varepsilon(n)(\alpha)
   &=
   \sum_{\sigma_\beta=\pm1}
   \sum_{\sigma_\gamma=\pm1}
   \int_{\mathcal R_\alpha}
      \chi_\alpha\chi_\beta\chi_\gamma\,
      \Gamma_{\alpha\beta\gamma}^{2}\,\omega_\alpha
      \bigl(
         \omega_\alpha n_\beta n_\gamma
         + \omega_\gamma n_\alpha n_\beta
         + \omega_\beta n_\alpha n_\gamma
      \bigr)
      \notag\\
   &\qquad\qquad\qquad\qquad\times
      \frac{d\mathcal H^1(k_\beta,k_\gamma)}
      {J(\alpha,\beta,\gamma)} ,
\label{eq:coarea-representation}
\end{align}
where
\[
    J(\alpha,\beta,\gamma)
    :=
    \sqrt{
    \det\!\left(
       \nabla_{(k_\beta,k_\gamma)}(\Xi_1,\Xi_2)
       \nabla_{(k_\beta,k_\gamma)}(\Xi_1,\Xi_2)^{\top}
    \right)
    } .
\]
Here \(d\mathcal H^1\) denotes the one-dimensional Hausdorff measure induced on
the resonant manifold.

A direct computation gives
\begin{align}
J(\alpha,\beta,\gamma)
  =
  \sqrt{\;
      2\left(
        \frac{\sigma_\beta (k_\beta^2)^2}{|k_\beta|^3}
        -\frac{\sigma_\gamma (k_\gamma^2)^2}{|k_\gamma|^3}
      \right)^{2}
      +
      2\left(
        \frac{\sigma_\beta k_\beta^1 k_\beta^2}{|k_\beta|^3}
        -\frac{\sigma_\gamma k_\gamma^1 k_\gamma^2}{|k_\gamma|^3}
      \right)^{2}
     } .
\label{eq:coarea-Jacobian}
\end{align}
This formula is understood on the regular part where
\[
    k_\beta\neq0,
    \qquad
    k_\gamma\neq0.
\]
At such points, one checks from \eqref{eq:coarea-Jacobian} that \(J=0\) on the
resonant set can occur only when
\[
    k_\beta\parallel k_\gamma
\]
and, using the resonance constraints, this forces
\[
    k_\alpha^1=0.
\]
Thus, away from the zero-frequency output direction
\[
    \cos\theta_\alpha=0,
\]
the full co--area Jacobian is nonzero on the regular resonant branches
considered below.  This is one reason for imposing the angular cut-off near
\(\cos\theta=0\).

The collapsed endpoints
\[
    k_\beta=0
    \qquad\text{or}\qquad
    k_\gamma=0
\]
are not regular points of the full co--area chart, since the angle of the
collapsed vector is not defined there and the formula
\eqref{eq:coarea-Jacobian} is not meant to be evaluated at such points.  These
endpoints are treated instead through the explicit unit-scale parametrizations
below. In the explicit parametrizations below, the collapsed endpoints are treated as
limiting endpoints of the curves parametrized by \(\theta_\beta\).  They are
either excluded by the zero-frequency localization or appear as harmless corner
endpoints in the estimates.

We now derive the scaled representation used in the estimates.  This scaling
calculation does not come from the full co--area Jacobian
\eqref{eq:coarea-Jacobian}.  Instead, we first use the momentum constraint to
eliminate
\[
    k_\gamma=-k_\alpha-k_\beta .
\]
The remaining resonance condition is the scalar equation
\[
    \widetilde\Xi_1(k_\beta^1,k_\beta^2)=0,
\]
where
\[
    \widetilde\Xi_1(k_\beta)
    :=
    \sigma_\alpha\frac{k_\alpha^1}{|k_\alpha|}
    +\sigma_\beta\frac{k_\beta^1}{|k_\beta|}
    +\sigma_\gamma\frac{k_\gamma^1}{|k_\gamma|},
    \qquad
    k_\gamma=-k_\alpha-k_\beta .
\]
Thus the collision operator can be written locally as
\[
\mathcal C_\varepsilon(n)(\alpha)
=
\sum_{\sigma_\beta=\pm1}
\sum_{\sigma_\gamma=\pm1}
\int_{\mathbb R^2}
    I_\alpha(k_\beta)\,
    \delta\!\bigl(\widetilde\Xi_1(k_\beta)\bigr)
    \,dk_\beta,
\]
where
\[
\begin{aligned}
I_\alpha(k_\beta)
&=
\chi_\alpha\chi_\beta\chi_\gamma\,
\Gamma_{\alpha\beta\gamma}^{2}\,\omega_\alpha
\Bigl(
    \omega_\alpha n_\beta n_\gamma
    +\omega_\gamma n_\alpha n_\beta
    +\omega_\beta n_\alpha n_\gamma
\Bigr),
\\
&\hspace{7cm}
k_\gamma=-k_\alpha-k_\beta .
\end{aligned}
\]

On a regular patch where
\[
    \partial_{k_\beta^2}\widetilde\Xi_1\neq0,
\]
we parametrize the resonance curve
\[
    \widetilde\Xi_1(k_\beta^1,k_\beta^2)=0
\]
by
\[
    k_\beta=k_\beta(\theta_\beta)
    =
    \bigl(k_\beta^1(\theta_\beta),k_\beta^2(\theta_\beta)\bigr).
\]
Using the one-dimensional delta identity, we obtain
\begin{equation}\label{eq:delta-param-k}
\int_{\mathbb R^2}
    F(k_\beta)\,
    \delta\!\bigl(\widetilde\Xi_1(k_\beta)\bigr)
    \,dk_\beta
=
\int
    F(k_\beta(\theta_\beta))
    \frac{
        \left|\dfrac{d k_\beta^1}{d\theta_\beta}\right|
    }{
        \left|
        \partial_{k_\beta^2}
        \widetilde\Xi_1
        \bigl(k_\beta^1(\theta_\beta),k_\beta^2(\theta_\beta)\bigr)
        \right|
    }
    \,d\theta_\beta .
\end{equation}

We now introduce the scaling
\[
    k_\alpha=|k_\alpha|\varsigma_\alpha,
    \qquad
    k_\beta=|k_\alpha|\varsigma_\beta,
    \qquad
    k_\gamma=|k_\alpha|\varsigma_\gamma,
    \qquad
    |\varsigma_\alpha|=1,
\]
so that
\[
    \varsigma_\gamma=-\varsigma_\alpha-\varsigma_\beta .
\]
The variables \(\varsigma_\beta\) and \(\varsigma_\gamma\) are scaled wave
vectors and are not assumed to have unit length.  Since
\[
    |k_\beta|=|k_\alpha|\,|\varsigma_\beta|,
    \qquad
    |k_\gamma|=|k_\alpha|\,|\varsigma_\gamma|,
\]
the frequency constraint is homogeneous of degree zero:
\[
    \widetilde\Xi_1(k_\beta)
    =
    \widetilde\Xi_1(\varsigma_\beta),
\]
where
\[
    \widetilde\Xi_1(\varsigma_\beta)
    :=
    \sigma_\alpha\frac{\varsigma_\alpha^1}{|\varsigma_\alpha|}
    +\sigma_\beta\frac{\varsigma_\beta^1}{|\varsigma_\beta|}
    +\sigma_\gamma\frac{\varsigma_\gamma^1}{|\varsigma_\gamma|},
    \qquad
    \varsigma_\gamma=-\varsigma_\alpha-\varsigma_\beta .
\]

Although \(\widetilde\Xi_1\) is homogeneous of degree zero, its derivative with
respect to the unscaled variable has degree \(-1\):
\[
    \partial_{k_\beta^2}\widetilde\Xi_1(k_\beta)
    =
    |k_\alpha|^{-1}
    \partial_{\varsigma_\beta^2}
    \widetilde\Xi_1(\varsigma_\beta).
\]
Moreover,
\[
    \left|\frac{d k_\beta^1}{d\theta_\beta}\right|
    =
    |k_\alpha|
    \left|\frac{d \varsigma_\beta^1}{d\theta_\beta}\right|.
\]
Therefore the Jacobian ratio in \eqref{eq:delta-param-k} scales as
\begin{equation}\label{eq:jacobian-ratio-scaling}
\frac{
    \left|\dfrac{d k_\beta^1}{d\theta_\beta}\right|
}{
    \left|
    \partial_{k_\beta^2}\widetilde\Xi_1(k_\beta)
    \right|
}
=
|k_\alpha|^2
\frac{
    \left|\dfrac{d \varsigma_\beta^1}{d\theta_\beta}\right|
}{
    \left|
    \partial_{\varsigma_\beta^2}
    \widetilde\Xi_1(\varsigma_\beta)
    \right|
}.
\end{equation}

The interaction coefficient is homogeneous of degree one:
\[
    \Gamma_{\alpha\beta\gamma}
    =
    |k_\alpha|\,
    \Gamma(\varsigma_\alpha,\varsigma_\beta,\varsigma_\gamma),
\]
and hence
\[
    \Gamma_{\alpha\beta\gamma}^{2}
    =
    |k_\alpha|^2\,
    \Gamma^2(\varsigma_\alpha,\varsigma_\beta,\varsigma_\gamma).
\]
The frequencies are homogeneous of degree zero.  Combining this with
\eqref{eq:jacobian-ratio-scaling}, the parametrized collision operator becomes
\begin{equation}\label{eq:parametrized-collision}
\mathcal C_\varepsilon(n)(\alpha)
=
\sum_{\sigma_\beta=\pm1}
\sum_{\sigma_\gamma=\pm1}
|k_\alpha|^4
\int
    \widetilde I_\alpha(\theta_\beta)
    \frac{
        \left|\dfrac{d\varsigma_\beta^1}{d\theta_\beta}\right|
    }{
        \left|
        \partial_{\varsigma_\beta^2}
        \widetilde\Xi_1
        \bigl(
            \varsigma_\beta^1(\theta_\beta),
            \varsigma_\beta^2(\theta_\beta)
        \bigr)
        \right|
    }
    \,d\theta_\beta .
\end{equation}
Here
\[
    \varsigma_\gamma=-\varsigma_\alpha-\varsigma_\beta,
\]
and
\begin{align*}
\widetilde I_\alpha(\theta_\beta)
&=
\chi_\alpha\chi_\beta\chi_\gamma\,
\Gamma^2(\varsigma_\alpha,\varsigma_\beta,\varsigma_\gamma)
\,\omega(\sigma_\alpha,\varsigma_\alpha)
\\
&\quad\times
\Bigl[
    \omega(\sigma_\alpha,\varsigma_\alpha)
        n(\sigma_\beta,|k_\alpha|\varsigma_\beta)
        n(\sigma_\gamma,|k_\alpha|\varsigma_\gamma)
\\
&\qquad\quad
    +\omega(\sigma_\gamma,\varsigma_\gamma)
        n(\sigma_\alpha,|k_\alpha|\varsigma_\alpha)
        n(\sigma_\beta,|k_\alpha|\varsigma_\beta)
\\
&\qquad\quad
    +\omega(\sigma_\beta,\varsigma_\beta)
        n(\sigma_\alpha,|k_\alpha|\varsigma_\alpha)
        n(\sigma_\gamma,|k_\alpha|\varsigma_\gamma)
\Bigr].
\end{align*}
Thus the prefactor \(|k_\alpha|^4\) comes from two independent sources:
\(|k_\alpha|^2\) from the homogeneity of
\(\Gamma_{\alpha\beta\gamma}^{2}\), and \(|k_\alpha|^2\) from the scaling of
the one-dimensional delta Jacobian ratio in
\eqref{eq:jacobian-ratio-scaling}.

For later reference, since
\[
    \varsigma_\gamma=-\varsigma_\alpha-\varsigma_\beta,
\]
we have
\[
\nabla_{\varsigma_\beta}\widetilde\Xi_1
=
\sigma_\beta
\nabla\!\left(\frac{\varsigma_\beta^1}{|\varsigma_\beta|}\right)
-
\sigma_\gamma
\nabla\!\left(\frac{\varsigma_\gamma^1}{|\varsigma_\gamma|}\right).
\]
Consequently,
\[
\partial_{\varsigma_\beta^1}\widetilde\Xi_1
=
\sigma_\beta
\frac{(\varsigma_\beta^2)^2}{|\varsigma_\beta|^3}
-
\sigma_\gamma
\frac{(\varsigma_\gamma^2)^2}{|\varsigma_\gamma|^3},
\]
and
\[
\partial_{\varsigma_\beta^2}\widetilde\Xi_1
=
-\sigma_\beta
\frac{\varsigma_\beta^1\varsigma_\beta^2}{|\varsigma_\beta|^3}
+
\sigma_\gamma
\frac{\varsigma_\gamma^1\varsigma_\gamma^2}{|\varsigma_\gamma|^3}.
\]
The angular cut-offs restrict the analysis to regions where the co--area
representation and the parametrizations above are non-degenerate.  The
additional singularity near \(|\cos\theta|=\tfrac12\), which produces an
unbounded collision operator in the absence of angular localization, is treated
separately in Section~\ref{Section: Unboundedness}.

\section{Parametrization and Estimates for the $(+,+,+)$ Resonance Manifold}\label{Parametrization and Estimates for the $(+,+,+)$ Resonance Manifold}
\subsection{Unit--scale parametrization}
\label{Parametrization(+,+,+)}

In this section we study the unit--scale resonance curve for the
\((+,+,+)\) sign configuration.  By the homogeneity of the resonance
conditions, we fix
\[
    |k_\alpha|=1,
    \qquad
    k_\alpha=(\cos\theta_\alpha,\sin\theta_\alpha).
\]
By the mirror symmetries recorded in
Section~\ref{Symmetry of the resonance manifold}, it is enough to consider
\[
    0\le \theta_\alpha\le \frac{\pi}{2}.
\]
In the parametrization below we take
\[
    0\le \theta_\alpha<\frac{\pi}{2}.
\]
The endpoint \(\theta_\alpha=\frac{\pi}{2}\) belongs to the zero-frequency
degeneracy and is excluded by the angular cut-off.  When
\(\theta_\alpha=0\), some of the arcs below may degenerate; the formulas are
then understood in the limiting sense.

For fixed \(k_\alpha\), the \((+,+,+)\) resonance set is viewed as a curve in
the \(k_\beta\)-plane, with
\[
    k_\gamma=-k_\alpha-k_\beta .
\]

\begin{remark}[Central symmetry]
For fixed \(k_\alpha\), the map
\[
    k_\beta\longmapsto -k_\alpha-k_\beta
\]
is the half-turn about the point \(-k_\alpha/2\).  This map exchanges
\(k_\beta\) and \(k_\gamma\).  Since the \((+,+,+)\) resonance conditions are
symmetric under the exchange \(\beta\leftrightarrow\gamma\), the resonance
curve in the \(k_\beta\)-plane is centrally symmetric with respect to
\[
    -\frac{k_\alpha}{2}
    =
    \left(
        -\frac{\cos\theta_\alpha}{2},
        -\frac{\sin\theta_\alpha}{2}
    \right).
\]
This central symmetry is separate from the first-quadrant reduction for
\(\theta_\alpha\), which comes from the mirror symmetries of the full resonance
system.
\end{remark}
\begin{proposition}[Unit--scale parametrization of the \((+,+,+)\) resonance curve]
\label{prop:param-+++}
Fix
\[
    0\le \theta_\alpha<\frac{\pi}{2},
    \qquad
    k_\alpha=(\cos\theta_\alpha,\sin\theta_\alpha).
\]
The closures of the unit-scale \((+,+,+)\) resonance arcs are parametrized by
\[
    k_\beta
    =
    |k_\beta|(\cos\theta_\beta,\sin\theta_\beta),
\]
where
\[
|k_\beta|
=
\begin{cases}
\displaystyle
\frac{
    \sin\theta_\alpha(\cos\theta_\alpha+\cos\theta_\beta)
    -
    \cos\theta_\alpha
    \sqrt{1-(\cos\theta_\alpha+\cos\theta_\beta)^2}
}{
    -\sin\theta_\beta(\cos\theta_\alpha+\cos\theta_\beta)
    +
    \cos\theta_\beta
    \sqrt{1-(\cos\theta_\alpha+\cos\theta_\beta)^2}
},
&
\begin{array}{l}
\frac{\pi}{2}
\le
\theta_\beta
\le
\pi+\theta_\alpha,\\
\sin\theta_\gamma\le0,
\end{array}
\\[2.0em]
\displaystyle
\frac{
    \sin\theta_\alpha(\cos\theta_\alpha+\cos\theta_\beta)
    +
    \cos\theta_\alpha
    \sqrt{1-(\cos\theta_\alpha+\cos\theta_\beta)^2}
}{
    -\sin\theta_\beta(\cos\theta_\alpha+\cos\theta_\beta)
    -
    \cos\theta_\beta
    \sqrt{1-(\cos\theta_\alpha+\cos\theta_\beta)^2}
},
&
\begin{array}{l}
\pi+\theta_\alpha
\le
\theta_\beta
\le
2\pi-\arccos(1-\cos\theta_\alpha),\\
\sin\theta_\gamma\ge0,
\end{array}
\\[2.0em]
\displaystyle
\frac{
    \sin\theta_\alpha(\cos\theta_\alpha+\cos\theta_\beta)
    -
    \cos\theta_\alpha
    \sqrt{1-(\cos\theta_\alpha+\cos\theta_\beta)^2}
}{
    -\sin\theta_\beta(\cos\theta_\alpha+\cos\theta_\beta)
    +
    \cos\theta_\beta
    \sqrt{1-(\cos\theta_\alpha+\cos\theta_\beta)^2}
},
&
\begin{array}{l}
\frac{3\pi}{2}
\le
\theta_\beta
\le
2\pi-\arccos(1-\cos\theta_\alpha),\\
\sin\theta_\gamma\le0.
\end{array}
\end{cases}
\]
Here
\[
    \cos\theta_\gamma
    =
    -\cos\theta_\alpha-\cos\theta_\beta,
\]
and the sign of \(\sin\theta_\gamma\) distinguishes the two branches.

The endpoints are included only in the limiting sense.  At
\[
    \theta_\beta=\frac{\pi}{2}
    \qquad\text{and}\qquad
    \theta_\beta=\frac{3\pi}{2},
\]
one has
\[
    |k_\beta|=0,
\]
so these are collapsed endpoints with \(k_\beta=0\).  At
\[
    \theta_\beta=\pi+\theta_\alpha,
\]
one has
\[
    |k_\beta|=1,
    \qquad
    k_\beta=-k_\alpha,
    \qquad
    k_\gamma=0,
\]
so this is a collapsed endpoint with \(k_\gamma=0\).  Finally, at
\[
    \theta_\beta
    =
    2\pi-\arccos(1-\cos\theta_\alpha),
\]
the two choices \(\sin\theta_\gamma\le0\) and
\(\sin\theta_\gamma\ge0\) meet, since
\[
    \sqrt{1-(\cos\theta_\alpha+\cos\theta_\beta)^2}=0.
\]
\end{proposition}

\begin{proof}
We work in polar coordinates
\[
    k_j=|k_j|(\cos\theta_j,\sin\theta_j),
    \qquad
    j\in\{\alpha,\beta,\gamma\}.
\]

The momentum resonance
\[
    k_\alpha+k_\beta+k_\gamma=0
\]
is equivalent to
\[
    \cos\theta_\alpha
    +|k_\beta|\cos\theta_\beta
    +|k_\gamma|\cos\theta_\gamma=0,
\]
and
\[
    \sin\theta_\alpha
    +|k_\beta|\sin\theta_\beta
    +|k_\gamma|\sin\theta_\gamma=0.
\]
Solving this \(2\times2\) linear system gives
\[
    |k_\beta|
    =
    \frac{\sin(\theta_\gamma-\theta_\alpha)}
         {\sin(\theta_\beta-\theta_\gamma)},
    \qquad
    |k_\gamma|
    =
    \frac{\sin(\theta_\beta-\theta_\alpha)}
         {\sin(\theta_\gamma-\theta_\beta)}.
\]

For the \((+,+,+)\) sign configuration, the frequency resonance is
\[
    \cos\theta_\alpha+\cos\theta_\beta+\cos\theta_\gamma=0.
\]
Hence
\[
    \cos\theta_\gamma
    =
    -\cos\theta_\alpha-\cos\theta_\beta.
\]
The two possible choices of \(\theta_\gamma\) are therefore determined by
\[
    \sin\theta_\gamma
    =
    \pm
    \sqrt{1-(\cos\theta_\alpha+\cos\theta_\beta)^2}.
\]

If
\[
    \sin\theta_\gamma
    =
    -
    \sqrt{1-(\cos\theta_\alpha+\cos\theta_\beta)^2},
\]
then substituting this expression into the formula for \(|k_\beta|\) gives
\[
|k_\beta|
=
\frac{
    \sin\theta_\alpha(\cos\theta_\alpha+\cos\theta_\beta)
    -
    \cos\theta_\alpha
    \sqrt{1-(\cos\theta_\alpha+\cos\theta_\beta)^2}
}{
    -\sin\theta_\beta(\cos\theta_\alpha+\cos\theta_\beta)
    +
    \cos\theta_\beta
    \sqrt{1-(\cos\theta_\alpha+\cos\theta_\beta)^2}
}.
\]
The positivity conditions
\[
    |k_\beta|\ge0,
    \qquad
    |k_\gamma|\ge0
\]
select the two closed intervals
\[
    \frac{\pi}{2}
    \le
    \theta_\beta
    \le
    \pi+\theta_\alpha
\]
and
\[
    \frac{3\pi}{2}
    \le
    \theta_\beta
    \le
    2\pi-\arccos(1-\cos\theta_\alpha),
\]
with the endpoints interpreted in the limiting sense.

If
\[
    \sin\theta_\gamma
    =
    +
    \sqrt{1-(\cos\theta_\alpha+\cos\theta_\beta)^2},
\]
then substituting this expression into the formula for \(|k_\beta|\) gives
\[
|k_\beta|
=
\frac{
    \sin\theta_\alpha(\cos\theta_\alpha+\cos\theta_\beta)
    +
    \cos\theta_\alpha
    \sqrt{1-(\cos\theta_\alpha+\cos\theta_\beta)^2}
}{
    -\sin\theta_\beta(\cos\theta_\alpha+\cos\theta_\beta)
    -
    \cos\theta_\beta
    \sqrt{1-(\cos\theta_\alpha+\cos\theta_\beta)^2}
}.
\]
The positivity conditions select the closed interval
\[
    \pi+\theta_\alpha
    \le
    \theta_\beta
    \le
    2\pi-\arccos(1-\cos\theta_\alpha),
\]
again with endpoints interpreted in the limiting sense.

Combining the two choices of \(\sin\theta_\gamma\) gives the stated
parametrization.
\end{proof}

\begin{remark}[Orientation]
Proposition~\ref{prop:param-+++} describes the resonance curve as a union of
closed parametrized arcs.  In the co-area estimates, we use the corresponding
open intervals, since the endpoints are not regular co-area points.  If one
wants to traverse the compact curve with an orientation, the last arc is
traversed in the reverse direction:
\[
    \theta_\beta:
    2\pi-\arccos(1-\cos\theta_\alpha)
    \longrightarrow
    \frac{3\pi}{2}.
\]
This is only an orientation convention and does not affect the estimates,
which use unoriented integrals.
\end{remark}

\begin{figure}[htbp]
    \centering
    \includegraphics[width=\textwidth]{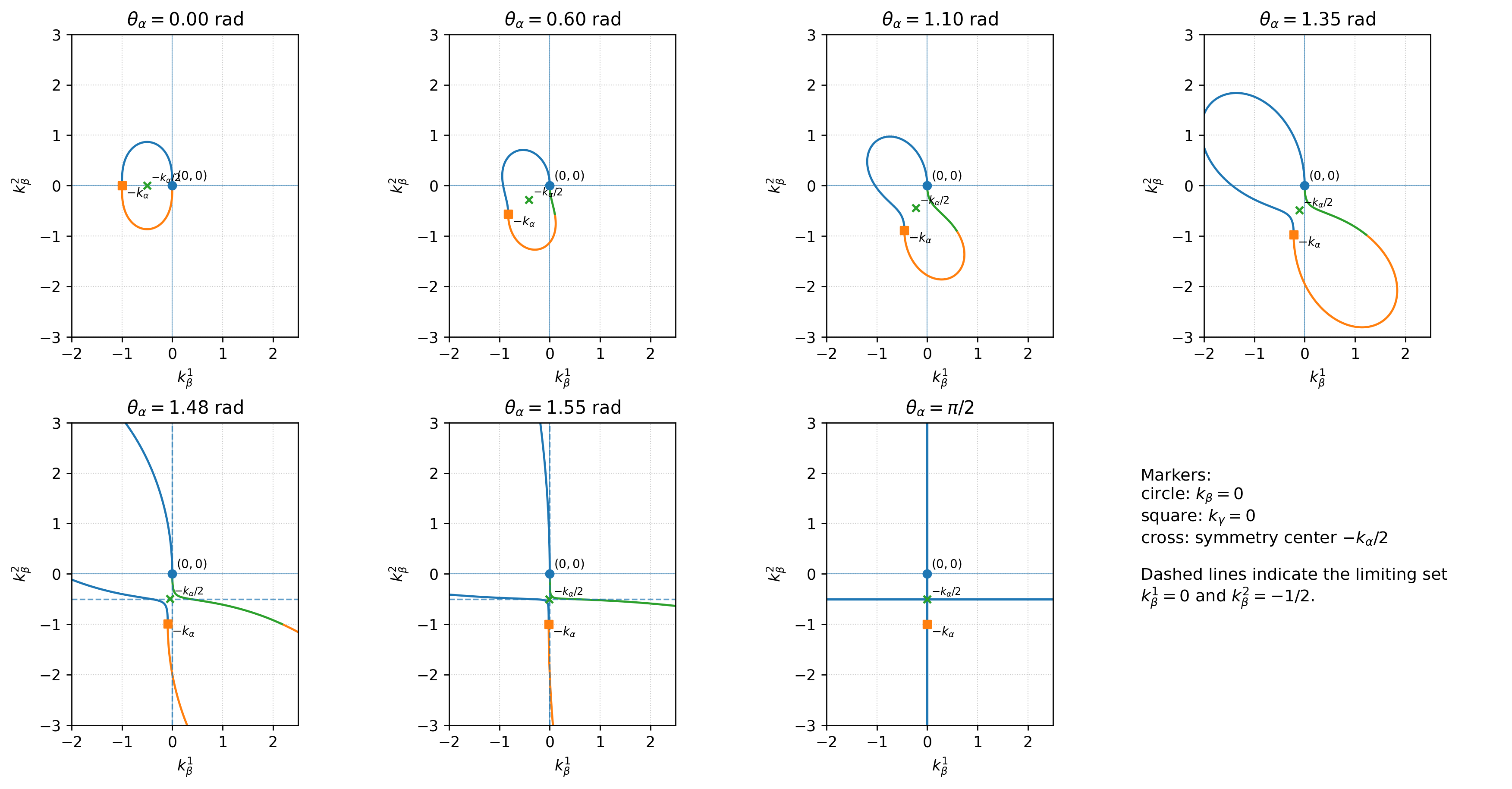}
\caption{
    Unit-scale \((+,+,+)\) resonance curve from
    Proposition~\ref{prop:param-+++}, plotted in the \(k_\beta\)-plane for
    increasing values of \(\theta_\alpha\).  We fix
    \(k_\alpha=(\cos\theta_\alpha,\sin\theta_\alpha)\) and
    \(k_\gamma=-k_\alpha-k_\beta\).  The marked points are
    \(0\), \(-k_\alpha\), and \(-k_\alpha/2\), corresponding respectively to
    \(k_\beta=0\), \(k_\gamma=0\), and the center of the central symmetry.
    The dashed lines indicate the limiting set
    \(k_\beta^1=0\) and \(k_\beta^2=-1/2\) at
    \(\theta_\alpha=\pi/2\).
}
    \label{fig:manifold-+++}
\end{figure}

\begin{remark}[Interpretation of Figure~\ref{fig:manifold-+++}]
Figure~\ref{fig:manifold-+++} illustrates the full, uncut resonance geometry.
As \(\theta_\alpha\to\pi/2\), the compact curve described in
Proposition~\ref{prop:param-+++} degenerates toward the non-compact limiting
set
\[
    k_\beta^1=0,
    \qquad
    k_\beta^2=-\frac12 .
\]
Indeed, at \(\theta_\alpha=\pi/2\), writing \(k_\beta=(x,y)\) gives
\[
    k_\gamma=(-x,-1-y),
\]
and the frequency constraint reduces to
\[
    \frac{x}{\sqrt{x^2+y^2}}
    -
    \frac{x}{\sqrt{x^2+(1+y)^2}}
    =0.
\]
Away from the collapsed endpoints \(k_\beta=0\) and \(k_\gamma=0\), this is
equivalent to
\[
    x=0
    \qquad\text{or}\qquad
    y=-\frac12 .
\]
The long branches in the last panels therefore represent the degeneration of
the uncut geometry, not the uniformly controlled region used in the estimates.

After applying the angular cut-off
\[
    \chi_\varepsilon(\theta_\alpha)
    \chi_\varepsilon(\theta_\beta)
    \chi_\varepsilon(\theta_\gamma),
\]
the admissible portion lies in a compact regular subset; the quantitative
lower bounds excluding the zero-frequency limiting directions are proved
below.
\end{remark}

\subsection{Co--area factor estimates}
\label{subsec:coarea-factor-+++}

We now estimate the one-dimensional co--area factor appearing in the
parametrized representation of the \((+,+,+)\) resonance curve.  Throughout
this subsection we work at unit scale:
\[
    |k_\alpha|=1,
    \qquad
    k_\alpha=(\cos\theta_\alpha,\sin\theta_\alpha),
    \qquad
    0\le \theta_\alpha<\frac{\pi}{2}.
\]

We split the unit-scale resonance curve into two halves using the central
symmetry center \(-k_\alpha/2\).  Define the upper half by
\[
    \mathcal R_{+++}^{\mathrm{up}}(\theta_\alpha)
    :=
    \left\{
        k_\beta\in\mathcal R_{+++}(\theta_\alpha):
        0\le
        \arg\left(k_\beta+\frac{k_\alpha}{2}\right)
        \le \pi
    \right\}.
\]
In the angular parametrization of Proposition~\ref{prop:param-+++}, this
corresponds to
\[
    \theta_\beta
    \in
    \left[\frac{\pi}{2},\,\pi+\theta_\alpha\right].
\]
The lower half is obtained from the upper half by the central symmetry
\[
    k_\beta\mapsto -k_\alpha-k_\beta,
\]
which exchanges \(k_\beta\) and \(k_\gamma\).

On the upper half, we have \(\sin\theta_\gamma\le0\).  We define
\[
    p(\theta_\alpha,\theta_\beta)
    :=
    \cos\theta_\alpha+\cos\theta_\beta,
\]
and
\[
    r(\theta_\alpha,\theta_\beta)
    :=
    \sqrt{1-p(\theta_\alpha,\theta_\beta)^2}.
\]
Thus
\[
    \cos\theta_\gamma=-p(\theta_\alpha,\theta_\beta),
    \qquad
    \sin\theta_\gamma=-r(\theta_\alpha,\theta_\beta).
\]
We also define
\[
    f(\theta_\alpha,\theta_\beta)
    :=
    -\sin\theta_\beta\,p(\theta_\alpha,\theta_\beta)
    +
    \cos\theta_\beta\,r(\theta_\alpha,\theta_\beta).
\]
Then the upper-half parametrization is
\[
    |k_\beta|
    =
    \frac{
        \sin\theta_\alpha\,p(\theta_\alpha,\theta_\beta)
        -
        \cos\theta_\alpha\,r(\theta_\alpha,\theta_\beta)
    }{
        f(\theta_\alpha,\theta_\beta)
    }.
\]

\begin{lemma}[Lower bound for the denominator]
\label{lem:f-lower-+++}
For every \(0\le\theta_\alpha<\pi/2\) and every
\[
    \theta_\beta
    \in
    \left[\frac{\pi}{2},\,\pi+\theta_\alpha\right],
\]
one has
\[
    |f(\theta_\alpha,\theta_\beta)|
    \ge
    \cos\theta_\alpha
    \sqrt{1-\frac{\cos^2\theta_\alpha}{4}}.
\]
Consequently,
\[
    |f(\theta_\alpha,\theta_\beta)|^{-1}
    \lesssim
    (\cos\theta_\alpha)^{-1},
\]
with an absolute implicit constant.
\end{lemma}

\begin{proof}
For readability, write
\[
    p=p(\theta_\alpha,\theta_\beta),
    \qquad
    r=r(\theta_\alpha,\theta_\beta),
    \qquad
    f=f(\theta_\alpha,\theta_\beta)
\]
inside this proof.  A direct differentiation gives
\[
    \partial_{\theta_\beta} f
    =
    -\frac{
        (r-\sin\theta_\beta)
        \bigl(
            \cos\theta_\alpha\cos\theta_\beta
            +
            \cos^2\theta_\beta
            +
            \sin\theta_\beta\,r
        \bigr)
    }{r}.
\]
On the interval
\[
    \theta_\beta\in
    \left[\frac{\pi}{2},\,\pi+\theta_\alpha\right],
\]
the relevant interior critical point is
\[
    \theta_\beta
    =
    \arccos\left(-\frac{\cos\theta_\alpha}{2}\right).
\]
At the endpoints,
\[
    f\left(\theta_\alpha,\frac{\pi}{2}\right)
    =
    f(\theta_\alpha,\pi+\theta_\alpha)
    =
    -\cos\theta_\alpha.
\]
At the critical point,
\[
    f\left(
        \theta_\alpha,
        \arccos\left(-\frac{\cos\theta_\alpha}{2}\right)
    \right)
    =
    -\cos\theta_\alpha
    \sqrt{1-\frac{\cos^2\theta_\alpha}{4}}.
\]
Since \(f<0\) on
\[
    \left[\frac{\pi}{2},\,\pi+\theta_\alpha\right],
\]
the minimum of \(|f|\) on this interval is
\[
    \cos\theta_\alpha
    \sqrt{1-\frac{\cos^2\theta_\alpha}{4}}.
\]
This proves the claim.
\end{proof}

\begin{lemma}[Auxiliary angular quotient]
\label{lem:g-bound-+++}
On the upper half, define
\[
    G(\theta_\alpha,\theta_\beta)
    :=
    \left|
    \frac{
        \sin\theta_\alpha\,p(\theta_\alpha,\theta_\beta)
        -
        \cos\theta_\alpha\,r(\theta_\alpha,\theta_\beta)
    }{
        r(\theta_\alpha,\theta_\beta)
    }
    \right|.
\]
Then
\[
    G(\theta_\alpha,\theta_\beta)
    \lesssim
    \begin{cases}
        1,
        & 0\le \theta_\alpha\le \dfrac{\pi}{3},
        \\[0.6em]
        (\cos\theta_\alpha)^{-1/2},
        & \dfrac{\pi}{3}<\theta_\alpha<\dfrac{\pi}{2}.
    \end{cases}
\]
\end{lemma}

\begin{proof}
Write
\[
    c:=\cos\theta_\alpha,
    \qquad
    s:=\sin\theta_\alpha,
    \qquad
    p:=c+\cos\theta_\beta,
    \qquad
    r:=\sqrt{1-p^2}.
\]
Since
\[
    \theta_\beta\in
    \left[\frac{\pi}{2},\,\pi+\theta_\alpha\right],
\]
we have
\[
    c-1\le p\le c.
\]
Moreover,
\[
    G(\theta_\alpha,\theta_\beta)
    \le
    c+\frac{s|p|}{r}.
\]

Assume first that
\[
    0<\theta_\alpha\le\frac{\pi}{3},
\]
so that \(c\ge\frac12\).  Then \(|p|\le c\) and
\[
    r\ge\sqrt{1-c^2}=s.
\]
Consequently,
\[
    G(\theta_\alpha,\theta_\beta)
    \le c+\frac{sc}{s}
    =2c
    \le2.
\]
For \(\theta_\alpha=0\), the same bound follows directly from
\(G(0,\theta_\beta)=1\) on the regular part of the arc.

Assume now that
\[
    \frac{\pi}{3}<\theta_\alpha<\frac{\pi}{2},
\]
so that \(0<c<\frac12\).  In this case,
\[
    |p|\le1-c,
    \qquad
    r\ge\sqrt{1-(1-c)^2}=\sqrt{c(2-c)}.
\]
Hence
\[
    G(\theta_\alpha,\theta_\beta)
    \le
    c+\frac{s(1-c)}{\sqrt{c(2-c)}}
    \lesssim
    c^{-1/2}.
\]
This proves the stated two-case estimate.
\end{proof}

After using momentum conservation to eliminate
\[
    k_\gamma=-k_\alpha-k_\beta,
\]
the reduced frequency constraint is
\[
    \widetilde\Xi_1(k_\beta)
    =
    \frac{k_\alpha^1}{|k_\alpha|}
    +
    \frac{k_\beta^1}{|k_\beta|}
    +
    \frac{k_\gamma^1}{|k_\gamma|}.
\]
On a regular patch parametrized by \(\theta_\beta\), the one-dimensional
delta identity gives the factor
\[
    \frac{
        \left|\dfrac{d k_\beta^1}{d\theta_\beta}\right|
    }{
        \left|
        \partial_{k_\beta^2}\widetilde\Xi_1
        \right|
    }.
\]

\begin{lemma}[Co--area factor on the upper half]
\label{lem:coarea-factor-upper-+++}
On the upper half
\[
    \theta_\beta
    \in
    \left(\frac{\pi}{2},\,\pi+\theta_\alpha\right),
\]
one has
\[
\begin{aligned}
    \frac{
        \left|\dfrac{d k_\beta^1}{d\theta_\beta}\right|
    }{
        \left|
        \partial_{k_\beta^2}\widetilde\Xi_1
        \right|
    }
    &=
    \left|
    \frac{
        \bigl[
        \sin\theta_\alpha\,p
        -
        \cos\theta_\alpha\,r
        \bigr]
        \bigl[
        \cos\theta_\alpha\sin\theta_\beta
        -
        \sin\theta_\alpha\cos\theta_\beta
        \bigr]
    }{
        r\,f^3
    }
    \right|,
\end{aligned}
\]
where
\[
    p=p(\theta_\alpha,\theta_\beta),
    \qquad
    r=r(\theta_\alpha,\theta_\beta),
    \qquad
    f=f(\theta_\alpha,\theta_\beta).
\]
Moreover,
\[
    \frac{
        \left|\dfrac{d k_\beta^1}{d\theta_\beta}\right|
    }{
        \left|
        \partial_{k_\beta^2}\widetilde\Xi_1
        \right|
    }
    \lesssim
    \begin{cases}
        1,
        & 0\le \theta_\alpha\le\dfrac{\pi}{3},
        \\[0.6em]
        (\cos\theta_\alpha)^{-7/2},
        & \dfrac{\pi}{3}<\theta_\alpha<\dfrac{\pi}{2}.
    \end{cases}
\]
The implicit constant is independent of \(\theta_\alpha\).
\end{lemma}

\begin{proof}
For readability, write
\[
    p=p(\theta_\alpha,\theta_\beta),
    \qquad
    r=r(\theta_\alpha,\theta_\beta),
    \qquad
    f=f(\theta_\alpha,\theta_\beta).
\]
A direct differentiation of the parametrization gives
\[
    \frac{d k_\beta^1}{d\theta_\beta}
    =
    \frac{J_1}{r\,f^2},
\]
where
\[
\begin{aligned}
    J_1
    &:=
    \sin\theta_\beta\cos\theta_\beta
    \bigl(
        \cos\theta_\alpha\sin\theta_\beta
        -
        \sin\theta_\alpha\cos\theta_\beta
    \bigr)
    \\
    &\qquad
    +
    pr
    \bigl(
        \sin\theta_\alpha\,p
        -
        \cos\theta_\alpha\,r
    \bigr).
\end{aligned}
\]
On the other hand, differentiating the reduced constraint gives
\[
    \partial_{k_\beta^2}\widetilde\Xi_1
    =
    -f
    \frac{
        J_1
    }{
        \bigl(
            \sin\theta_\alpha\,p
            -
            \cos\theta_\alpha\,r
        \bigr)
        \bigl(
            \cos\theta_\alpha\sin\theta_\beta
            -
            \sin\theta_\alpha\cos\theta_\beta
        \bigr)
    }.
\]
Dividing these two identities gives
\[
\begin{aligned}
    \frac{
        \left|\dfrac{d k_\beta^1}{d\theta_\beta}\right|
    }{
        \left|
        \partial_{k_\beta^2}\widetilde\Xi_1
        \right|
    }
    &=
    \left|
    \frac{
        \bigl[
        \sin\theta_\alpha\,p
        -
        \cos\theta_\alpha\,r
        \bigr]
        \bigl[
        \cos\theta_\alpha\sin\theta_\beta
        -
        \sin\theta_\alpha\cos\theta_\beta
        \bigr]
    }{
        r\,f^3
    }
    \right|.
\end{aligned}
\]

We now estimate the right-hand side.  Since
\[
    \left|
        \cos\theta_\alpha\sin\theta_\beta
        -
        \sin\theta_\alpha\cos\theta_\beta
    \right|
    =
    |\sin(\theta_\beta-\theta_\alpha)|
    \le 1,
\]
we get
\[
\begin{aligned}
    \frac{
        \left|\dfrac{d k_\beta^1}{d\theta_\beta}\right|
    }{
        \left|
        \partial_{k_\beta^2}\widetilde\Xi_1
        \right|
    }
    &\le
    G(\theta_\alpha,\theta_\beta)
    |f(\theta_\alpha,\theta_\beta)|^{-3}.
\end{aligned}
\]
By Lemmas~\ref{lem:g-bound-+++} and~\ref{lem:f-lower-+++},
\[
    G(\theta_\alpha,\theta_\beta)
    \lesssim
    \begin{cases}
        1,
        & 0\le \theta_\alpha\le \dfrac{\pi}{3},
        \\[0.4em]
        (\cos\theta_\alpha)^{-1/2},
        & \dfrac{\pi}{3}<\theta_\alpha<\dfrac{\pi}{2},
    \end{cases}
\]
and
\[
    |f(\theta_\alpha,\theta_\beta)|^{-3}
    \lesssim
    (\cos\theta_\alpha)^{-3}.
\]
If \(0\le\theta_\alpha\le\pi/3\), then
\(\cos\theta_\alpha\ge1/2\), and hence
\[
    G(\theta_\alpha,\theta_\beta)
    |f(\theta_\alpha,\theta_\beta)|^{-3}
    \lesssim 1.
\]
If \(\pi/3<\theta_\alpha<\pi/2\), then
\[
    G(\theta_\alpha,\theta_\beta)
    |f(\theta_\alpha,\theta_\beta)|^{-3}
    \lesssim
    (\cos\theta_\alpha)^{-7/2}.
\]
This proves the desired estimate.
\end{proof}

\begin{lemma}[Co--area factor on the lower half]
\label{lem:coarea-factor-lower-+++}
On the lower half of the unit-scale \((+,+,+)\) resonance curve, one has
\[
    \frac{
        \left|\dfrac{d k_\gamma^1}{d\theta_\gamma}\right|
    }{
        \left|
        \partial_{k_\gamma^2}\overline\Xi_1
        \right|
    }
    \lesssim
    \begin{cases}
        1,
        & 0\le \theta_\alpha\le\dfrac{\pi}{3},
        \\[0.6em]
        (\cos\theta_\alpha)^{-7/2},
        & \dfrac{\pi}{3}<\theta_\alpha<\dfrac{\pi}{2}.
    \end{cases}
\]
Here
\[
    \overline\Xi_1(k_\gamma)
    =
    \frac{k_\alpha^1}{|k_\alpha|}
    +
    \frac{k_\beta^1}{|k_\beta|}
    +
    \frac{k_\gamma^1}{|k_\gamma|},
    \qquad
    k_\beta=-k_\alpha-k_\gamma .
\]
\end{lemma}

\begin{proof}
The lower half is obtained from the upper half by the central symmetry
\[
    k_\beta\mapsto -k_\alpha-k_\beta.
\]
This map exchanges \(k_\beta\) and \(k_\gamma\).  Therefore, on the lower half
we use \(\theta_\gamma\) as the angular parameter and eliminate
\[
    k_\beta=-k_\alpha-k_\gamma.
\]
The corresponding reduced frequency constraint is
\[
    \overline\Xi_1(k_\gamma)
    =
    \frac{k_\alpha^1}{|k_\alpha|}
    +
    \frac{k_\beta^1}{|k_\beta|}
    +
    \frac{k_\gamma^1}{|k_\gamma|}.
\]

Define
\[
    p_\gamma(\theta_\alpha,\theta_\gamma)
    :=
    \cos\theta_\alpha+\cos\theta_\gamma,
\]
\[
    r_\gamma(\theta_\alpha,\theta_\gamma)
    :=
    \sqrt{1-p_\gamma(\theta_\alpha,\theta_\gamma)^2},
\]
and
\[
    f_\gamma(\theta_\alpha,\theta_\gamma)
    :=
    -\sin\theta_\gamma\,
    p_\gamma(\theta_\alpha,\theta_\gamma)
    +
    \cos\theta_\gamma\,
    r_\gamma(\theta_\alpha,\theta_\gamma).
\]
Repeating the computation in Lemma~\ref{lem:coarea-factor-upper-+++}, with
\(\theta_\gamma\) replacing \(\theta_\beta\), gives
\[
\begin{aligned}
    \frac{
        \left|\dfrac{d k_\gamma^1}{d\theta_\gamma}\right|
    }{
        \left|
        \partial_{k_\gamma^2}\overline\Xi_1
        \right|
    }
    &=
    \left|
    \frac{
        \bigl[
        \sin\theta_\alpha\,p_\gamma
        -
        \cos\theta_\alpha\,r_\gamma
        \bigr]
        \bigl[
        \cos\theta_\alpha\sin\theta_\gamma
        -
        \sin\theta_\alpha\cos\theta_\gamma
        \bigr]
    }{
        r_\gamma\,f_\gamma^3
    }
    \right|.
\end{aligned}
\]
The lower bound for \(f_\gamma\) and the corresponding auxiliary angular
quotient estimate are identical to Lemmas~\ref{lem:f-lower-+++} and
\ref{lem:g-bound-+++}.  Hence the same bound follows:
\[
    \frac{
        \left|\dfrac{d k_\gamma^1}{d\theta_\gamma}\right|
    }{
        \left|
        \partial_{k_\gamma^2}\overline\Xi_1
        \right|
    }
    \lesssim
    \begin{cases}
        1,
        & 0\le \theta_\alpha\le\dfrac{\pi}{3},
        \\[0.6em]
        (\cos\theta_\alpha)^{-7/2},
        & \dfrac{\pi}{3}<\theta_\alpha<\dfrac{\pi}{2}.
    \end{cases}
\]
\end{proof}
\subsection{A priori boundedness estimate for the \((+,+,+)\) contribution}
\label{subsec:apriori-+++}

We now return from the unit-scale normalization to general
\(k_\alpha\neq0\).  Write
\[
    k_\alpha=|k_\alpha|\varsigma_\alpha,
    \qquad
    k_\beta=|k_\alpha|\varsigma_\beta,
    \qquad
    k_\gamma=|k_\alpha|\varsigma_\gamma,
    \qquad
    |\varsigma_\alpha|=1.
\]
The estimates in the previous subsection were obtained for the unit-scale
variables \(\varsigma_\beta,\varsigma_\gamma\).  By the scaling reduction
in Section~\ref{Scaling and co–area representation}, the \((+,+,+)\)
contribution has an overall factor \(|k_\alpha|^4\).

We denote by
\[
    \mathcal C_{+++}(n)(k_\alpha)
\]
the contribution to the collision operator with
\[
    \sigma_\alpha=\sigma_\beta=\sigma_\gamma=+1 .
\]
By the central symmetry of the unit-scale curve, it is enough to integrate over
the upper half and multiply by \(2\).  On this half,
\[
    \theta_\beta\in\left[\frac{\pi}{2},\,\pi+\theta_\alpha\right],
    \qquad
    \sin\theta_\gamma\le0,
\]
and we recall the notation
\[
    p=\cos\theta_\alpha+\cos\theta_\beta,
    \qquad
    r=\sqrt{1-p^2},
    \qquad
    f=-\sin\theta_\beta\,p+\cos\theta_\beta\,r .
\]
Thus
\[
    \varsigma_\alpha
    =
    (\cos\theta_\alpha,\sin\theta_\alpha),
\]
\[
    \varsigma_\beta
    =
    \frac{
        \sin\theta_\alpha\,p-\cos\theta_\alpha\,r
    }{
        f
    }
    (\cos\theta_\beta,\sin\theta_\beta),
\]
and
\[
    \varsigma_\gamma
    =
    \frac{
        \sin(\theta_\beta-\theta_\alpha)
    }{
        -f
    }
    (-p,-r).
\]

\begin{proposition}[A priori bound for the \((+,+,+)\) contribution]
\label{prop:apriori-+++}
Let \(m\ge4\), and fix the angular cut-off scale \(0<\varepsilon\ll1\).
Then, for every \(n\in L^\infty_m\),
\[
    \sup_{k_\alpha\in\mathbb R^2}
    \langle k_\alpha\rangle^m
    \left|
        \mathcal C_{+++}(n)(k_\alpha)
    \right|
    \lesssim_{\varepsilon,m}
    \|n\|_{L^\infty_m}^{2}.
\]
More precisely, the implicit constant may be chosen of order
\[
    \varepsilon^{-\left(\frac{11}{2}+2m\right)}.
\]
\end{proposition}
Before proving Proposition~\ref{prop:apriori-+++}, we record the two
cut-off positivity estimates needed below.

\begin{lemma}[Positivity of \(\varsigma_\beta\)]
\label{prop:positivity-beta}
Let \(0<\varepsilon\ll1\).  On the upper half of the unit-scale
\((+,+,+)\) curve, assume
\[
    |\cos\theta_\alpha|>\varepsilon,
    \qquad
    |\cos\theta_\beta|>\varepsilon .
\]
Then
\[
    |\varsigma_\beta|
    \gtrsim
    \varepsilon .
\]
More precisely,
\[
    \inf_{\substack{
        |\cos\theta_\alpha|>\varepsilon\\
        |\cos\theta_\beta|>\varepsilon
    }}
    |\varsigma_\beta|
    \ge
    \inf_{\substack{
        |\cos\theta_\alpha|>\varepsilon\\
        |\cos\theta_\beta|>\varepsilon
    }}
    \left|
        \sin\theta_\alpha\,p-\cos\theta_\alpha\,r
    \right|
    \gtrsim
    \varepsilon .
\]
\end{lemma}

\begin{proof}
On the upper half,
\[
    |\varsigma_\beta|
    =
    \left|
    \frac{\sin\theta_\alpha p-\cos\theta_\alpha r}{f}
    \right|.
\]
Since \(p^2+r^2=1\), we have
\[
    |f|
    =
    |-\sin\theta_\beta p+\cos\theta_\beta r|
    \le 1.
\]
Hence
\[
    |\varsigma_\beta|
    \ge
    |\sin\theta_\alpha p-\cos\theta_\alpha r|.
\]
It remains to obtain a lower bound for
\[
    F_1(\theta_\beta)
    :=
    \sin\theta_\alpha p(\theta_\beta)
    -
    \cos\theta_\alpha r(\theta_\beta).
\]

Since
\[
    p=-\cos\theta_\gamma,
    \qquad
    r=|\sin\theta_\gamma|,
    \qquad
    \sin\theta_\gamma\le0
\]
on the upper half, we have
\[
    F_1(\theta_\beta)
    =
    \sin(\theta_\gamma-\theta_\alpha).
\]
Moreover, on the upper half,
\[
    \theta_\beta\in\left[\frac{\pi}{2},\,\pi+\theta_\alpha\right],
\]
and the endpoint \(\theta_\beta=\pi/2\) corresponds to the zero-frequency
direction for \(\beta\), while \(\theta_\beta=\pi+\theta_\alpha\) corresponds
to the endpoint where \(\cos\theta_\gamma=0\).  After imposing
\[
    |\cos\theta_\beta|>\varepsilon,
\]
the closest allowed point to \(\theta_\beta=\pi/2\) is determined by
\[
    \cos\theta_\beta=-\varepsilon.
\]
Consequently,
\[
    p=\cos\theta_\alpha-\varepsilon,
    \qquad
    r=\sqrt{1-(\cos\theta_\alpha-\varepsilon)^2},
\]
and therefore the boundary value is
\[
    \sin\theta_\alpha(-\cos\theta_\alpha+\varepsilon)
    +
    \cos\theta_\alpha
    \sqrt{1-(\cos\theta_\alpha-\varepsilon)^2}.
\]
The other relevant endpoint gives the value \(\cos\theta_\alpha\).  Hence
\[
\begin{aligned}
    \inf_{\substack{
        |\cos\theta_\alpha|>\varepsilon\\
        |\cos\theta_\beta|>\varepsilon
    }}
    |F_1(\theta_\beta)|
    &\ge
    \inf_{|\cos\theta_\alpha|>\varepsilon}
    \min\Bigl\{
        \sin\theta_\alpha(-\cos\theta_\alpha+\varepsilon)
        +
        \cos\theta_\alpha
        \sqrt{1-(\cos\theta_\alpha-\varepsilon)^2},
        \\
        &\hspace{9em}
        \cos\theta_\alpha
    \Bigr\}.
\end{aligned}
\]
It remains to bound the first term from below.  Define
\[
    F_2(\theta_\alpha)
    :=
    \sin\theta_\alpha(-\cos\theta_\alpha+\varepsilon)
    +
    \cos\theta_\alpha
    \sqrt{1-(\cos\theta_\alpha-\varepsilon)^2}.
\]
A direct differentiation gives
\[
\frac{dF_2}{d\theta_\alpha}
=
\left(
    \frac{
        \cos\theta_\alpha(\cos\theta_\alpha-\varepsilon)
    }{
        \sqrt{1-(\cos\theta_\alpha-\varepsilon)^2}
    }
    +
    \sin\theta_\alpha
\right)
\left(
    \sin\theta_\alpha
    -
    \sqrt{1-(\cos\theta_\alpha-\varepsilon)^2}
\right).
\]
For
\[
    0\le\theta_\alpha\le\arccos\varepsilon,
\]
we have \(\cos\theta_\alpha-\varepsilon\ge0\), and the second factor is
non-positive.  Hence \(F_2\) is decreasing on this interval.  Therefore its
minimum on the cut-off support is attained at
\[
    \theta_\alpha=\arccos\varepsilon,
\]
and gives a lower bound comparable to \(\varepsilon\).  Since the second
candidate \(\cos\theta_\alpha\) is also bounded below by \(\varepsilon\) on
the support of the cut-off, we obtain
\[
    |\varsigma_\beta|\gtrsim\varepsilon .
\]
\end{proof}

\begin{lemma}[Positivity of \(\varsigma_\gamma\)]
\label{Positivity of gamma}
Let \(0<\varepsilon\ll1\).  On the upper half of the unit-scale
\((+,+,+)\) curve, assume
\[
    |\cos\theta_\alpha|>\varepsilon,
    \qquad
    |\cos\theta_\gamma|>\varepsilon .
\]
Then
\[
    |\varsigma_\gamma|
    \gtrsim
    \varepsilon .
\]
More precisely,
\[
    \inf_{\substack{
        |\cos\theta_\alpha|>\varepsilon\\
        |\cos\theta_\gamma|>\varepsilon
    }}
    |\varsigma_\gamma|
    \ge
    \inf_{\substack{
        |\cos\theta_\alpha|>\varepsilon\\
        |\cos\theta_\gamma|>\varepsilon
    }}
    |\sin(\theta_\beta-\theta_\alpha)|
    \gtrsim
    \varepsilon .
\]
\end{lemma}

\begin{proof}
On the upper half,
\[
    |\varsigma_\gamma|
    =
    \left|
    \frac{\sin(\theta_\beta-\theta_\alpha)}{-f}
    \right|.
\]

Again, since \(|f|\le1\), it
suffices to prove a lower bound for
\[
    |\sin(\theta_\beta-\theta_\alpha)|.
\]

The function
\[
    \sin(\theta_\beta-\theta_\alpha)
\]
is positive on
\[
    \theta_\beta\in\left[\frac{\pi}{2},\,\pi+\theta_\alpha\right],
\]
increases on
\[
    \left[\frac{\pi}{2},\,\frac{\pi}{2}+\theta_\alpha\right],
\]
and decreases on
\[
    \left[\frac{\pi}{2}+\theta_\alpha,\,\pi+\theta_\alpha\right].
\]
Therefore the minimum on the cut-off support is attained at one of the
cut-off endpoints determined by
\[
    |\cos\theta_\gamma|=\varepsilon,
\]
or at the endpoint giving \(\cos\theta_\alpha\).

There are two cases, depending on whether the endpoint
\[
    \cos\theta_\gamma=\varepsilon
\]
is present.

First, when
\[
    \theta_\alpha\in[0,\arccos(1-\varepsilon)],
\]
we get the endpoint
\[
    \theta_\beta=\arccos(-\cos\theta_\alpha+\varepsilon).
\]
Define
\[
    F_3(\theta_\alpha)
    :=
    \sin\left(
        \arccos(-\cos\theta_\alpha+\varepsilon)
        -
        \theta_\alpha
    \right).
\]
Equivalently,
\[
    F_3(\theta_\alpha)
    =
    \cos\theta_\alpha
    \sqrt{1-(\cos\theta_\alpha-\varepsilon)^2}
    +
    \sin\theta_\alpha(\cos\theta_\alpha-\varepsilon).
\]
A direct differentiation gives
\[
\frac{dF_3}{d\theta_\alpha}
=
\frac{
    \left[
        \cos\theta_\alpha(\cos\theta_\alpha-\varepsilon)
        -
        \sin\theta_\alpha
        \sqrt{1-(\cos\theta_\alpha-\varepsilon)^2}
    \right]
    \left[
        \sqrt{1-(\cos\theta_\alpha-\varepsilon)^2}
        +
        \sin\theta_\alpha
    \right]
}{
    \sqrt{1-(\cos\theta_\alpha-\varepsilon)^2}
}.
\]
On the interval
\[
    \theta_\alpha\in[0,\arccos(1-\varepsilon)],
\]
the first bracket is nonnegative for \(0<\varepsilon\ll1\).  Hence
\(F_3\) is increasing, and therefore
\[
    F_3(\theta_\alpha)
    \ge
    F_3(0)
    =
    \sqrt{2\varepsilon-\varepsilon^2}
    \gtrsim
    \sqrt\varepsilon .
\]
In particular, this contribution is \(\gtrsim\varepsilon\).

Second, when
\[
    \theta_\alpha\in[\arccos(1-\varepsilon),\,\arccos\varepsilon],
\]
the relevant endpoint is
\[
    \theta_\beta=\arccos(-\cos\theta_\alpha-\varepsilon).
\]
Define
\[
    F_4(\theta_\alpha)
    :=
    \sin\left(
        \arccos(-\cos\theta_\alpha-\varepsilon)
        -
        \theta_\alpha
    \right).
\]
Then
\[
    F_4(\theta_\alpha)
    =
    \cos\theta_\alpha
    \sqrt{1-(\cos\theta_\alpha+\varepsilon)^2}
    +
    \sin\theta_\alpha(\cos\theta_\alpha+\varepsilon).
\]
Differentiating gives
\[
\frac{dF_4}{d\theta_\alpha}
=
\frac{
    \left[
        \sqrt{1-(\cos\theta_\alpha+\varepsilon)^2}
        +
        \sin\theta_\alpha
    \right]
    H(\theta_\alpha)
}{
    \sqrt{1-(\cos\theta_\alpha+\varepsilon)^2}
},
\]
where
\[
    H(\theta_\alpha)
    :=
    \cos\theta_\alpha(\cos\theta_\alpha+\varepsilon)
    -
    \sin\theta_\alpha
    \sqrt{1-(\cos\theta_\alpha+\varepsilon)^2}.
\]
Moreover,
\[
\frac{dH}{d\theta_\alpha}
=
-\frac{
    \left[
        \sqrt{1-(\cos\theta_\alpha+\varepsilon)^2}
        +
        \sin\theta_\alpha
    \right]
    \left[
        \varepsilon\sin\theta_\alpha
        +
        \cos\theta_\alpha
        \sqrt{1-(\cos\theta_\alpha+\varepsilon)^2}
        +
        \sin\theta_\alpha\cos\theta_\alpha
    \right]
}{
    \sqrt{1-(\cos\theta_\alpha+\varepsilon)^2}
}
\le0.
\]
Thus \(H\) is decreasing.  Since \(H\) vanishes at
\[
    \theta_\alpha
    =
    \arccos\left(
        \frac{\sqrt{2-\varepsilon^2}-\varepsilon}{2}
    \right),
\]
the function \(F_4\) increases before this point and decreases after this
point.  Hence its minimum on
\[
    [\arccos(1-\varepsilon),\,\arccos\varepsilon]
\]
is attained at one of the endpoints.  The endpoint values give
\[
    \inf_{\theta_\alpha\in[\arccos(1-\varepsilon),\,\arccos\varepsilon]}
    F_4(\theta_\alpha)
    \ge
    \min\left\{
        \sqrt{2\varepsilon-\varepsilon^2},
        \;
        \varepsilon\sqrt{1-4\varepsilon^2}
        +
        2\varepsilon\sqrt{1-\varepsilon^2}
    \right\}
    \gtrsim
    \varepsilon .
\]
Combining the two cases, and using also
\[
    \cos\theta_\alpha\ge\varepsilon
\]
on the cut-off support, gives
\[
    |\varsigma_\gamma|\gtrsim\varepsilon .
\]
\end{proof}


\begin{proof}[Proof of Proposition~\ref{prop:apriori-+++}]
It suffices to estimate the upper-half contribution.  The lower half is
obtained by the central symmetry
\[
    k_\beta\mapsto -k_\alpha-k_\beta,
\]
which exchanges \(k_\beta\) and \(k_\gamma\).

By the scaled co-area representation,
\[
\begin{aligned}
\left|
    \mathcal C_{+++}^{\mathrm{up}}(n)(k_\alpha)
\right|
&\lesssim
|k_\alpha|^4
\int_{\pi/2}^{\pi+\theta_\alpha}
    \chi_\alpha\chi_\beta\chi_\gamma\,
    |\widetilde I(\varsigma_\beta^1,\varsigma_\beta^2)|
\\
&\qquad\qquad\times
    \frac{
        \left|\dfrac{d k_\beta^1}{d\theta_\beta}\right|
    }{
        \left|
        \partial_{k_\beta^2}\widetilde\Xi_1
        \right|
    }
    d\theta_\beta .
\end{aligned}
\]
On the cut-off support, Lemma~\ref{lem:coarea-factor-upper-+++} gives
\[
    \frac{
        \left|\dfrac{d k_\beta^1}{d\theta_\beta}\right|
    }{
        \left|
        \partial_{k_\beta^2}\widetilde\Xi_1
        \right|
    }
    \lesssim
    \varepsilon^{-7/2}.
\]
Moreover, by the unit-scale radius upper bounds,
\[
    1+|\varsigma_\beta|+|\varsigma_\gamma|
    \lesssim
    \varepsilon^{-1},
\]
and by Lemmas~\ref{prop:positivity-beta} and~\ref{Positivity of gamma},
\[
    |\varsigma_\beta|\gtrsim\varepsilon,
    \qquad
    |\varsigma_\gamma|\gtrsim\varepsilon .
\]
Using the definition of \(\widetilde I\), we therefore obtain
\[
\begin{aligned}
|\widetilde I(\varsigma_\beta^1,\varsigma_\beta^2)|
&\lesssim
\varepsilon^{-2}
\|n\|_{L^\infty_m}^{2}
\Bigl(
    \langle |k_\alpha|\varsigma_\beta\rangle^{-m}
    \langle |k_\alpha|\varsigma_\gamma\rangle^{-m}
\\
&\qquad\qquad
    +
    \langle k_\alpha\rangle^{-m}
    \langle |k_\alpha|\varsigma_\beta\rangle^{-m}
    +
    \langle k_\alpha\rangle^{-m}
    \langle |k_\alpha|\varsigma_\gamma\rangle^{-m}
\Bigr).
\end{aligned}
\]
The factor \(\varepsilon^{-2}\) comes from
\[
    (1+|\varsigma_\beta|+|\varsigma_\gamma|)^2
    \lesssim \varepsilon^{-2}
\]
in the interaction coefficient; all remaining angular factors in
\(\widetilde I\) are uniformly bounded.

Combining the last two estimates gives
\[
\begin{aligned}
&\langle k_\alpha\rangle^m
\left|
    \mathcal C_{+++}^{\mathrm{up}}(n)(k_\alpha)
\right|
\\
&\qquad\lesssim
\varepsilon^{-11/2}
\|n\|_{L^\infty_m}^{2}
\langle k_\alpha\rangle^{m+4}
\int_{\pi/2}^{\pi+\theta_\alpha}
\Bigl(
    \langle |k_\alpha|\varsigma_\beta\rangle^{-m}
    \langle |k_\alpha|\varsigma_\gamma\rangle^{-m}
\\
&\hspace{13em}
    +
    \langle k_\alpha\rangle^{-m}
    \langle |k_\alpha|\varsigma_\beta\rangle^{-m}
    +
    \langle k_\alpha\rangle^{-m}
    \langle |k_\alpha|\varsigma_\gamma\rangle^{-m}
\Bigr)
d\theta_\beta .
\end{aligned}
\]
Since
\[
    |\varsigma_\beta|\gtrsim\varepsilon,
    \qquad
    |\varsigma_\gamma|\gtrsim\varepsilon,
\]
we have
\[
    \langle |k_\alpha|\varsigma_\beta\rangle^{-m},
    \;
    \langle |k_\alpha|\varsigma_\gamma\rangle^{-m}
    \lesssim
    \langle \varepsilon |k_\alpha|\rangle^{-m}.
\]
The integration interval has bounded length, and hence
\[
\begin{aligned}
\langle k_\alpha\rangle^m
\left|
    \mathcal C_{+++}^{\mathrm{up}}(n)(k_\alpha)
\right|
&\lesssim
\varepsilon^{-11/2}
\|n\|_{L^\infty_m}^{2}
\langle k_\alpha\rangle^{m+4}
\\
&\qquad\times
\left(
    \langle \varepsilon |k_\alpha|\rangle^{-2m}
    +
    2\langle k_\alpha\rangle^{-m}
    \langle \varepsilon |k_\alpha|\rangle^{-m}
\right).
\end{aligned}
\]
Finally, since \(0<\varepsilon<1\),
\[
    \langle \varepsilon |k_\alpha|\rangle^{-m}
    \le
    \varepsilon^{-m}
    \langle k_\alpha\rangle^{-m}.
\]
Thus
\[
\begin{aligned}
\langle k_\alpha\rangle^m
\left|
    \mathcal C_{+++}^{\mathrm{up}}(n)(k_\alpha)
\right|
&\lesssim
\varepsilon^{-\left(\frac{11}{2}+2m\right)}
\|n\|_{L^\infty_m}^{2}
\langle k_\alpha\rangle^{4-m}
\\
&\lesssim
\varepsilon^{-\left(\frac{11}{2}+2m\right)}
\|n\|_{L^\infty_m}^{2},
\end{aligned}
\]
because \(m\ge4\).  The lower-half contribution satisfies the same estimate by
central symmetry.  This proves the proposition.
\end{proof}

\section{Parametrization and Estimates for the $(+,+,-)$ Resonance Manifold}\label{Parametrization and Estimates for the $(+,+,-)$ Resonance Manifold}


\subsection{Unit--scale parametrization}
\label{Parametrization(+,+,-)}

We consider the sign configuration
\[
    (\sigma_\alpha,\sigma_\beta,\sigma_\gamma)=(+,+,-),
\]
corresponding to \eqref{manifold1}.  By homogeneity, we work at unit scale:
\[
    |k_\alpha|=1,
    \qquad
    k_\alpha=(\cos\theta_\alpha,\sin\theta_\alpha),
    \qquad
    0\le \theta_\alpha<\frac{\pi}{2}.
\]
The \((+,+,-)\) unit-scale resonance curve is non-compact.  It has two
asymptotic ends.  In addition, when
\[
    \frac{\pi}{3}\le \theta_\alpha<\frac{\pi}{2},
\]
the curve reaches \(k_\beta=0\).  This point is a cusp of the resonance curve
in the \(k_\beta\)-plane.

As before, write
\[
    k_j=|k_j|(\cos\theta_j,\sin\theta_j),
    \qquad
    j\in\{\alpha,\beta,\gamma\}.
\]
Solving the polar momentum system gives
\begin{equation}
\label{eq:polar-momentum-++-}
    |k_\beta|
    =
    \frac{\sin(\theta_\gamma-\theta_\alpha)}
         {\sin(\theta_\beta-\theta_\gamma)},
    \qquad
    |k_\gamma|
    =
    \frac{\sin(\theta_\beta-\theta_\alpha)}
         {\sin(\theta_\gamma-\theta_\beta)}.
\end{equation}
For the \((+,+,-)\) configuration, the frequency resonance is
\[
    \cos\theta_\alpha+\cos\theta_\beta-\cos\theta_\gamma=0,
\]
so that
\[
    \cos\theta_\gamma
    =
    \cos\theta_\alpha+\cos\theta_\beta.
\]
Substituting the two choices
\[
    \sin\theta_\gamma
    =
    \pm\sqrt{1-(\cos\theta_\alpha+\cos\theta_\beta)^2}
\]
into \eqref{eq:polar-momentum-++-}, and imposing
\[
    |k_\beta|\ge0,
    \qquad
    |k_\gamma|\ge0,
\]
gives the following parametrization.


\begin{proposition}[Unit--scale parametrization of the \((+,+,-)\) resonance curve]
\label{prop:param-++-}
Let \(0\le\theta_\alpha<\pi/2\).  Define the two asymptotic angles
\[
    \theta_-^\infty
    :=
    \pi-\arccos\left(\frac{\cos\theta_\alpha}{2}\right),
    \qquad
    \theta_+^\infty
    :=
    \pi+\arccos\left(\frac{\cos\theta_\alpha}{2}\right).
\]

\medskip
\noindent\textbf{Case 1.} If
\begin{equation}
\label{case1}
    0\le\theta_\alpha\le \frac{\pi}{3},
\end{equation}
then the unit-scale \((+,+,-)\) resonance curve is parametrized by
\[
    k_\beta(\theta_\beta)
    =
    |k_\beta|(\cos\theta_\beta,\sin\theta_\beta),
\]
where
\begin{equation}
\label{eq:param-++--case1}
|k_\beta|
=
\begin{cases}
\displaystyle
\frac{
    \sin\theta_\alpha(\cos\theta_\alpha+\cos\theta_\beta)
    +
    \cos\theta_\alpha
    \sqrt{1-(\cos\theta_\alpha+\cos\theta_\beta)^2}
}{
    -\sin\theta_\beta(\cos\theta_\alpha+\cos\theta_\beta)
    -
    \cos\theta_\beta
    \sqrt{1-(\cos\theta_\alpha+\cos\theta_\beta)^2}
},
&
\begin{array}{l}
\theta_\beta\in
\left(
    \theta_-^\infty,\,
    \pi+\theta_\alpha
\right],\\
\sin\theta_\gamma\le0,
\end{array}
\\[2.0em]
\displaystyle
\frac{
    \sin\theta_\alpha(\cos\theta_\alpha+\cos\theta_\beta)
    -
    \cos\theta_\alpha
    \sqrt{1-(\cos\theta_\alpha+\cos\theta_\beta)^2}
}{
    -\sin\theta_\beta(\cos\theta_\alpha+\cos\theta_\beta)
    +
    \cos\theta_\beta
    \sqrt{1-(\cos\theta_\alpha+\cos\theta_\beta)^2}
},
&
\begin{array}{l}
\theta_\beta\in
\left[
    \pi+\theta_\alpha,\,
    \theta_+^\infty
\right),\\
\sin\theta_\gamma\ge0.
\end{array}
\end{cases}
\end{equation}
The first line corresponds to \(\sin\theta_\gamma\le0\), while the second line
corresponds to \(\sin\theta_\gamma\ge0\).

For \(0\le\theta_\alpha<\pi/3\), this curve does not reach \(k_\beta=0\).
At the threshold
\[
    \theta_\alpha=\frac{\pi}{3},
\]
the cusp \(k_\beta=0\) appears in the degenerate limiting form, with the single
limiting direction
\[
    \theta_\beta=\pi.
\]

\medskip
\noindent\textbf{Case 2.} If
\begin{equation}
\label{case2}
    \frac{\pi}{3}<\theta_\alpha<\frac{\pi}{2},
\end{equation}
define the two limiting cusp directions
\[
    \theta_\beta^\ast
    :=
    \pi-\arccos(2\cos\theta_\alpha),
    \qquad
    \widetilde{\theta}_\beta^\ast
    :=
    \pi+\arccos(2\cos\theta_\alpha).
\]
Then the unit-scale \((+,+,-)\) resonance curve is parametrized by
\[
    k_\beta(\theta_\beta)
    =
    |k_\beta|(\cos\theta_\beta,\sin\theta_\beta),
\]
where
\begin{equation}
\label{eq:param-++--case2}
|k_\beta|
=
\begin{cases}
\displaystyle
\frac{
    \sin\theta_\alpha(\cos\theta_\alpha+\cos\theta_\beta)
    +
    \cos\theta_\alpha
    \sqrt{1-(\cos\theta_\alpha+\cos\theta_\beta)^2}
}{
    -\sin\theta_\beta(\cos\theta_\alpha+\cos\theta_\beta)
    -
    \cos\theta_\beta
    \sqrt{1-(\cos\theta_\alpha+\cos\theta_\beta)^2}
},
&
\begin{array}{l}
\theta_\beta\in
\left(
    \theta_-^\infty,\,
    \theta_\beta^\ast
\right]
\cup
\left[
    \widetilde{\theta}_\beta^\ast,\,
    \pi+\theta_\alpha
\right],\\
\sin\theta_\gamma\le0,
\end{array}
\\[2.0em]
\displaystyle
\frac{
    \sin\theta_\alpha(\cos\theta_\alpha+\cos\theta_\beta)
    -
    \cos\theta_\alpha
    \sqrt{1-(\cos\theta_\alpha+\cos\theta_\beta)^2}
}{
    -\sin\theta_\beta(\cos\theta_\alpha+\cos\theta_\beta)
    +
    \cos\theta_\beta
    \sqrt{1-(\cos\theta_\alpha+\cos\theta_\beta)^2}
},
&
\begin{array}{l}
\theta_\beta\in
\left[
    \pi+\theta_\alpha,\,
    \theta_+^\infty
\right),\\
\sin\theta_\gamma\ge0.
\end{array}
\end{cases}
\end{equation}
Again, the first line corresponds to \(\sin\theta_\gamma\le0\), while the
second line corresponds to \(\sin\theta_\gamma\ge0\).

The two values
\[
    \theta_\beta^\ast,
    \qquad
    \widetilde{\theta}_\beta^\ast
\]
both correspond to the same physical point
\[
    k_\beta=0.
\]
Since polar coordinates are singular at \(k_\beta=0\), these two angles should
be understood as the two limiting approach directions to the cusp.

In both cases, the point
\[
    \theta_\beta=\pi+\theta_\alpha
\]
corresponds to the collapsed endpoint
\[
    k_\gamma=0.
\]
Indeed, at this endpoint,
\[
    |k_\beta|=1,
    \qquad
    k_\beta=-k_\alpha.
\]

Finally, the two non-compact ends are given by
\[
    \lim_{\theta_\beta\downarrow\theta_-^\infty}
    |k_\beta(\theta_\beta)|
    =
    \infty,
    \qquad
    \lim_{\theta_\beta\uparrow\theta_+^\infty}
    |k_\beta(\theta_\beta)|
    =
    \infty.
\]
\end{proposition}
%
For later use, we also record the corresponding expression for
\(|k_\gamma|\).  By \eqref{eq:polar-momentum-++-}, on the branch
\(\sin\theta_\gamma\le0\),
\begin{equation}
\label{eq:kgamma-++--minus}
|k_\gamma|
=
\frac{
    \sin(\theta_\beta-\theta_\alpha)
}{
    -\sin\theta_\beta(\cos\theta_\alpha+\cos\theta_\beta)
    -
    \cos\theta_\beta
    \sqrt{1-(\cos\theta_\alpha+\cos\theta_\beta)^2}
},
\end{equation}
while on the branch \(\sin\theta_\gamma\ge0\),
\begin{equation}
\label{eq:kgamma-++--plus}
|k_\gamma|
=
\frac{
    \sin(\theta_\beta-\theta_\alpha)
}{
    -\sin\theta_\beta(\cos\theta_\alpha+\cos\theta_\beta)
    +
    \cos\theta_\beta
    \sqrt{1-(\cos\theta_\alpha+\cos\theta_\beta)^2}
}.
\end{equation}
The \(\theta_\beta\)-intervals are the corresponding branch intervals listed
in Proposition~\ref{prop:param-++-}.  The endpoints are understood in the
limiting sense; in particular,
\[
    \theta_\beta=\pi+\theta_\alpha
\]
corresponds to the collapsed endpoint \(k_\gamma=0\).
\begin{figure}[htbp]
    \centering
    \includegraphics[width=\textwidth]{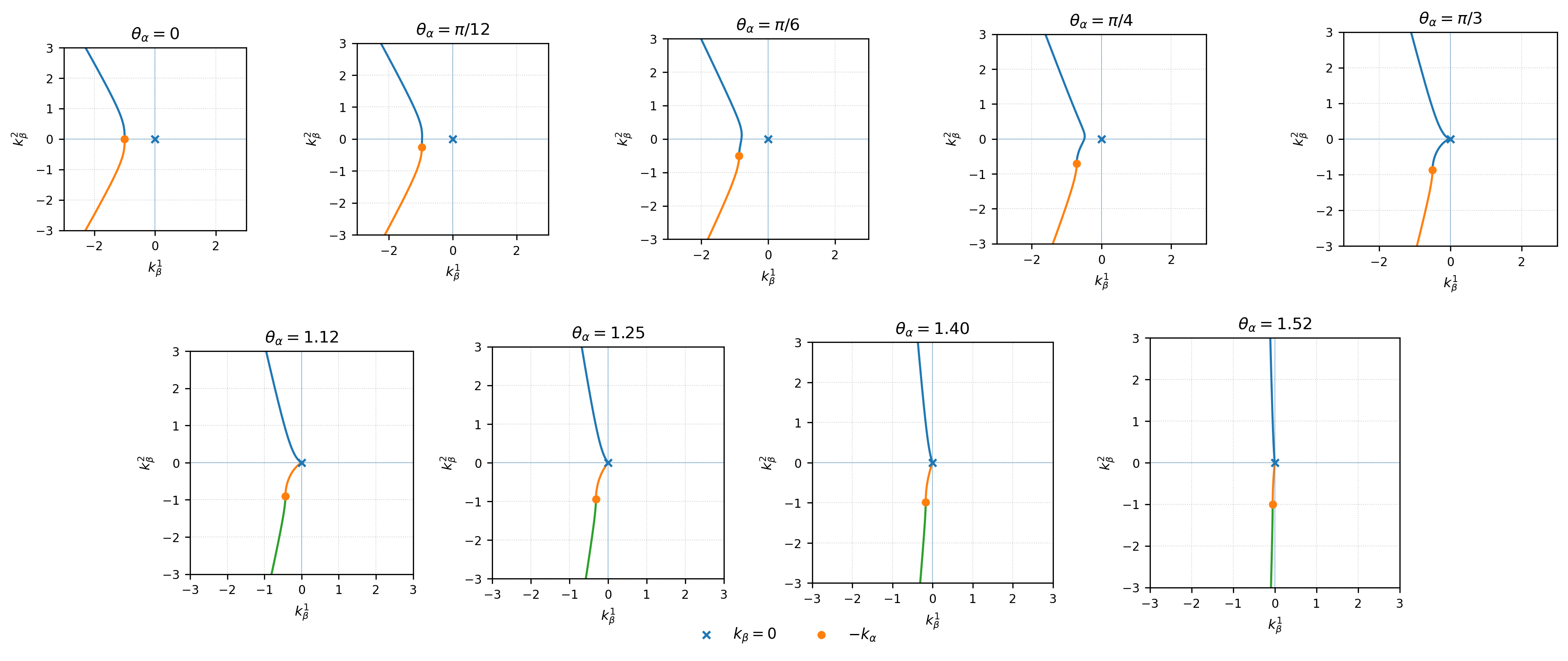}
    \caption{
        Unit-scale \((+,+,-)\) resonance curves in the \(k_\beta\)-plane.
        Top row: \(0\le\theta_\alpha\le\pi/3\).  For
        \(0\le\theta_\alpha<\pi/3\), the curve does not reach \(k_\beta=0\);
        at \(\theta_\alpha=\pi/3\), the limiting cusp direction is
        \(\theta_\beta=\pi\).  Bottom row:
        \(\pi/3<\theta_\alpha<\pi/2\), where the curve reaches
        \(k_\beta=0\) and develops a cusp.  The two non-compact ends
        correspond to \(\theta_\beta=\theta_-^\infty\) and
        \(\theta_\beta=\theta_+^\infty\).  In the cusp regime the two
        limiting approach directions are
        \(
            \theta_\beta^\ast
            =
            \pi-\arccos(2\cos\theta_\alpha),
            \qquad
            \widetilde{\theta}_\beta^\ast
            =
            \pi+\arccos(2\cos\theta_\alpha).
        \)
        The cross marks \(k_\beta=0\), and the dot marks
        \(k_\beta=-k_\alpha\), equivalently \(k_\gamma=0\).
    }
    \label{fig:manifold-++-}
\end{figure}

\begin{remark}[Cusp and asymptotic ends]
The cusp condition comes from
\[
    |k_\beta|=0.
\]
Using the \(\sin\theta_\gamma\le0\) branch formula, this is equivalent to
\[
    \cos\theta_\beta=-2\cos\theta_\alpha.
\]
This equation has real solutions precisely when
\[
    \cos\theta_\alpha\le\frac12,
    \qquad
    \text{i.e.}
    \qquad
    \theta_\alpha\ge\frac{\pi}{3}.
\]
At the threshold \(\theta_\alpha=\pi/3\), the two limiting cusp directions
coalesce at
\[
    \theta_\beta=\pi.
\]
For \(\theta_\alpha>\pi/3\), the two limiting approach directions are
\[
    \theta_\beta^\ast
    =
    \pi-\arccos(2\cos\theta_\alpha),
    \qquad
    \widetilde{\theta}_\beta^\ast
    =
    \pi+\arccos(2\cos\theta_\alpha),
\]
and both correspond to the same physical cusp \(k_\beta=0\).

The asymptotic angles
\[
    \theta_-^\infty
    =
    \pi-\arccos\left(\frac{\cos\theta_\alpha}{2}\right),
    \qquad
    \theta_+^\infty
    =
    \pi+\arccos\left(\frac{\cos\theta_\alpha}{2}\right)
\]
come from the zeros of the two denominators in
\eqref{eq:param-++--case1}--\eqref{eq:param-++--case2}.  They correspond to
the two non-compact ends of the \((+,+,-)\) resonance curve.
\end{remark}

\subsection{Co--area factor for the \((+,+,-)\) curve}
\label{subsec:coarea-factor-++-}

We now compute the one-dimensional co--area factor associated with the
unit-scale \((+,+,-)\) parametrization.  The identities below are understood
on the regular open arcs of Proposition~\ref{prop:param-++-}; in particular,
they are used away from the cusp \(k_\beta=0\), the collapsed endpoint
\(k_\gamma=0\), and the two non-compact ends.  Near such limiting points, the
formulas are applied on punctured neighborhoods and the estimates are obtained
by taking limits.
Throughout this subsection we use the
notation
\[
    p:=\cos\theta_\alpha+\cos\theta_\beta,
    \qquad
    r:=\sqrt{1-p^2},
\]
and
\[
    \widetilde f
    :=
    -\sin\theta_\beta\,p-\cos\theta_\beta\,r,
    \qquad
    f
    :=
    -\sin\theta_\beta\,p+\cos\theta_\beta\,r .
\]
We also set
\[
    \widetilde F_1(\theta_\alpha,\theta_\beta)
    :=
    \sin\theta_\alpha\,p+\cos\theta_\alpha\,r,
\]
and
\[
    F_1(\theta_\alpha,\theta_\beta)
    :=
    \sin\theta_\alpha\,p-\cos\theta_\alpha\,r.
\]
Thus, on the branch \(\sin\theta_\gamma\le0\),
\[
    |k_\beta|
    =
    \frac{\widetilde F_1}{\widetilde f},
    \qquad
    |k_\gamma|
    =
    \frac{\sin(\theta_\beta-\theta_\alpha)}{\widetilde f},
\]
while on the branch \(\sin\theta_\gamma\ge0\),
\[
    |k_\beta|
    =
    \frac{F_1}{f},
    \qquad
    |k_\gamma|
    =
    \frac{\sin(\theta_\beta-\theta_\alpha)}{f}.
\]

After eliminating
\[
    k_\gamma=-k_\alpha-k_\beta,
\]
the reduced frequency constraint for the \((+,+,-)\) configuration is
\[
    \widetilde\Xi_1(k_\beta)
    =
    \frac{k_\alpha^1}{|k_\alpha|}
    +
    \frac{k_\beta^1}{|k_\beta|}
    -
    \frac{k_\gamma^1}{|k_\gamma|}.
\]

\begin{lemma}[Co--area ratio for the \((+,+,-)\) curve]
\label{lem:coarea-ratio-++-}

On each regular open branch of the unit-scale \((+,+,-)\) resonance curve in
Proposition~\ref{prop:param-++-}, one has

\[
\frac{
    \left|\dfrac{d k_\beta^1}{d\theta_\beta}\right|
}{
    \left|\partial_{k_\beta^2}\widetilde\Xi_1\right|
}
=
\begin{cases}
\displaystyle
\left|
\frac{
    \bigl(\sin\theta_\alpha p+\cos\theta_\alpha r\bigr)
    \bigl(
        \cos\theta_\alpha\sin\theta_\beta
        -
        \sin\theta_\alpha\cos\theta_\beta
    \bigr)
}{
    r\,\widetilde f^{\,3}
}
\right|,
&
\sin\theta_\gamma\le0,
\\[1.4em]
\displaystyle
\left|
\frac{
    \bigl(\sin\theta_\alpha p-\cos\theta_\alpha r\bigr)
    \bigl(
        \cos\theta_\alpha\sin\theta_\beta
        -
        \sin\theta_\alpha\cos\theta_\beta
    \bigr)
}{
    r\,f^{3}
}
\right|,
&
\sin\theta_\gamma\ge0 .
\end{cases}
\]
\end{lemma}

\begin{proof}
We only record the identities used in the computation.  A direct
differentiation of the parametrization gives
\[
   \frac{d k_\beta^1}{d\theta_\beta}
   =
   \begin{cases}
   \displaystyle
   \frac{J_2}{r\,\widetilde f^{\,2}},
   & \sin\theta_\gamma\le0,
   \\[1em]
   \displaystyle
   \frac{J_1}{r\,f^2},
   & \sin\theta_\gamma\ge0,
   \end{cases}
\]
where
\[
\begin{aligned}
J_1
&:=
\sin\theta_\beta\cos\theta_\beta
\bigl(
    \cos\theta_\alpha\sin\theta_\beta
    -
    \sin\theta_\alpha\cos\theta_\beta
\bigr)
+
pr
\bigl(
    \sin\theta_\alpha p-\cos\theta_\alpha r
\bigr),
\\
J_2
&:=
-\sin\theta_\beta\cos\theta_\beta
\bigl(
    \cos\theta_\alpha\sin\theta_\beta
    -
    \sin\theta_\alpha\cos\theta_\beta
\bigr)
+
pr
\bigl(
    \sin\theta_\alpha p+\cos\theta_\alpha r
\bigr).
\end{aligned}
\]
On the other hand, differentiating the reduced constraint gives
\[
\partial_{k_\beta^2}\widetilde\Xi_1
=
\begin{cases}
\displaystyle
-\widetilde f
\frac{
    J_2
}{
    \bigl(\sin\theta_\alpha p+\cos\theta_\alpha r\bigr)
    \bigl(
        \cos\theta_\alpha\sin\theta_\beta
        -
        \sin\theta_\alpha\cos\theta_\beta
    \bigr)
},
& \sin\theta_\gamma\le0,
\\[1.4em]
\displaystyle
f
\frac{
    J_1
}{
    \bigl(\sin\theta_\alpha p-\cos\theta_\alpha r\bigr)
    \bigl(
        \cos\theta_\alpha\sin\theta_\beta
        -
        \sin\theta_\alpha\cos\theta_\beta
    \bigr)
},
& \sin\theta_\gamma\ge0.
\end{cases}
\]
Dividing the two identities yields the stated formula.
\end{proof}

\begin{lemma}[Lower bounds for \(r\)]
\label{lem:r-lower-++-}
Let
\[
    r=\sqrt{1-(\cos\theta_\alpha+\cos\theta_\beta)^2}.
\]
On all regular branch intervals appearing in Proposition~\ref{prop:param-++-},
the factor \(r\) is bounded away from zero.  Consequently,
\[
    \frac1r\lesssim 1
\]
uniformly on the \((+,+,-)\) branches used below.
\end{lemma}

\begin{proof}
Since
\[
    r^2=1-(\cos\theta_\alpha+\cos\theta_\beta)^2,
\]
the possible minima on the relevant intervals occur at endpoints or at the
critical point \(\theta_\beta=\pi\).  Evaluating these finitely many
possibilities in Case~\ref{case1} and Case~\ref{case2} gives a positive lower
bound:
\[
    r\ge \frac{\sqrt3}{2}
\]
on the branches in Case~\ref{case1} and on the
\(\sin\theta_\gamma\le0\) branch in Case~\ref{case2}, while
\[
    r\ge \frac{\sqrt{15}}{4}
\]
on the \(\sin\theta_\gamma\ge0\) branch in Case~\ref{case2}.  This proves the
claim.
\end{proof}

\subsection{Control of \(f,\widetilde{f}\) near the asymptotic ends}
\label{subsec:asymptotic-ends-++-}

We next describe the behavior of the two denominators at the non-compact
ends.  Recall the asymptotic angles from Proposition~\ref{prop:param-++-}:
\[
    \theta_-^\infty
    =
    \pi-\arccos\left(\frac{\cos\theta_\alpha}{2}\right),
    \qquad
    \theta_+^\infty
    =
    \pi+\arccos\left(\frac{\cos\theta_\alpha}{2}\right).
\]
For convenience, define
\[
    \theta_-^0
    :=
    \pi-\arccos\left(
        \frac{\cos\theta_\alpha}{2}
        +
        \sqrt{\frac12-\frac{\cos^2\theta_\alpha}{4}}
    \right),
\]
and
\[
    \theta_+^0
    :=
    \pi+\arccos\left(
        \frac{\cos\theta_\alpha}{2}
        +
        \sqrt{\frac12-\frac{\cos^2\theta_\alpha}{4}}
    \right).
\]


\begin{lemma}[Right-hand linear asymptote]
\label{Right-hand linear asymptote}
For every \(0\le\theta_\alpha<\pi/2\), and every
\[
    \theta_\beta\in[\theta_-^\infty,\theta_-^0],
\]
one has
\[
    \frac{3}{\pi}
    \bigl(\theta_\beta-\theta_-^\infty\bigr)
    \le
    \widetilde f(\theta_\alpha,\theta_\beta)
    \le
    2\bigl(\theta_\beta-\theta_-^\infty\bigr).
\]
In particular,
\[
    \widetilde f(\theta_\alpha,\theta_\beta)
    \sim
    \theta_\beta-\theta_-^\infty,
\]
with constants independent of \(\theta_\alpha\) and of the cut-off scale
\(\varepsilon\).
\end{lemma}

\begin{proof}
Recall
\[
    \widetilde f
    =
    -\sin\theta_\beta\,p-\cos\theta_\beta\,r,
    \qquad
    p=\cos\theta_\alpha+\cos\theta_\beta,
    \qquad
    r=\sqrt{1-p^2}.
\]
At
\[
    \theta_-^\infty
    =
    \pi-\arccos\left(\frac{\cos\theta_\alpha}{2}\right),
\]
we have
\[
    \widetilde f(\theta_\alpha,\theta_-^\infty)=0.
\]
A direct computation gives
\[
    \partial_{\theta_\beta}\widetilde f
    =
    \frac{
        (r+\sin\theta_\beta)
        (-\cos\theta_\beta p+\sin\theta_\beta r)
    }{r},
    \qquad
    \partial_{\theta_\beta}\widetilde f
    (\theta_\alpha,\theta_-^\infty)=2.
\]
Moreover,
\[
    \partial_{\theta_\beta}^2\widetilde f
    =
    \frac{
        r^2\cos\theta_\beta(-\cos\theta_\beta p+\sin\theta_\beta r)
        -2r^2\sin\theta_\beta\,\widetilde f
        -r^3\widetilde f
        +\sin^2\theta_\beta\cos\theta_\beta
    }{
        r^3
    }
    \le0
\]
on \([\theta_-^\infty,\theta_-^0]\).  Thus \(\widetilde f\) is increasing and
concave on this interval.

The upper bound follows from the endpoint derivative:
\[
    \widetilde f(\theta_\alpha,\theta_\beta)
    \le
    2(\theta_\beta-\theta_-^\infty).
\]
For the lower bound, concavity implies that \(\widetilde f\) lies above the
chord joining the endpoints.  Since
\[
    \widetilde f(\theta_\alpha,\theta_-^0)\ge1,
    \qquad
    \theta_-^0-\theta_-^\infty\le\frac{\pi}{3},
\]
the chord slope is at least \(3/\pi\).  Hence
\[
    \widetilde f(\theta_\alpha,\theta_\beta)
    \ge
    \frac{3}{\pi}
    (\theta_\beta-\theta_-^\infty).
\]
\end{proof}

\begin{lemma}[Left-hand linear asymptote]
\label{Left-hand linear asymptote}
For every \(0\le\theta_\alpha<\pi/2\), and every
\[
    \theta_\beta\in[\theta_+^0,\theta_+^\infty],
\]
one has
\[
    \frac{3}{\pi}
    \bigl(\theta_+^\infty-\theta_\beta\bigr)
    \le
    |f(\theta_\alpha,\theta_\beta)|
    \le
    2\bigl(\theta_+^\infty-\theta_\beta\bigr).
\]
In particular,
\[
    |f(\theta_\alpha,\theta_\beta)|
    \sim
    \theta_+^\infty-\theta_\beta,
\]
with constants independent of \(\theta_\alpha\) and of the cut-off scale
\(\varepsilon\).
\end{lemma}

\begin{proof}
The proof is the mirror analogue of Lemma~\ref{Right-hand linear asymptote}.
Indeed,
\[
    f(\theta_\alpha,\theta_+^\infty)=0,
\]
and direct differentiation gives
\[
    \partial_{\theta_\beta}f
    =
    -\frac{
        (r-\sin\theta_\beta)
        (\cos\theta_\beta p+\sin\theta_\beta r)
    }{r},
    \qquad
    \partial_{\theta_\beta}f(\theta_\alpha,\theta_+^\infty)=2,
\]
as well as
\[
    \partial_{\theta_\beta}^2 f
    =
    \frac{
        r^2\cos\theta_\beta(\cos\theta_\beta p+\sin\theta_\beta r)
        +2r^2\sin\theta_\beta f
        -r^3 f
        -\sin^2\theta_\beta\cos\theta_\beta
    }{
        r^3
    }
    \ge0
\]
on \([\theta_+^0,\theta_+^\infty]\).  Hence \(f\) is increasing and convex,
with \(f\le0\), so \(-f=|f|\) is concave.  The endpoint derivative gives
\[
    |f(\theta_\alpha,\theta_\beta)|
    \le
    2(\theta_+^\infty-\theta_\beta).
\]
The lower bound follows from the same chord argument as before, using
\[
    |f(\theta_\alpha,\theta_+^0)|\ge1,
    \qquad
    \theta_+^\infty-\theta_+^0\le\frac{\pi}{3}.
\]
Thus
\[
    |f(\theta_\alpha,\theta_\beta)|
    \ge
    \frac{3}{\pi}
    (\theta_+^\infty-\theta_\beta).
\]
\end{proof}
\subsection{Control of \(\widetilde{F_1}\) near the cusp}
\label{subsec:cusp-++-}

We now record the estimates near and away from the cusp.  The cusp appears
only in the regime
\[
    \frac{\pi}{3}<\theta_\alpha<\frac{\pi}{2}.
\]
As in Proposition~\ref{prop:param-++-}, we write
\[
    \theta_\beta^\ast
    :=
    \pi-\arccos(2\cos\theta_\alpha),
    \qquad
    \widetilde{\theta}_\beta^\ast
    :=
    \pi+\arccos(2\cos\theta_\alpha).
\]
These two angles correspond to the two limiting approach directions to the
cusp \(k_\beta=0\).

\begin{lemma}[No-cusp lower bound]
\label{I_{1,2}without cusp}
Fix \(0<\varepsilon\ll1\), and assume
\[
    \frac12+\varepsilon
    \le
    \cos\theta_\alpha
    \le
    1,\qquad \theta_\alpha\in[0,\frac \pi 2).
\]
Then, for
\[
\theta_\beta\in
\left[
 \theta_-^0,
\,
\pi+\theta_\alpha
\right],
\]
one has
\[
    \varepsilon
    \lesssim
    \widetilde F_1(\theta_\alpha,\theta_\beta)
    \le 1,
\]
\[
    \frac12\le \widetilde f(\theta_\alpha,\theta_\beta)\le1,
\]
and therefore
\[
    \varepsilon
    \lesssim
    |k_\beta|
    =
    \frac{\widetilde F_1}{\widetilde f}
    \lesssim
    1.
\]
\end{lemma}

\begin{proof}
For fixed \(\theta_\alpha\) in the stated range, a direct derivative check
shows that \(\widetilde F_1(\theta_\alpha,\theta_\beta)\) is positive and decreasing on
\[
\left[
 \theta_-^0,
\pi
\right]
\]

and increasing on
\[
    [\pi,\pi+\theta_\alpha].
\]
Thus its minimum on the stated interval occurs at
\[
    \theta_\beta=\pi.
\]
At this point,
\[
\widetilde F_1(\theta_\alpha,\pi)
=
\sin\theta_\alpha(\cos\theta_\alpha-1)
+
\cos\theta_\alpha
\sqrt{2\cos\theta_\alpha-\cos^2\theta_\alpha}.
\]
Using
\[
    \frac12+\varepsilon\le\cos\theta_\alpha\le1,
\]
one obtains
\[
    \widetilde F_1(\theta_\alpha,\pi)
    \gtrsim \varepsilon.
\]
The upper bound \(\widetilde F_1\le1\) is immediate from the definition.

Similarly, on the same interval, \(\widetilde f\) stays away from zero.  More
precisely,
\[
    \frac12\le \cos\theta_\alpha\le \widetilde f\le1.
\]
Since
\[
    |k_\beta|=\frac{\widetilde F_1}{\widetilde f},
\]
the claimed bound for \(|k_\beta|\) follows.
\end{proof}

\begin{lemma}[Linear behavior near the first zero]
\label{lem:first-zero}
Assume
\[
    \varepsilon
    \le
    \cos\theta_\alpha
    \le
    \frac12-\varepsilon,\qquad \theta_\alpha\in[0,\frac \pi 2),
\]
for some sufficiently small \(\varepsilon>0\).  Then there exists
\(\delta(\varepsilon)>0\) such that, whenever
\[
    |\theta_\beta-\theta_\beta^\ast(\theta_\alpha)|
    \le
    \delta(\varepsilon),
\]
one has
\[
    \sqrt{\varepsilon}\,
    |\theta_\beta-\theta_\beta^\ast(\theta_\alpha)|
    \lesssim
    |\widetilde F_1(\theta_\alpha,\theta_\beta)|
    \lesssim
    |\theta_\beta-\theta_\beta^\ast(\theta_\alpha)|.
\]
\end{lemma}

\begin{proof}
Differentiating in \(\theta_\beta\), we get
\[
\partial_{\theta_\beta}\widetilde F_1
=
-\sin\theta_\beta
\left[
    \sin\theta_\alpha
    -
    \frac{
        \cos\theta_\alpha(\cos\theta_\alpha+\cos\theta_\beta)
    }{
        \sqrt{1-(\cos\theta_\alpha+\cos\theta_\beta)^2}
    }
\right].
\]
At the first zero
\[
    \theta_\beta^\ast
    =
    \pi-\arccos(2\cos\theta_\alpha),
\]
we have
\[
    \left|
    \partial_{\theta_\beta}\widetilde F_1
    (\theta_\alpha,\theta_\beta^\ast)
    \right|
    =
    \frac{\sqrt{1-4\cos^2\theta_\alpha}}{\sin\theta_\alpha}.
\]
On the range
\[
    \varepsilon
    \le
    \cos\theta_\alpha
    \le
    \frac12-\varepsilon,
\]
this derivative is bounded above by an absolute constant and bounded below by
a constant comparable to \(\sqrt\varepsilon\).  By uniform continuity of
\(\partial_{\theta_\beta}\widetilde F_1\), after choosing
\(\delta(\varepsilon)>0\) sufficiently small, the same bounds hold in a
\(\delta(\varepsilon)\)-neighborhood of \(\theta_\beta^\ast\).  The mean value
theorem gives the claim.
\end{proof}

\begin{lemma}[Linear behavior near the second zero]
\label{lem:second-zero}
Assume
\[
    \varepsilon
    \le
    \cos\theta_\alpha
    \le
    \frac12-\varepsilon,\qquad \theta_\alpha\in[0,\frac \pi 2).
\]
Then there exists \(\delta(\varepsilon)>0\) such that, whenever
\[
    |\theta_\beta-\widetilde{\theta}_\beta^\ast(\theta_\alpha)|
    \le
    \delta(\varepsilon),
\]
one has
\[
    \sqrt{\varepsilon}\,
    |\theta_\beta-\widetilde{\theta}_\beta^\ast(\theta_\alpha)|
    \lesssim
    |\widetilde F_1(\theta_\alpha,\theta_\beta)|
    \lesssim
    |\theta_\beta-\widetilde{\theta}_\beta^\ast(\theta_\alpha)|.
\]
\end{lemma}

\begin{proof}
The proof is identical to the proof of Lemma~\ref{lem:first-zero}, replacing
\(\theta_\beta^\ast\) by
\(\widetilde{\theta}_\beta^\ast\).  
\end{proof}


\subsection{A priori boundedness estimate for the \((+,+,-)\) contribution}
\label{subsec:apriori-++-}

We now return from the unit-scale normalization to general
\(k_\alpha\neq0\).  Write
\[
    k_\alpha=|k_\alpha|\varsigma_\alpha,
    \qquad
    k_\beta=|k_\alpha|\varsigma_\beta,
    \qquad
    k_\gamma=|k_\alpha|\varsigma_\gamma,
    \qquad
    |\varsigma_\alpha|=1.
\]
The unit-scale variables \(\varsigma_\beta,\varsigma_\gamma\) are given by
Proposition~\ref{prop:param-++-}.  By the scaled co-area representation
\eqref{eq:parametrized-collision}, the \((+,+,-)\) contribution carries the
overall factor \(|k_\alpha|^4\).

We denote by
\[
    \mathcal C_{++-}(n)(k_\alpha)
\]
the contribution to the collision operator with
\[
    \sigma_\alpha=\sigma_\beta=+1,
    \qquad
    \sigma_\gamma=-1.
\]

\begin{proposition}[A priori bound for the \((+,+,-)\) contribution]
\label{prop:apriori-++-}
Let \(m>4\), and fix the angular cut-off scale \(0<\varepsilon\ll1\).
Let \(\delta(\varepsilon)>0\) be the cusp scale fixed in
Lemmas~\ref{lem:first-zero} and~\ref{lem:second-zero}.  Then, for every
\(n\in L^\infty_m\),
\[
    \sup_{k_\alpha\in\mathbb R^2}
    \langle k_\alpha\rangle^m
    \left|
        \mathcal C_{++-}(n)(k_\alpha)
    \right|
    \lesssim_m
    \mathfrak C_{++-}(\varepsilon,\delta,m)
    \|n\|_{L^\infty_m}^2,
\]
where one may take
\[
    \mathfrak C_{++-}(\varepsilon,\delta,m)
    :=
    \varepsilon^{-2m}
    +
    \varepsilon^{-m}
    +
    \varepsilon^{-\frac{3m}{2}}\delta(\varepsilon)^{-m}
    +
    \varepsilon^{-\frac{3m}{2}-3}\delta(\varepsilon)^{-m+4}
    +
    \varepsilon^{-\frac{3m}{2}-3}\delta(\varepsilon)^{-m}.
\]
In particular, for fixed \(0<\varepsilon\ll1\),
\[
    \mathcal C_{++-}:L^\infty_m\to L^\infty_m
\]
for every \(m>4\).
\end{proposition}
Before proving Proposition~\ref{prop:apriori-++-}, we record two auxiliary
facts used in the estimates.

For the upper half of the manifold, the positivity of $|\varsigma_\gamma|$ follows directly from Proposition~\ref{Positivity of gamma}. By the same argument, applied symmetrically, we obtain the corresponding positivity on the lower half of the manifold.

\begin{lemma}[Second positivity of \(|\varsigma_\gamma|\)]
\label{SecondPositivity of gamma}
Let \(0<\varepsilon\ll1\).  On the branch
\[
    \theta_\beta
    \in
    \left[
        \pi+\theta_\alpha,\,
         \theta_+^\infty
    \right),
\]
assume
\[
    |\cos\theta_\alpha|>\varepsilon,
    \qquad
    |\cos\theta_\gamma|>\varepsilon .
\]
Then
\[
    |\varsigma_\gamma|\gtrsim\varepsilon.
\]
More precisely,
\[
  \inf_{\substack{|\cos\theta_\alpha|>\varepsilon\\
                  |\cos\theta_\gamma|>\varepsilon}}
  |\varsigma_\gamma|
  \ge
  \inf_{\substack{|\cos\theta_\alpha|>\varepsilon\\
                  |\cos\theta_\gamma|>\varepsilon}}
  |\sin(\theta_\beta-\theta_\alpha)|
  \gtrsim \varepsilon .
\]
\end{lemma}

\begin{lemma}[cusp cancellation]
\label{lem:cusp-cancellation-++-}
In the cusp regime
\[
    \varepsilon\le \cos\theta_\alpha\le \frac12-\varepsilon,
\]
one has, near the first cusp direction,
\[
    1+|\varsigma_\beta|-|\varsigma_\gamma|
    \lesssim
    |\theta_\beta-\theta_\beta^\ast|,
\]
and, near the second cusp direction,
\[
    1+|\varsigma_\beta|-|\varsigma_\gamma|
    \lesssim
    |\theta_\beta-\widetilde\theta_\beta^\ast|.
\]
\end{lemma}

\begin{proof}
We prove the estimate near \(\theta_\beta^\ast\).  The proof near
\(\widetilde\theta_\beta^\ast\) is identical, using
Lemma~\ref{lem:second-zero} instead of Lemma~\ref{lem:first-zero}.

At unit scale, momentum conservation gives
\[
    \varsigma_\alpha+\varsigma_\beta+\varsigma_\gamma=0,
    \qquad
    |\varsigma_\alpha|=1.
\]
Hence
\[
    |\varsigma_\gamma|
    =
    |\varsigma_\alpha+\varsigma_\beta|.
\]
By the triangle inequality,
\[
    |\varsigma_\alpha+\varsigma_\beta|
    \ge
    |\varsigma_\alpha|-|\varsigma_\beta|
    =
    1-|\varsigma_\beta|.
\]
Therefore
\[
    1+|\varsigma_\beta|-|\varsigma_\gamma|
    \le
    2|\varsigma_\beta|.
\]

On the branch \(\sin\theta_\gamma\le0\),
\[
    |\varsigma_\beta|
    =
    \frac{\widetilde F_1}{\widetilde f}.
\]
In a fixed neighborhood of the cusp direction \(\theta_\beta^\ast\), the
denominator \(\widetilde f\) is bounded above and below away from zero.  Hence
\[
    |\varsigma_\beta|
    \lesssim
    |\widetilde F_1|.
\]
By Lemma~\ref{lem:first-zero},
\[
    |\widetilde F_1(\theta_\alpha,\theta_\beta)|
    \lesssim
    |\theta_\beta-\theta_\beta^\ast|.
\]
Thus
\[
    1+|\varsigma_\beta|-|\varsigma_\gamma|
    \lesssim
    |\theta_\beta-\theta_\beta^\ast|.
\]
\end{proof}
We now prove Proposition~\ref{prop:apriori-++-}. 
We use the scaled co--area formula as before and write
\[
    \mathcal C_{++-}(n)(k_\alpha)=I_1+I_2,
\]
where \(I_1\) is the contribution from the branch
\[
    \sin\theta_\gamma\le0,
\]
and \(I_2\) is the contribution from the branch
\[
    \sin\theta_\gamma\ge0.
\]
The corresponding co--area factors are given by
Lemma~\ref{lem:coarea-ratio-++-}.  On the \(I_1\)-branch,
\[
    |k_\beta|
    =
    |k_\alpha|\frac{\widetilde F_1}{\widetilde f},
    \qquad
    |k_\gamma|
    =
    |k_\alpha|
    \frac{\sin(\theta_\beta-\theta_\alpha)}{\widetilde f},
\]
whereas on the \(I_2\)-branch,
\[
    |k_\beta|
    =
    |k_\alpha|\frac{F_1}{f},
    \qquad
    |k_\gamma|
    =
    |k_\alpha|
    \frac{\sin(\theta_\beta-\theta_\alpha)}{f}.
\]

For notational convenience, we write
\[
    I_1
    =
    |k_\alpha|^4
    \int_{\{\sin\theta_\gamma\le0\}}
        \mathfrak k^1_{++-}(\theta_\beta)
    \,d\theta_\beta,
    \qquad
    I_2
    =
    |k_\alpha|^4
    \int_{\{\sin\theta_\gamma\ge0\}}
        \mathfrak k^2_{++-}(\theta_\beta)
    \,d\theta_\beta.
\]
Here \(\mathfrak k^1_{++-}\) and \(\mathfrak k^2_{++-}\) denote the full
unit-scale integrands on the two branches.  More precisely,
\[
    \mathfrak k^1_{++-}
    =
    \widetilde I_{++-}(\theta_\beta)
    \left|
    \frac{
        \widetilde F_1
        \bigl(
            \cos\theta_\alpha\sin\theta_\beta
            -
            \sin\theta_\alpha\cos\theta_\beta
        \bigr)
    }{
        r\,\widetilde f^{\,3}
    }
    \right|,
\]
and
\[
    \mathfrak k^2_{++-}
    =
    \widetilde I_{++-}(\theta_\beta)
    \left|
    \frac{
        F_1
        \bigl(
            \cos\theta_\alpha\sin\theta_\beta
            -
            \sin\theta_\alpha\cos\theta_\beta
        \bigr)
    }{
        r\,f^{3}
    }
    \right|.
\]
The factor \(\widetilde I_{++-}\) is the unit-scale collision integrand,
including the angular cut-off, the interaction coefficient, and the quadratic
products of \(n\).

The proof is divided into three cases.


\subsubsection{Case 1: the no-cusp regime}
\label{subsubsec:++-case1-nocusp}

Assume
\[
    \cos\theta_\alpha\ge \frac12+\varepsilon .
\]
In this regime, the resonance curve does not pass through \(k_\beta=0\).  We
decompose
\[
    I_1=I_{1,1}+I_{1,2},
\]
where
\[
    I_{1,1}
    =
    |k_\alpha|^4
    \int_{\theta_-^\infty}^{\theta_-^0}
        \mathfrak k^1_{++-}(\theta_\beta)
    \,d\theta_\beta,
\]
and
\[
    I_{1,2}
    =
    |k_\alpha|^4
    \int_{\theta_-^0}^{\pi+\theta_\alpha}
        \mathfrak k^1_{++-}(\theta_\beta)
    \,d\theta_\beta.
\]
Here \(\mathfrak k^1_{++-}\) denotes the full integrand on the branch
\(\sin\theta_\gamma\le0\), including the angular cut-off, the unit-scale
interaction coefficient, the quadratic products of \(n\), and the co-area
factor.

We first estimate \(I_{1,1}\).  On this interval,
\[
    \theta_\beta\in[\theta_-^\infty,\theta_-^0],
\]
Lemma~\ref{Right-hand linear asymptote} gives
\[
    \widetilde f
    \sim
    \theta_\beta-\theta_-^\infty.
\]
Moreover, \(\widetilde F_1(\theta_\alpha,\cdot)\) is decreasing on this
interval, and Lemma~\ref{lem:psi_decreasing} in
Appendix~\ref{appendix:additional-results} gives
\[
    c_0
    \le
 \widetilde F_1(\theta_\alpha,\theta_-^0)
    \le
    \widetilde F_1(\theta_\alpha,\theta_\beta)
    \le
    1,
\]
where
\[
    c_0
    :=
    \frac{\sqrt3}{2}
    \left(\frac14-\frac{\sqrt7}{4}\right)
    +
    \frac12
    \sqrt{\frac12+\frac{\sqrt7}{8}}
    >0.
\]
Equivalently,
\[
    \widetilde F_1\sim1
\]
on \(I_{1,1}\), with an absolute implicit constant.  By the cut-off positivity
estimate, see Lemma~\ref{Positivity of gamma},
\[
    \varepsilon
    \lesssim
    \sin(\theta_\beta-\theta_\alpha)
    \le1.
\]
Therefore
\[
    |k_\beta|
    =
    |k_\alpha|
    \frac{\widetilde F_1}{\widetilde f}
    \sim
    \frac{|k_\alpha|}{\theta_\beta-\theta_-^\infty},
\]
and
\[
    |k_\gamma|
    =
    |k_\alpha|
    \frac{\sin(\theta_\beta-\theta_\alpha)}{\widetilde f}
    \gtrsim
    \varepsilon
    \frac{|k_\alpha|}{\theta_\beta-\theta_-^\infty}.
\]
Also, by momentum conservation and the triangle inequality,
\[
    1+|\varsigma_\beta|-|\varsigma_\gamma|
    \le
    1+|\varsigma_\beta+\varsigma_\gamma|
    =
    1+|\varsigma_\alpha|
    =
    2.
\]

We now split the estimate of \(I_{1,1}\) according to the size of
\(|k_\alpha|\).

First assume \(|k_\alpha|\le1\).  We split the integration interval into
\[
    (\theta_-^\infty,\theta_-^\infty+|k_\alpha|]
    \cup
    [\theta_-^\infty+|k_\alpha|,\theta_-^0].
\]
On the first interval,
\[
    0<\theta_\beta-\theta_-^\infty\le |k_\alpha|,
\]
we have
\[
    \langle k_\alpha\rangle^{-m}\sim1,
\qquad
    \langle k_\beta\rangle^{-m}
    \sim
    \frac{(\theta_\beta-\theta_-^\infty)^m}{|k_\alpha|^m},
\qquad
    \langle k_\gamma\rangle^{-m}
    \lesssim
    \varepsilon^{-m}
    \frac{(\theta_\beta-\theta_-^\infty)^m}{|k_\alpha|^m}.
\]
On the second interval,
\[
[\theta_-^\infty+|k_\alpha|,\theta_-^0],
 \]
we have for \(m\ge 0\),
\[
\langle k_\alpha\rangle^{-m}\sim 1,\qquad
\langle k_\beta\rangle^{-m}\sim 1,\qquad
\langle k_\gamma\rangle^{-m}\lesssim 1.
\]
Using the co-area factor
\[
    \frac1{\widetilde f^3}
    \lesssim
    \frac1{(\theta_\beta-\theta_-^\infty)^3},
\]
we obtain
\[
\begin{aligned}
I_{1,1}
&\lesssim
|k_\alpha|^4
\int_{\theta_-^\infty}^{\theta_-^\infty+|k_\alpha|}
\varepsilon^{-m}
\left(
    \frac{(\theta_\beta-\theta_-^\infty)^{m-3}}{|k_\alpha|^m}
    +
    \frac{(\theta_\beta-\theta_-^\infty)^{2m-3}}{|k_\alpha|^{2m}}
\right)
d\theta_\beta
\\
&\quad
+
|k_\alpha|^4
\int_{\theta_-^\infty+|k_\alpha|}^{\theta_-^0}
\frac{d\theta_\beta}
     {(\theta_\beta-\theta_-^\infty)^3}.
\end{aligned}
\]
Hence
\[
\begin{aligned}
I_{1,1}
&\lesssim
\varepsilon^{-m}
|k_\alpha|^{4-m}|k_\alpha|^{m-2}
+
\varepsilon^{-m}
|k_\alpha|^{4-2m}|k_\alpha|^{2m-2}
+
|k_\alpha|^4\frac{1}{|k_\alpha|^2}
\\
&\lesssim
\varepsilon^{-m}|k_\alpha|^2+|k_\alpha|^2
\lesssim
\varepsilon^{-m},
\end{aligned}
\]
provided
\[
    m-3>-1,
    \qquad
    2m-3>-1,
    \qquad
    0<\varepsilon\ll1.
\]
Thus, for \(|k_\alpha|\le1\),
\[
    \langle k_\alpha\rangle^m |I_{1,1}|
    \lesssim
    \varepsilon^{-m}\|n\|_{L^\infty_m}^2.
\]

Now assume \(|k_\alpha|>1\).  Since
\[
    \theta_-^0-\theta_-^\infty
    =
    \arccos\left(\frac{\cos\theta_\alpha}{2}\right)
    -
    \arccos\left(
        \frac{\cos\theta_\alpha}{2}
        +
        \sqrt{\frac12-\frac{\cos^2\theta_\alpha}{4}}
    \right)
    \lesssim1
    \le |k_\alpha|,
\]
we are always in the large-radius regime on \(I_{1,1}\).  Thus
\[
    \langle k_\alpha\rangle^{-m}\sim |k_\alpha|^{-m},
\qquad
    \langle k_\beta\rangle^{-m}
    \sim
    \frac{(\theta_\beta-\theta_-^\infty)^m}{|k_\alpha|^m},
\qquad
    \langle k_\gamma\rangle^{-m}
    \lesssim
    \varepsilon^{-m}
    \frac{(\theta_\beta-\theta_-^\infty)^m}{|k_\alpha|^m}.
\]
Setting
\[
    s:=\theta_\beta-\theta_-^\infty,
\]
we get
\[
\begin{aligned}
|k_\alpha|^m |I_{1,1}|
&\lesssim
|k_\alpha|^{m+4}
\int_{0\le s\lesssim1}
\varepsilon^{-m}
\left(
    \frac{s^{m-3}}{|k_\alpha|^{2m}}
    +
    \frac{s^{2m-3}}{|k_\alpha|^{2m}}
\right)
ds
\\
&\lesssim
\varepsilon^{-m}|k_\alpha|^{-m+4}
\lesssim
\varepsilon^{-m},
\end{aligned}
\]
where we used \(m\ge4, \quad0<\varepsilon\ll 1\).  Hence
\[
    \langle k_\alpha\rangle^m |I_{1,1}|
    \lesssim
    \varepsilon^{-m}\|n\|_{L^\infty_m}^2
\]
for all \(k_\alpha\).

We next estimate \(I_{1,2}\).  On
\[
    \theta_\beta\in[\theta_-^0,\pi+\theta_\alpha],
\]
Lemma~\ref{I_{1,2}without cusp} gives
\[
    \frac1{\widetilde f}\sim1,
    \qquad
    \varepsilon\lesssim|\varsigma_\beta|\lesssim1,
    \qquad
    \varepsilon\lesssim|\varsigma_\gamma|\lesssim1.
\]
Thus
\[
\begin{aligned}
\langle k_\alpha\rangle^m |I_{1,2}|
&\lesssim
\langle k_\alpha\rangle^m |k_\alpha|^4
\int_{\theta_-^0}^{\pi+\theta_\alpha}
\langle k_\alpha\rangle^{-2m}
\left(\varepsilon^{-m}+\varepsilon^{-2m}\right)
d\theta_\beta
\\
&\lesssim
\varepsilon^{-2m}
|k_\alpha|^4
\langle k_\alpha\rangle^{-m}
\|n\|_{L^\infty_m}^2
\\
&\lesssim
\varepsilon^{-2m}
\|n\|_{L^\infty_m}^2,
\end{aligned}
\]
provided \(m\ge4,\quad0<\varepsilon\ll 1\).

It remains to estimate \(I_2\).  On the \(I_2\)-branch,
\[
    \theta_\beta\in[\pi+\theta_\alpha,\theta_+^\infty],
\]
and Lemma~\ref{Left-hand linear asymptote} gives
\[
    |f|
    \sim
    \theta_+^\infty-\theta_\beta.
\]
Moreover,  $F_1(\theta_\alpha, \cdot)$ is negative and increasing on this interval.  By
Lemma~\ref{lem:positivity_decreasing_term} in
Appendix~\ref{appendix:additional-results},
\[
    c_1
    \le
    |F_1|
    \le
    \cos\theta_\alpha
    \le1,
\]
where
\[
    c_1
    :=
    \frac{-\sqrt3+\sqrt{15}}{8}
    >0.
\]
By Lemma~\ref{SecondPositivity of gamma},
\[
    \varepsilon
    \lesssim
    |\sin(\theta_\beta-\theta_\alpha)|
    \le1.
\]
Therefore
\[
    |k_\beta|
    =
    |k_\alpha|\frac{|F_1|}{|f|}
    \sim
    \frac{|k_\alpha|}{\theta_+^\infty-\theta_\beta},
\]
and
\[
    |k_\gamma|
    =
    |k_\alpha|
    \frac{|\sin(\theta_\beta-\theta_\alpha)|}{|f|}
    \gtrsim
    \varepsilon
    \frac{|k_\alpha|}{\theta_+^\infty-\theta_\beta}.
\]
Also, by the triangle inequality,
\[
    1+|\varsigma_\beta|-|\varsigma_\gamma|
    \le
    1+|\varsigma_\beta+\varsigma_\gamma|
    =
    1+|\varsigma_\alpha|
    =
    2.
\]

We first consider \(|k_\alpha|\le1\).  Set
\[
    s:=\theta_+^\infty-\theta_\beta.
\]
Splitting the interval at \(s=|k_\alpha|\), we have on \(0<s\le|k_\alpha|\),
\[
    \langle k_\alpha\rangle^{-m}\sim1,
\qquad
    \langle k_\beta\rangle^{-m}
    \sim
    \frac{s^m}{|k_\alpha|^m},
\qquad
    \langle k_\gamma\rangle^{-m}
    \lesssim
    \varepsilon^{-m}
    \frac{s^m}{|k_\alpha|^m}.
\]
On \({|k_\alpha|}<s\le {\theta_+^\infty-(\pi+\theta_\alpha)}\),
we have
\[ \langle k_\alpha\rangle^{-m}\sim1,
\qquad
 \langle k_\beta\rangle^{-m}\sim1,
 \qquad
  \langle k_\gamma\rangle^{-m}\lesssim1,\]
Using
\[
    \frac1{|f|^3}
    \lesssim
    \frac1{s^3},
\]
we obtain
\[
\begin{aligned}
I_2
&\lesssim
|k_\alpha|^4
\int_{0}^{|k_\alpha|}
\varepsilon^{-m}
\left(
    \frac{s^{m-3}}{|k_\alpha|^m}
    +
    \frac{s^{2m-3}}{|k_\alpha|^{2m}}
\right)
ds
\\
&\quad
+
|k_\alpha|^4
\int_{|k_\alpha|}^{\theta_+^\infty-(\pi+\theta_\alpha)}
\frac{ds}{s^3}.
\end{aligned}
\]
Hence
\[
\begin{aligned}
I_2
&\lesssim
\varepsilon^{-m}
|k_\alpha|^{4-m}|k_\alpha|^{m-2}
+
\varepsilon^{-m}
|k_\alpha|^{4-2m}|k_\alpha|^{2m-2}
+
|k_\alpha|^4\frac1{|k_\alpha|^2}
\\
&\lesssim
\varepsilon^{-m}|k_\alpha|^2+|k_\alpha|^2
\lesssim
\varepsilon^{-m},
\end{aligned}
\]
provided \(m>2,\quad 0<\varepsilon\ll 1\).
Thus, for \(|k_\alpha|\le1\),
\[
    \langle k_\alpha\rangle^m |I_2|
    \lesssim
    \varepsilon^{-m}
    \|n\|_{L^\infty_m}^2.
\]

Now assume \(|k_\alpha|>1\).  Since
\[
    \theta_+^\infty-(\pi+\theta_\alpha)
    \lesssim1
    \le |k_\alpha|,
\]
we are always in the large-radius regime.  Hence
\[
    \langle k_\alpha\rangle^{-m}\sim |k_\alpha|^{-m},
\qquad
    \langle k_\beta\rangle^{-m}
    \sim
    \frac{s^m}{|k_\alpha|^m},
\qquad
    \langle k_\gamma\rangle^{-m}
    \lesssim
    \varepsilon^{-m}
    \frac{s^m}{|k_\alpha|^m}.
\]
Therefore
\[
\begin{aligned}
|k_\alpha|^m |I_2|
&\lesssim
|k_\alpha|^{m+4}
\int_{0\le s\lesssim1}
\varepsilon^{-m}
\left(
    \frac{s^{m-3}}{|k_\alpha|^{2m}}
    +
    \frac{s^{2m-3}}{|k_\alpha|^{2m}}
\right)
ds
\\
&\lesssim
\varepsilon^{-m}|k_\alpha|^{-m+4}
\lesssim
\varepsilon^{-m},
\end{aligned}
\]
because \(m\ge4,\quad 0<\varepsilon\ll 1\).  Hence
\[
    \langle k_\alpha\rangle^m |I_2|
    \lesssim
    \varepsilon^{-m}
    \|n\|_{L^\infty_m}^2.
\]

Combining the estimates for \(I_{1,1}\), \(I_{1,2}\), and \(I_2\), we obtain
\[
    \langle k_\alpha\rangle^m
    |\mathcal C_{++-}(n)(k_\alpha)|
    \lesssim
    \varepsilon^{-2m}
    \|n\|_{L^\infty_m}^2.
\]

\subsubsection{Case 2: separated asymptotic and cusp regions}
\label{subsubsec:++-case2-separated}

Assume
\[
    \varepsilon\le \cos\theta_\alpha\le\frac12-\varepsilon,
\]
and suppose
\[
    \theta_-^0
    <
    \theta_\beta^\ast-\delta(\varepsilon).
\]
Under this condition we have 
\[
\frac{1}{\sqrt{5}} \;\le\; \cos\theta_\alpha 
\;\le\; \frac{1}{2}-\varepsilon .
\]
Thus the asymptotic region near \(\theta_-^\infty\) and the cusp region
near \(\theta_\beta^\ast\) are separated.

We decompose
\[
    I_1
    =
    I_{1,1}+I_{1,2}+I_{1,3}+I_{1,4}+I_{1,5},
\]
where
\[
    I_{1,1}
    =
    |k_\alpha|^4
    \int_{\theta_-^\infty}^{\theta_-^0}
        \mathfrak k^1_{++-}(\theta_\beta)
    \,d\theta_\beta,
\]
\[
    I_{1,2}
    =
    |k_\alpha|^4
    \int_{\theta_-^0}^{\theta_\beta^\ast-\delta(\varepsilon)}
        \mathfrak k^1_{++-}(\theta_\beta)
    \,d\theta_\beta,
\]
\[
    I_{1,3}
    =
    |k_\alpha|^4
    \int_{\theta_\beta^\ast-\delta(\varepsilon)}^{\theta_\beta^\ast}
        \mathfrak k^1_{++-}(\theta_\beta)
    \,d\theta_\beta,
\]
\[
    I_{1,4}
    =
    |k_\alpha|^4
    \int_{\widetilde\theta_\beta^\ast}^{\widetilde\theta_\beta^\ast+\delta(\varepsilon)}
        \mathfrak k^1_{++-}(\theta_\beta)
    \,d\theta_\beta,
\]
and
\[
    I_{1,5}
    =
    |k_\alpha|^4
    \int_{\widetilde\theta_\beta^\ast+\delta(\varepsilon)}^{\pi+\theta_\alpha}
        \mathfrak k^1_{++-}(\theta_\beta)
    \,d\theta_\beta.
\]

The estimate of \(I_{1,1}\) is the same asymptotic-end estimate as in
Case~\ref{subsubsec:++-case1-nocusp}, except that the numerator
\(\widetilde F_1\) is no longer uniformly bounded below by an absolute
constant.  Since the interval is separated from the first cusp by
\(\delta(\varepsilon)\), Lemma~\ref{lem:first-zero} gives
\[
    1\ge\widetilde F_1
    \gtrsim
    \sqrt{\varepsilon}\,\delta(\varepsilon).
\]
Together with Lemma~\ref{Right-hand linear asymptote} and the monotonicity of \(\sin(\cdot-\theta_\alpha)\) for fixed \(\theta_\alpha\) on the interval under consideration,
\[
    \widetilde f
    \sim
    \theta_\beta-\theta_-^\infty,
    \qquad
    \frac{\sqrt{19}+2}{10}\le \sin(\arccos(\frac{\cos\theta_\alpha}{2})+\theta_\alpha)\le
    \sin(\theta_\beta-\theta_\alpha)
    \le1.
\]
Hence the same one-dimensional estimate as in Case~1 yields
\[
    \langle k_\alpha\rangle^m |I_{1,1}|
    \lesssim
    \varepsilon^{-\frac{m}{2}}
    \delta(\varepsilon)^{-m}
    \|n\|_{L^\infty_m}^2,
\]
provided that \(m\ge4\) and \(0<\varepsilon\ll1\).

On \(I_{1,2}\), the curve is away from both the asymptotic end and the cusp.
Indeed,
\[
    1\ge \widetilde F_1
    \gtrsim
    \sqrt{\varepsilon}\,\delta(\varepsilon),
    \]
    and
    \[
    \frac{\sqrt{19}+2}{10}\le \sin(\theta_\alpha+\arccos(\frac{\cos\theta_\alpha}{2}))=\min\{\sin(\theta_\alpha+\arccos(2\cos\theta_\alpha)),\sin(\theta_\alpha+\arccos(\frac{\cos\theta_\alpha}{2}))\}\le \sin(\theta_\beta-\theta_\alpha)\le 1.
\]
By the monotonicity of \(\widetilde f\) on the relevant subinterval of
\((\theta_-^0,\theta_+^\infty)\), we have
\[
    \frac{\sqrt3}{2}
    \le
    \widetilde f(\theta_\alpha,\theta_\beta^\ast)
    \le
    \widetilde f(\theta_\alpha,\theta_\beta^\ast-\delta(\varepsilon))
    \le
    \widetilde f(\theta_\alpha,\theta_\beta)
    \le1.
\]
Thus \(\widetilde f\sim1\) on \(I_{1,2}\).
Therefore
\[
    |\varsigma_\beta|
    =
    \frac{\widetilde F_1}{\widetilde f}
    \gtrsim
    \sqrt{\varepsilon}\,\delta(\varepsilon),
    \qquad
    |\varsigma_\gamma|
    =
    \frac{\sin(\theta_\beta-\theta_\alpha)}{\widetilde f}
    \gtrsim1.
\]
The bounded-region estimate then gives
\[
    \langle k_\alpha\rangle^m |I_{1,2}|
    \lesssim
    \varepsilon^{-\frac m2}
    \delta(\varepsilon)^{-m}
    \|n\|_{L^\infty_m}^2,
\]
provided that \(m\ge4\)  and \(0<\varepsilon\ll1\).

We now estimate the first cusp piece \(I_{1,3}\).  Set
\[
    s:=\theta_\beta^\ast-\theta_\beta,
    \qquad
    0<s\le\delta(\varepsilon).
\]
By Lemma~\ref{lem:first-zero},
\[
    \sqrt{\varepsilon}\,s
    \lesssim
    |\widetilde F_1(\theta_\alpha,\theta_\beta)|
    \lesssim
    s.
\]
Moreover, on this cusp neighborhood,
\[
    \widetilde f\sim1,
    \qquad
    r\sim1,
    \qquad
    \sin(\theta_\beta-\theta_\alpha)\sim1.
\]
Hence
\[
    |\varsigma_\beta|
    =
    \frac{|\widetilde F_1|}{\widetilde f}
    \gtrsim
    \sqrt{\varepsilon}\,s,
    \qquad
    |\varsigma_\gamma|
    =
    \frac{\sin(\theta_\beta-\theta_\alpha)}{\widetilde f}
    \sim1.
\]
The co-area factor satisfies
\[
    \left|
    \frac{
        \widetilde F_1
        \bigl(
            \cos\theta_\alpha\sin\theta_\beta
            -
            \sin\theta_\alpha\cos\theta_\beta
        \bigr)
    }{
        r\,\widetilde f^{\,3}
    }
    \right|
    \lesssim
    s.
\]
Furthermore, by 
Lemma~\ref{lem:cusp-cancellation-++-},
\[
    1+|\varsigma_\beta|-|\varsigma_\gamma|
    \lesssim
    s.
\]
Thus the factor
\[
    (1+|\varsigma_\beta|-|\varsigma_\gamma|)^2
\]
contributes \(s^2\).  Therefore, using the weighted \(L^\infty_m\) bound for
\(n\), the integrand in \(I_{1,3}\) is bounded by
\[
\begin{aligned}
\mathfrak k^1_{++-}(\theta_\beta)
&\lesssim
s^3
\|n\|_{L^\infty_m}^2
\Bigl(
    \langle |k_\alpha|\varsigma_\beta\rangle^{-m}
    \langle |k_\alpha|\varsigma_\gamma\rangle^{-m}
\\
&\qquad\qquad
    +
    \langle k_\alpha\rangle^{-m}
    \langle |k_\alpha|\varsigma_\beta\rangle^{-m}
    +
    \langle k_\alpha\rangle^{-m}
    \langle |k_\alpha|\varsigma_\gamma\rangle^{-m}
\Bigr).
\end{aligned}
\]
Since
\[
    |\varsigma_\beta|\gtrsim\sqrt{\varepsilon}\,s,
    \qquad
    |\varsigma_\gamma|\sim1,
\]
we have
\[
    \langle |k_\alpha|\varsigma_\beta\rangle^{-m}
    \lesssim
    \langle \sqrt{\varepsilon}\,|k_\alpha|s\rangle^{-m},
    \qquad
    \langle |k_\alpha|\varsigma_\gamma\rangle^{-m}
    \lesssim
    \langle |k_\alpha|\rangle^{-m}.
\]

First assume \(|k_\alpha|\le1\).  Then
\[
    \langle |k_\alpha|\rangle^{-m}\sim1,
    \qquad
    \langle \sqrt{\varepsilon}\,|k_\alpha|s\rangle^{-m}\lesssim1
\]
on \(0<s\le\delta(\varepsilon)\).  Hence
\[
\begin{aligned}
|I_{1,3}|
&\lesssim
|k_\alpha|^4
\int_0^{\delta(\varepsilon)}
s^3\,ds\,
\|n\|_{L^\infty_m}^2
\\
&\lesssim
\|n\|_{L^\infty_m}^2.
\end{aligned}
\]
Thus
\[
    \langle k_\alpha\rangle^m |I_{1,3}|
    \lesssim
    \|n\|_{L^\infty_m}^2,
    \qquad |k_\alpha|\le1.
\]

Now assume \(|k_\alpha|>1\).  We split the cusp integral at
\[
    s=|k_\alpha|^{-1}.
\]
On \(0<s\le |k_\alpha|^{-1}\), we use
\[
    \langle \sqrt{\varepsilon}\,|k_\alpha|s\rangle^{-m}\lesssim1,
    \qquad
    \langle |k_\alpha|\rangle^{-m}\sim |k_\alpha|^{-m}.
\]
On \(|k_\alpha|^{-1}\le s\le\delta(\varepsilon)\), we use
\[
    \langle \sqrt{\varepsilon}\,|k_\alpha|s\rangle^{-m}
    \lesssim
    \varepsilon^{-\frac m2}|k_\alpha|^{-m}s^{-m},
    \qquad
    \langle |k_\alpha|\rangle^{-m}\sim |k_\alpha|^{-m}.
\]
Therefore
\[
\begin{aligned}
\langle k_\alpha\rangle^m |I_{1,3}|
&\lesssim
|k_\alpha|^{m+4}
\int_0^{|k_\alpha|^{-1}}
s^3 |k_\alpha|^{-m}\,ds\,
\|n\|_{L^\infty_m}^2
\\
&\quad
+
|k_\alpha|^{m+4}
\int_{|k_\alpha|^{-1}}^{\delta(\varepsilon)}
s^3\,
\varepsilon^{-\frac m2}
|k_\alpha|^{-m}s^{-m}
|k_\alpha|^{-m}
\,ds\,
\|n\|_{L^\infty_m}^2
\\
&\lesssim
\left[
1
+
\varepsilon^{-\frac m2}
|k_\alpha|^{4-m}
\int_{|k_\alpha|^{-1}}^{\delta(\varepsilon)}
s^{3-m}\,ds
\right]
\|n\|_{L^\infty_m}^2.
\end{aligned}
\]
For \(m>4\), and \(0<\varepsilon\ll 1\),
\[
    \int_{|k_\alpha|^{-1}}^{\delta(\varepsilon)}
    s^{3-m}\,ds
    \lesssim
    |k_\alpha|^{m-4}
    +
    \delta(\varepsilon)^{-m+4}.
\]

The estimate of \(I_{1,4}\) is identical.  Indeed, set
\[
    \widetilde s
    :=
    \theta_\beta-\widetilde\theta_\beta^\ast,
    \qquad
    0<\widetilde s\le\delta(\varepsilon).
\]
Using Lemma~\ref{lem:second-zero} in place of Lemma~\ref{lem:first-zero} and
the second estimate in Lemma~\ref{lem:cusp-cancellation-++-}, the same
argument gives
\[
    \langle k_\alpha\rangle^m |I_{1,4}|
    \lesssim
    \varepsilon^{-\frac m2}
    \delta(\varepsilon)^{-m+4}
    \|n\|_{L^\infty_m}^2,
\]
again provided that \(m>4\) and \(0<\varepsilon\ll1\).

Finally, on \(I_{1,5}\), the curve is away from the second cusp and away
from the asymptotic end.  We have
\[ \frac{1}{\sqrt{5}}\le\cos\theta_\alpha=\tilde{f}(\theta_\alpha,\theta_\alpha+\pi)\le\tilde{f}\le 1,
    \qquad
    \widetilde F_1
    \gtrsim
    \sqrt{\varepsilon}\,\delta(\varepsilon),
\]
and by the cut-off positivity estimate,
\[
    \varepsilon
    \lesssim
    \sin(\theta_\beta-\theta_\alpha)
    \le1.
\]
Thus
\[
    |\varsigma_\beta|
    \gtrsim
    \sqrt{\varepsilon}\,\delta(\varepsilon),
    \qquad
    |\varsigma_\gamma|
    \gtrsim
    \varepsilon.
\]
The bounded-region estimate gives
\[
    \langle k_\alpha\rangle^m |I_{1,5}|
    \lesssim
    \varepsilon^{-\frac{3m}{2}}
    \delta(\varepsilon)^{-m}
    \|n\|_{L^\infty_m}^2,
\]
provided that \(m\ge4\) and \(0<\varepsilon\ll1\).

The \(I_2\)-branch is estimated exactly as in Case~\ref{subsubsec:++-case1-nocusp},
using Lemma~\ref{Left-hand linear asymptote},
Lemma~\ref{SecondPositivity of gamma}, and observing
\[\frac{\sqrt{19}-2}{10}\le -\frac{\sin\theta_\alpha\cos\theta_\alpha}{2}+\cos\theta_\alpha\sqrt{1-\frac{\cos^2\theta_\alpha}{4}}\le |F_1|\le 1.\]
Hence
\[
    \langle k_\alpha\rangle^m |I_2|
    \lesssim
    \varepsilon^{-m}
    \|n\|_{L^\infty_m}^2,
\]
provided that \(m\ge4\) and \(0<\varepsilon\ll1\).
Combining the estimates above, we obtain in Case~2

\[
\begin{aligned}
\langle k_\alpha\rangle^m
|\mathcal C_{++-}(n)(k_\alpha)|
& \lesssim_{\varepsilon,m}
\|n\|_{L^\infty_m}^2 ,
\end{aligned}
\]
provided that \(m>4\) and \(0<\varepsilon\ll1\).


\subsubsection{Case 3: overlapping asymptotic and cusp regions}
\label{subsubsec:++-case3-overlap}

Assume
\[
    \varepsilon\le \cos\theta_\alpha\le \frac12-\varepsilon,
\]
and suppose
\[
    \theta_-^0
    \ge
    \theta_\beta^\ast-\delta(\varepsilon).
\]
In this case, the asymptotic region overlaps with the first cusp
neighborhood.  Thus the transition piece \(I_{1,2}\) from
Case~\ref{subsubsec:++-case2-separated} is absent, and we decompose
\[
    I_1=I_{1,1}+I_{1,3}+I_{1,4}+I_{1,5},
\]
where
\[
    I_{1,1}
    =
    |k_\alpha|^4
    \int_{\theta_-^\infty}^{\theta_\beta^\ast-\delta(\varepsilon)}
        \mathfrak k^1_{++-}(\theta_\beta)
    \,d\theta_\beta,
\]
and \(I_{1,3}, I_{1,4}, I_{1,5}\) are defined as in
Case~\ref{subsubsec:++-case2-separated}.

On \(I_{1,1}\), we use the same asymptotic estimate as in Case~2.  Namely,
by Lemma~\ref{Right-hand linear asymptote},
\[
    \widetilde f
    \sim
    \theta_\beta-\theta_-^\infty.
\]
Moreover, since the interval stops at
\(\theta_\beta^\ast-\delta(\varepsilon)\), Lemma~\ref{lem:first-zero} gives
\[
    \widetilde F_1
    \gtrsim
    \sqrt{\varepsilon}\,\delta(\varepsilon),
\]
and monotonicity yields
\[\varepsilon\lesssim\sin(\arccos(\frac{\varepsilon}{2})+\arccos({\varepsilon}))\le \sin(\arccos(\frac{\cos\theta_\alpha}{2})+\theta_\alpha)\le\sin(\theta_\beta-\theta_\alpha)\le 1.\]

Therefore
\[
    |k_\beta|
    =
    |k_\alpha|\frac{\widetilde F_1}{\widetilde f}
    \gtrsim
    \sqrt{\varepsilon}\,\delta(\varepsilon)
    \frac{|k_\alpha|}{\theta_\beta-\theta_-^\infty},
\]
and
\[
    |k_\gamma|
    =
    |k_\alpha|
    \frac{\sin(\theta_\beta-\theta_\alpha)}{\widetilde f}
    \gtrsim
    \varepsilon
    \frac{|k_\alpha|}{\theta_\beta-\theta_-^\infty}.
\]
Thus the same calculation as for the corresponding asymptotic piece in
Case~2 gives
\[
    \langle k_\alpha\rangle^m |I_{1,1}|
    \lesssim
    \varepsilon^{-\frac{3m}{2}}
    \delta(\varepsilon)^{-m}
    \|n\|_{L^\infty_m}^2,
\]
provided that \(m\ge4\) and \(0<\varepsilon\ll1\).

We next estimate the first cusp piece \(I_{1,3}\).  Set
\[
    s:=\theta_\beta^\ast-\theta_\beta,
    \qquad
    0<s\le\delta(\varepsilon).
\]
By Lemma~\ref{lem:first-zero},
\[
    \sqrt{\varepsilon}\,s
    \lesssim
    |\widetilde F_1|
    \lesssim
    s.
\]
In the present Case~3 regime, \(\widetilde f\) is controlled only by the
cut-off lower bound:
\[
    \widetilde f\gtrsim\varepsilon,
    \qquad
    \frac{1}{\widetilde f^3}
    \lesssim
    \varepsilon^{-3}.
\]
Moreover,
\[
    \sin(\theta_\beta-\theta_\alpha)\sim1
\]
near the cusp, and by Lemma~\ref{lem:cusp-cancellation-++-},
\[
    1+|\varsigma_\beta|-|\varsigma_\gamma|
    \lesssim s.
\]
Therefore the cusp estimate is the same as in Case~2, except with the
additional loss \(\varepsilon^{-3}\) coming from
\(\widetilde f^{-3}\).  Hence
\[
    \langle k_\alpha\rangle^m |I_{1,3}|
    \lesssim
    \varepsilon^{-\frac{3m}{2}-3}
    \delta(\varepsilon)^{-m+4}
    \|n\|_{L^\infty_m}^2.
\]
Here, as in Case~2, the estimate is obtained by splitting the cusp integral
at
\[
    s=|k_\alpha|^{-1}
\]
when \(|k_\alpha|>1\).  The borderline case \(m=4\) gives the logarithmic loss
discussed in Remark~\ref{rem:endpoint-m4-cusp}; the displayed bound is the
clean form for \(m>4\).

The estimate of \(I_{1,4}\) is analogous. Indeed, set
\[
    \widetilde s
    :=
    \theta_\beta-\widetilde\theta_\beta^\ast,
    \qquad
    0<\widetilde s\le\delta(\varepsilon).
\]
Thus the same cusp calculation gives
\[
    \langle k_\alpha\rangle^m |I_{1,4}|
    \lesssim
    \varepsilon^{-\frac{3m}{2}-3}
    \delta(\varepsilon)^{-m+4}
    \|n\|_{L^\infty_m}^2,
\]
again with the endpoint \(m=4\) understood as in
Remark~\ref{rem:endpoint-m4-cusp}.

We now estimate \(I_{1,5}\).  On this interval,
\[
    \theta_\beta\in
    [\widetilde\theta_\beta^\ast+\delta(\varepsilon),\,\pi+\theta_\alpha].
\]
Hence the curve is away from the second cusp, and Lemma~\ref{lem:second-zero}
gives
\[
    \widetilde F_1
    \gtrsim
    \sqrt{\varepsilon}\,\delta(\varepsilon).
\]
In Case~3, the denominator satisfies
\[
    \varepsilon
    \le
    \cos\theta_\alpha
    =
    \widetilde f(\theta_\alpha,\pi+\theta_\alpha)
    \le
    \widetilde f
    \le 1,
\]
so that
\[
    \frac{1}{\widetilde f^3}
    \lesssim
    \varepsilon^{-3}.
\]
By the cut-off positivity estimate,
\[
    \varepsilon
    \lesssim
    \sin(\theta_\beta-\theta_\alpha)
    \le1.
\]
Therefore
\[
    |\varsigma_\beta|
    =
    \frac{\widetilde F_1}{\widetilde f}
    \gtrsim
    \sqrt{\varepsilon}\,\delta(\varepsilon),
\]
and
\[
    |\varsigma_\gamma|
    =
    \frac{\sin(\theta_\beta-\theta_\alpha)}{\widetilde f}
    \gtrsim
    \varepsilon.
\]
The bounded-region estimate then gives
\[
    \langle k_\alpha\rangle^m |I_{1,5}|
    \lesssim
    \varepsilon^{-\frac{3m}{2}-3}
    \delta(\varepsilon)^{-m}
    \|n\|_{L^\infty_m}^2,
\]
provided that \(m\ge4\) and \(0<\varepsilon\ll1\).

Finally, the \(I_2\)-branch is estimated as in
Case~\ref{subsubsec:++-case1-nocusp}.  By
Lemma~\ref{Left-hand linear asymptote},
Lemma~\ref{SecondPositivity of gamma}, and the cut-off lower bound for
\(F_1\),
\[\varepsilon\lesssim -\frac{\sin\theta_\alpha\cos\theta_\alpha}{2}+\cos\theta_\alpha\sqrt{1-\frac{\cos^2\theta_\alpha}{4}}\le |F_1|\le 1,\]
we obtain
\[
    \langle k_\alpha\rangle^m |I_2|
    \lesssim
    \varepsilon^{-2m}
    \|n\|_{L^\infty_m}^2,
\]
provided that \(m\ge4\) and \(0<\varepsilon\ll1\).

Combining the estimates for \(I_{1,1},I_{1,3},I_{1,4},I_{1,5}\), and \(I_2\),
we obtain, in Case~3,
\[
\begin{aligned}
\langle k_\alpha\rangle^m
|\mathcal C_{++-}(n)(k_\alpha)|
&\lesssim
\Bigl(
    \varepsilon^{-\frac{3m}{2}}
        \delta(\varepsilon)^{-m}
    +
    \varepsilon^{-\frac{3m}{2}-3}
        \delta(\varepsilon)^{-m+4}
\\
&\qquad
    +
    \varepsilon^{-\frac{3m}{2}-3}
        \delta(\varepsilon)^{-m}
    +
    \varepsilon^{-2m}
\Bigr)
\|n\|_{L^\infty_m}^2 .
\end{aligned}
\]
provided that \(m>4\) and \(0<\varepsilon\ll1\).
\begin{proof}[Completion of Proposition~\ref{prop:apriori-++-}]
The angular cut-off excludes
\[
    |\cos\theta_\alpha|<\varepsilon
\]
and
\[
    \left||\cos\theta_\alpha|-\frac12\right|<\varepsilon.
\]
After the first-quadrant reduction, the support of the cut-off is covered by
the three cases above:
\[
    \cos\theta_\alpha\ge\frac12+\varepsilon,
\]
or
\[
    \varepsilon\le\cos\theta_\alpha\le\frac12-\varepsilon,
\]
with the latter split according to whether
\[
    \theta_-^0<\theta_\beta^\ast-\delta(\varepsilon)
\]
or
\[
    \theta_-^0\ge\theta_\beta^\ast-\delta(\varepsilon).
\]
Combining the three cases gives
\[
    \langle k_\alpha\rangle^m
    \left|
        \mathcal C_{++-}(n)(k_\alpha)
    \right|
    \lesssim_m
    \mathfrak C_{++-}(\varepsilon,\delta,m)
    \|n\|_{L^\infty_m}^2,
\]
provided that \(m>4\) and \(0<\varepsilon\ll1\).
Taking the supremum over \(k_\alpha\) proves the proposition.
\end{proof}

\begin{remark}[Endpoint \(m=4\)]\label{rem:endpoint-m4-cusp}
The proof above closes the \((+,+,-)\) estimate for \(m>4\).  The only
obstruction to the endpoint \(m=4\) in the present absolute-value argument
comes from the cusp pieces \(I_{1,3}\) and \(I_{1,4}\).  In those estimates
one encounters the integral
\[
    \int_{|k_\alpha|^{-1}}^{\delta(\varepsilon)}
        s^{3-m}\,ds .
\]
For \(m>4\), this is bounded by
\[
    |k_\alpha|^{m-4}
    +
    \delta(\varepsilon)^{-m+4},
\]
which is sufficient for the weighted \(L^\infty_m\) estimate.  However, at
\(m=4\) the same integral becomes
\[
    \int_{|k_\alpha|^{-1}}^{\delta(\varepsilon)}
        \frac{ds}{s}
    =
    \log\bigl(\delta(\varepsilon)|k_\alpha|\bigr),
\]
which is not uniformly bounded in \(k_\alpha\).  Thus the present proof does
not give a uniform bound in the plain space \(L^\infty_4\).

One possible way to treat the endpoint would be to use a logarithmically
stronger weight, for example
\[
    \|n\|_{L^\infty_{4,\ell}}
    :=
    \sup_{\sigma=\pm1}\sup_{k\in\mathbb R^2}
    \langle k\rangle^4
    \bigl(\log(e+|k|)\bigr)^\ell
    |n(\sigma,k)|,
    \qquad
    \ell>1.
\]
Indeed, the logarithmic weight would replace the borderline integral by
\[
    \int_{|k_\alpha|^{-1}}^{\delta(\varepsilon)}
    \frac{ds}
    {s\bigl(\log(e+|k_\alpha|s)\bigr)^\ell},
\]
which is uniformly bounded for \(\ell>1\).  Alternatively, one would need an
additional cancellation in the cusp contribution to recover the endpoint
\(m=4\) in the unmodified \(L^\infty_4\) space.
\end{remark}

\section{Parametrization and Estimates for the \((+,-,-)\) Resonance Manifold}
\label{Parametrization and Estimates for the $(+,-,-)$ Resonance Manifold}
\subsection{Unit--scale parametrization}
\label{Parametrization(+,-,-)}
We consider the sign configuration
\[
    (\sigma_\alpha,\sigma_\beta,\sigma_\gamma)=(+,-,-),
\]
corresponding to~\eqref{manifold3}.  By homogeneity, we work at unit
scale:
\[
    |k_\alpha|=1,
    \qquad
    k_\alpha=(\cos\theta_\alpha,\sin\theta_\alpha),
    \qquad
    0\le \theta_\alpha<\frac{\pi}{2}.
\]
As before, write
\[
    k_j=|k_j|(\cos\theta_j,\sin\theta_j),
    \qquad
    j\in\{\alpha,\beta,\gamma\}.
\]
Solving the polar momentum system
\[
    k_\alpha+k_\beta+k_\gamma=0
\]
gives
\begin{equation}
\label{eq:polar-momentum-+--}
    |k_\beta|
    =
    \frac{\sin(\theta_\gamma-\theta_\alpha)}
         {\sin(\theta_\beta-\theta_\gamma)},
    \qquad
    |k_\gamma|
    =
    \frac{\sin(\theta_\beta-\theta_\alpha)}
         {\sin(\theta_\gamma-\theta_\beta)}.
\end{equation}
For the \((+,-,-)\) configuration, the frequency resonance is
\[
    \cos\theta_\alpha-\cos\theta_\beta-\cos\theta_\gamma=0,
\]
so that
\[
    \cos\theta_\gamma
    =
    \cos\theta_\alpha-\cos\theta_\beta.
\]
Substituting the two choices
\[
    \sin\theta_\gamma
    =
    \pm
    \sqrt{1-(\cos\theta_\alpha-\cos\theta_\beta)^2}
\]
into \eqref{eq:polar-momentum-+--}, and imposing
\[
    |k_\beta|\ge0,
    \qquad
    |k_\gamma|\ge0,
\]
gives the following parametrization.
\begin{proposition}[Unit--scale parametrization of the \((+,-,-)\) resonance curve]
\label{prop:param-+--}
If
\[
    0\le \theta_\alpha\le\frac{\pi}{3},
\]
then the \((+,-,-)\) resonant set has no nontrivial one-dimensional component.
At \(\theta_\alpha=\pi/3\), its closure degenerates to the two collapsed
points
\[
    k_\beta=0,
    \qquad
    k_\beta=-k_\alpha.
\]

Assume now that
\[
    \frac{\pi}{3}<\theta_\alpha<\frac{\pi}{2}.
\]
Define
\[
    \theta_\beta^1
    :=
    \arccos(2\cos\theta_\alpha),
    \qquad
    \theta_\beta^2
    :=
    -\arccos(2\cos\theta_\alpha),
\]
and
\[
    \theta_\beta^3
    :=
    \pi+\arccos(1-\cos\theta_\alpha),
    \qquad
    \theta_\beta^4
    :=
    \pi+\theta_\alpha.
\]
Then
\[
    \theta_\beta^2<\theta_\beta^1,
    \qquad
    \theta_\beta^3<\theta_\beta^4.
\]

The closures of the unit-scale \((+,-,-)\) resonance arcs are parametrized by
\[
    k_\beta
    =
    |k_\beta|(\cos\theta_\beta,\sin\theta_\beta),
\]
where
\[
|k_\beta|
=
\begin{cases}
\displaystyle
\frac{
\sin\theta_\alpha(\cos\theta_\alpha-\cos\theta_\beta)
+
\cos\theta_\alpha
\sqrt{1-(\cos\theta_\alpha-\cos\theta_\beta)^2}
}{
-\sin\theta_\beta(\cos\theta_\alpha-\cos\theta_\beta)
-\cos\theta_\beta
\sqrt{1-(\cos\theta_\alpha-\cos\theta_\beta)^2}
},
&
\begin{array}{l}
\theta_\beta^2\le \theta_\beta\le \theta_\beta^1,\\
\sin\theta_\gamma\le0,
\end{array}
\\[2.0em]
\displaystyle
\frac{
\sin\theta_\alpha(\cos\theta_\alpha-\cos\theta_\beta)
+
\cos\theta_\alpha
\sqrt{1-(\cos\theta_\alpha-\cos\theta_\beta)^2}
}{
-\sin\theta_\beta(\cos\theta_\alpha-\cos\theta_\beta)
-\cos\theta_\beta
\sqrt{1-(\cos\theta_\alpha-\cos\theta_\beta)^2}
},
&
\begin{array}{l}
\theta_\beta^3\le \theta_\beta\le \theta_\beta^4,\\
\sin\theta_\gamma\le0,
\end{array}
\\[2.0em]
\displaystyle
\frac{
\sin\theta_\alpha(\cos\theta_\alpha-\cos\theta_\beta)
-
\cos\theta_\alpha
\sqrt{1-(\cos\theta_\alpha-\cos\theta_\beta)^2}
}{
-\sin\theta_\beta(\cos\theta_\alpha-\cos\theta_\beta)
+
\cos\theta_\beta
\sqrt{1-(\cos\theta_\alpha-\cos\theta_\beta)^2}
},
&
\begin{array}{l}
\theta_\beta^3\le \theta_\beta\le \theta_\beta^4,\\
\sin\theta_\gamma\ge0.
\end{array}
\end{cases}
\]
Here
\[
    \cos\theta_\gamma
    =
    \cos\theta_\alpha-\cos\theta_\beta,
\]
and the sign of \(\sin\theta_\gamma\) distinguishes the two branches.

The first line gives the component whose closure contains the collapsed (or corner)
endpoint \(k_\beta=0\).  Indeed, both limiting directions
\[
    \theta_\beta=\theta_\beta^1,
    \qquad
    \theta_\beta=\theta_\beta^2
\]
correspond to \(k_\beta=0\).  The second and third lines form the other
closed component.  At
\[
    \theta_\beta=\theta_\beta^4=\pi+\theta_\alpha,
\]
both formulas give
\[
    |k_\beta|=1,
    \qquad
    k_\beta=-k_\alpha,
\]
and hence
\[
    k_\gamma=0.
\]
The two components are exchanged by the central symmetry
\[
    k_\beta\longmapsto -k_\alpha-k_\beta,
\]
that is, by the half-turn around
\[
    -\frac{k_\alpha}{2}
    =
    \left(
        -\frac{\cos\theta_\alpha}{2},
        -\frac{\sin\theta_\alpha}{2}
    \right).
\]
\end{proposition}

For later use, we only work on the first component,
\[
    \theta_\beta\in(\theta_\beta^2,\theta_\beta^1),
\]
since the second component follows by the central symmetry above.
\begin{figure}[H]
    \centering
    \includegraphics[width=0.78\textwidth]{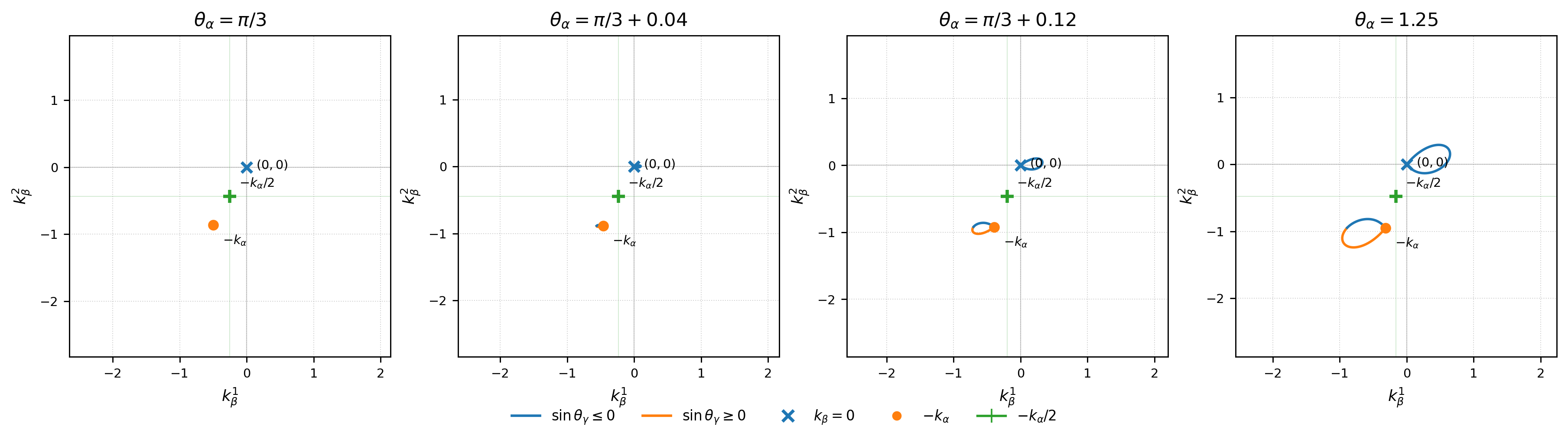}
    \caption{
        Unit-scale \((+,-,-)\) resonance set in the \(k_\beta\)-plane near
        the threshold \(\theta_\alpha=\pi/3\).  The window is centered at the
        symmetry point \(-k_\alpha/2\).  At \(\theta_\alpha=\pi/3\), the set
        degenerates to \(0\) and \(-k_\alpha\); for
        \(\theta_\alpha>\pi/3\), these points open into two closed curves.
        Blue and orange indicate the branches
        \(\sin\theta_\gamma\le0\) and \(\sin\theta_\gamma\ge0\), respectively.
    }
    \label{fig:manifold3-threshold}
\end{figure}

\begin{figure}[H]
    \centering
    \includegraphics[width=0.78\textwidth]{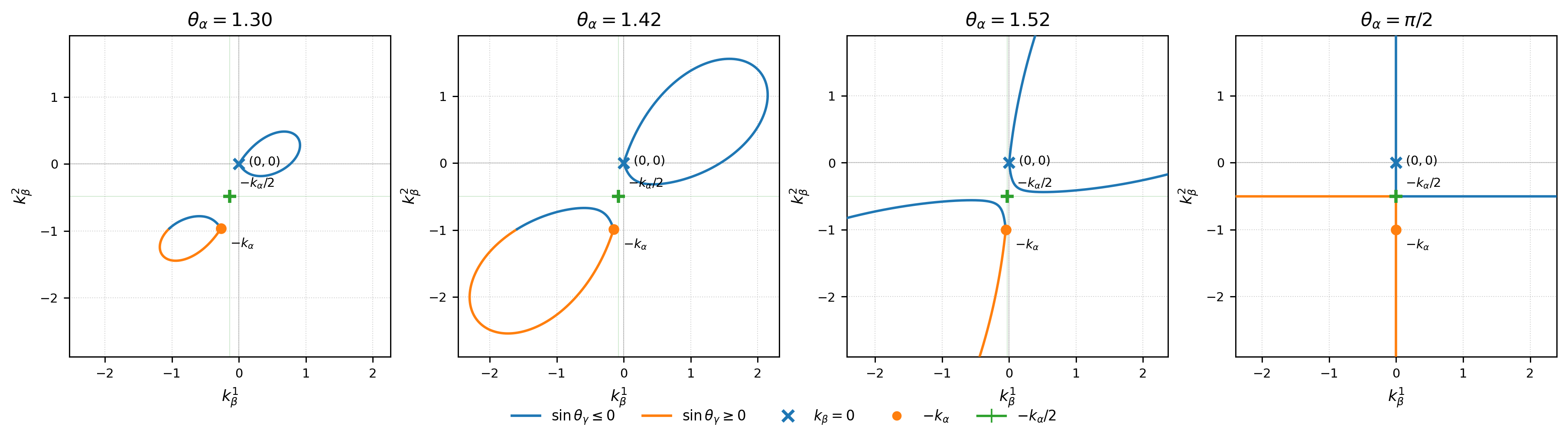}
    \caption{
        Unit-scale \((+,-,-)\) resonance set as
        \(\theta_\alpha\to\pi/2\).  The final panel shows the limiting
        degenerate set
        \(
            k_\beta^1=0
            \quad\text{and}\quad
            k_\beta^2=-\frac12,
        \)
        intersecting at the symmetry center \(-k_\alpha/2=(0,-1/2)\).
        Blue and orange indicate the branches
        \(\sin\theta_\gamma\le0\) and \(\sin\theta_\gamma\ge0\), respectively.
    }
    \label{fig:manifold3-vertical}
\end{figure}

\subsection{Co--area factor and corner estimates}
\label{subsec:coarea-corner-+--}

We work on the component
\[
    \theta_\beta\in(\theta_\beta^2,\theta_\beta^1),
\]
where
\[
    \theta_\beta^1=\arccos(2\cos\theta_\alpha),
    \qquad
    \theta_\beta^2=-\arccos(2\cos\theta_\alpha).
\]
Throughout this subsection, we use the notation
\[
    u:=\cos\theta_\alpha-\cos\theta_\beta,
    \qquad
    \overline r:=|\sin\theta_\gamma|
    =
    \sqrt{1-u^2},
\]
and
\[
    \overline f
    :=
    -\sin\theta_\beta\,u
    -
    \cos\theta_\beta\,\overline r.
\]
We also set
\[
    \overline F_1(\theta_\alpha,\theta_\beta)
    :=
    \sin\theta_\alpha\,u
    +
    \cos\theta_\alpha\,\overline r.
\]
Thus, on this component,
\[
    |k_\beta|
    =
    \frac{\overline F_1}{\overline f},
    \qquad
    |k_\gamma|
    =
    \frac{\sin(\theta_\beta-\theta_\alpha)}{\overline f}.
\]

After eliminating
\[
    k_\gamma=-k_\alpha-k_\beta,
\]
the reduced frequency constraint is
\[
    \widetilde\Xi_1(k_\beta)
    =
    \frac{k_\alpha^1}{|k_\alpha|}
    -
    \frac{k_\beta^1}{|k_\beta|}
    -
    \frac{k_\gamma^1}{|k_\gamma|}.
\]

\begin{lemma}[Co--area factor for the \((+,-,-)\) curve]
\label{lem:coarea-ratio-+--}
On the unit-scale \((+,-,-)\) resonance curve,
one has
\[
    \frac{
        \left|\dfrac{d k_\beta^1}{d\theta_\beta}\right|
    }{
        \left|\partial_{k_\beta^2}\widetilde\Xi_1\right|
    }
    =
    \left|
    \frac{
        \overline F_1(\theta_\alpha,\theta_\beta)
        \sin(\theta_\beta-\theta_\alpha)
    }{
        \overline r\,\overline f^{\,3}
    }
    \right|.
\]
\end{lemma}

\begin{proof}
All identities in this lemma are understood on the open component
\[
    \theta_\beta^2<\theta_\beta<\theta_\beta^1,
\]
away from the collapsed endpoints.  The endpoint behavior is obtained by
taking limits.
A direct computation gives
\[
    \partial_{k_\beta^2}\widetilde\Xi_1
    =
    \frac{
        \overline f\,J_3
    }{
        \sin(\theta_\beta-\theta_\alpha)\,
        \overline F_1
    },
\]
and
\[
    \frac{d k_\beta^1}{d\theta_\beta}
    =
    \frac{J_3}{\overline r\,\overline f^{\,2}},
\]
where
\[
\begin{aligned}
J_3
&=
u\overline r
(\sin\theta_\alpha u+\cos\theta_\alpha\overline r)
\\
&\qquad
+
\sin\theta_\beta\cos\theta_\beta
(\cos\theta_\alpha\sin\theta_\beta
-\sin\theta_\alpha\cos\theta_\beta).
\end{aligned}
\]
Dividing these two identities gives the claim.
\end{proof}

\begin{lemma}[Cut-off lower bounds on the \((+,-,-)\) component]
\label{lem:cutoff-lower-+--}
Assume
\[
    \varepsilon<\cos\theta_\alpha<\frac12-\varepsilon,
    \qquad
    \theta_\beta^2<\theta_\beta<\theta_\beta^1.
\]
Then on the support of the angular cut-off,
\[
    |\overline f|\gtrsim\varepsilon,
    \qquad
    |\sin(\theta_\beta-\theta_\alpha)|\gtrsim\varepsilon,
    \qquad
    \overline r\gtrsim\sqrt\varepsilon.
\]
Consequently,
\[
    \frac{1}{\overline r\,|\overline f|^3}
    \lesssim
    \varepsilon^{-7/2}.
\]
\end{lemma}

\begin{proof}
By monotonicity on the component,
\[
    |\overline f|
    \ge
    |\overline f(\theta_\alpha,\theta_\beta^1)|
    =
    -\cos\theta_\alpha\sqrt{1-4\cos^2\theta_\alpha}
    +
    2\cos\theta_\alpha\sin\theta_\alpha
    \gtrsim\varepsilon.
\]
The same endpoint expression gives
\[
    |\sin(\theta_\beta-\theta_\alpha)|
    \ge
    |\sin(\theta_\beta^1-\theta_\alpha)|
    \gtrsim\varepsilon.
\]
Moreover,
\[
    \overline r
    \ge
    \sqrt{2\cos\theta_\alpha-\cos^2\theta_\alpha}
    \gtrsim\sqrt\varepsilon.
\]
Therefore
\[
    \frac{1}{\overline r\,|\overline f|^3}
    \lesssim
    \varepsilon^{-1/2}\varepsilon^{-3}
    =
    \varepsilon^{-7/2}.
\]
\end{proof}
\begin{lemma}[Left-hand linear behavior near the corner]
\label{lem:left-zero}
Let \(\overline F_1\) be as above.  If
\[
    \varepsilon\le \cos\theta_\alpha\le\frac12-\varepsilon,
    \qquad
    \varepsilon>0 \text{ sufficiently small},
\]
then there exists \(\delta_1(\varepsilon)>0\) such that, whenever
\[
    |\theta_\beta-\theta_\beta^1(\theta_\alpha)|
    \le\delta_1(\varepsilon),
\]
one has
\[
    \sqrt\varepsilon\,
    |\theta_\beta-\theta_\beta^1(\theta_\alpha)|
    \lesssim
    |\overline F_1(\theta_\alpha,\theta_\beta)|
    \lesssim
    |\theta_\beta-\theta_\beta^1(\theta_\alpha)|.
\]
\end{lemma}

\begin{proof}
Set
\[
    \vartheta_\beta:=\pi-\theta_\beta.
\]
By the definitions of \(p,r\) in Section~4 and \(u,\overline r\) in the
present section,
\[
    u(\theta_\alpha,\theta_\beta)
    =
    p(\theta_\alpha,\vartheta_\beta),
    \qquad
    \overline r(\theta_\alpha,\theta_\beta)
    =
    r(\theta_\alpha,\vartheta_\beta).
\]
Consequently,
\[
    \overline F_1(\theta_\alpha,\theta_\beta)
    =
    \widetilde F_1(\theta_\alpha,\vartheta_\beta).
\]
Moreover,
\[
    \theta_\beta^\ast(\theta_\alpha)
    =
    \pi-\theta_\beta^1(\theta_\alpha),
\]
and hence
\[
    \left|
        \vartheta_\beta-\theta_\beta^\ast(\theta_\alpha)
    \right|
    =
    \left|
        \theta_\beta-\theta_\beta^1(\theta_\alpha)
    \right|.
\]
The result therefore follows directly from
Lemma~\ref{lem:first-zero}, with
\(\delta_1(\varepsilon)=\delta(\varepsilon)\).
\end{proof}

Since \(\overline F_1\) is even in \(\theta_\beta\), the other corner is
symmetric.

\begin{lemma}[Right-hand linear behavior near the corner]
\label{lem:right-zero}
Under the same hypotheses as in Lemma~\ref{lem:left-zero}, there exists
\(\delta_1(\varepsilon)>0\) such that, whenever
\[
    |\theta_\beta-\theta_\beta^2(\theta_\alpha)|
    \le\delta_1(\varepsilon),
\]
one has
\[
    \sqrt\varepsilon\,
    |\theta_\beta-\theta_\beta^2(\theta_\alpha)|
    \lesssim
    |\overline F_1(\theta_\alpha,\theta_\beta)|
    \lesssim
    |\theta_\beta-\theta_\beta^2(\theta_\alpha)|.
\]
\end{lemma}

\begin{proof}
Since
\[
    \theta_\beta^2=-\theta_\beta^1
\]
and \(\overline F_1(\theta_\alpha,\cdot)\) is even, this follows
immediately from Lemma~\ref{lem:left-zero}.
\end{proof}
\subsection{A priori boundedness estimate for the \((+,-,-)\) contribution}
\label{subsec:apriori-+--}

We return to general \(k_\alpha\neq0\), writing
\[
    k_\alpha=|k_\alpha|\varsigma_\alpha,
    \qquad
    k_\beta=|k_\alpha|\varsigma_\beta,
    \qquad
    k_\gamma=|k_\alpha|\varsigma_\gamma.
\]
By central symmetry of the two components, it is enough to estimate the
component
\[
    \theta_\beta\in(\theta_\beta^2,\theta_\beta^1)
\]
and multiply by \(2\).

We denote by
\[
    \mathcal C_{+--}(n)(k_\alpha)
\]
the contribution to the collision operator with
\[
    \sigma_\alpha=+1,
    \qquad
    \sigma_\beta=\sigma_\gamma=-1.
\]

\begin{proposition}[A priori bound for the \((+,-,-)\) contribution]
\label{prop:apriori-+--}
Let \(m>4\), and fix \(0<\varepsilon\ll1\).  Let
\(\delta_1(\varepsilon)>0\) be the corner scale fixed in
Lemmas~\ref{lem:left-zero} and~\ref{lem:right-zero}.  Then, for every
\(n\in L^\infty_m\),
\[
    \sup_{k_\alpha\in\mathbb R^2}
    \langle k_\alpha\rangle^m
    \left|
        \mathcal C_{+--}(n)(k_\alpha)
    \right|
    \lesssim_m
    \varepsilon^{-\frac{3m+7}{2}}
    \left(
        \delta_1(\varepsilon)^{-m}
        +
        \delta_1(\varepsilon)^{-m+4}
    \right)
    \|n\|_{L^\infty_m}^2.
\]
\end{proposition}

\begin{proof}
Using the scaled co-area representation and the central symmetry, we write
\[
    \mathcal C_{+--}(n)(k_\alpha)
    =
    2|k_\alpha|^4
    \int_{\theta_\beta^2}^{\theta_\beta^1}
        \mathfrak k_{+--}(\theta_\beta)
    \,d\theta_\beta,
\]
where
\[
\begin{aligned}
\mathfrak k_{+--}
&=
\chi_\alpha\chi_\beta\chi_\gamma
{\pi^4}
(\sin\theta_\alpha-\sin\theta_\beta+\overline r)^2
(1-|\varsigma_\beta|-|\varsigma_\gamma|)^2
\cos\theta_\alpha
\\
&\quad\times
\Bigl(
\cos\theta_\alpha
n(-,|k_\alpha|\varsigma_\beta)
n(-,|k_\alpha|\varsigma_\gamma)
-
u
n(+,|k_\alpha|\varsigma_\alpha)
n(-,|k_\alpha|\varsigma_\beta)
\\
&\qquad
-
\cos\theta_\beta
n(+,|k_\alpha|\varsigma_\alpha)
n(-,|k_\alpha|\varsigma_\gamma)
\Bigr)
\left|
\frac{
    \overline F_1(\theta_\alpha,\theta_\beta)
    \sin(\theta_\beta-\theta_\alpha)
}{
    \overline r\,\overline f^3
}
\right|.
\end{aligned}
\]

We split the integral into two corner pieces and one interior piece:
\[
    \mathcal C_{+--}(n)(k_\alpha)
    =
    I_1+I_2+I_3,
\]
where
\[
I_1
=
2|k_\alpha|^4
\int_{\theta_\beta^2}^{\theta_\beta^2+\delta_1(\varepsilon)}
\mathfrak k_{+--}(\theta_\beta)\,d\theta_\beta,
\]
\[
I_2
=
2|k_\alpha|^4
\int_{\theta_\beta^2+\delta_1(\varepsilon)}
     ^{\theta_\beta^1-\delta_1(\varepsilon)}
\mathfrak k_{+--}(\theta_\beta)\,d\theta_\beta,
\]
and
\[
I_3
=
2|k_\alpha|^4
\int_{\theta_\beta^1-\delta_1(\varepsilon)}^{\theta_\beta^1}
\mathfrak k_{+--}(\theta_\beta)\,d\theta_\beta.
\]

We first estimate \(I_1\).  Set
\[
    \tau:=\theta_\beta-\theta_\beta^2,
    \qquad
    0<\tau\le\delta_1(\varepsilon).
\]
By Lemma~\ref{lem:right-zero},
\[
    \sqrt{\varepsilon}\,\tau
    \lesssim
    |\overline F_1|
    \lesssim
    \tau.
\]
By Lemma~\ref{lem:cutoff-lower-+--},
\[
    \frac{1}{\overline r\,|\overline f|^3}
    \lesssim
    \varepsilon^{-7/2}.
\]

As in the \((+,+,-)\) corner estimate, the triangle inequality and
\(\varsigma_\alpha+\varsigma_\beta+\varsigma_\gamma=0\) give
\[
    \bigl|1-|\varsigma_\beta|-|\varsigma_\gamma|\bigr|
    \le 2|\varsigma_\beta|.
\]
Since
\[
    |\varsigma_\beta|
    =
    \left|\frac{\overline F_1}{\overline f}\right|
    \lesssim_\varepsilon
    |\theta_\beta-\theta_\beta^j|,
    \qquad j=1,2,
\]
we obtain
\[
    \bigl|1-|\varsigma_\beta|-|\varsigma_\gamma|\bigr|
    \lesssim_\varepsilon
    |\theta_\beta-\theta_\beta^j|.
\]
Thus the co-area factor and the interaction factor together contribute
\[
    \varepsilon^{-7/2}\tau^3.
\]

The remaining estimate is the same corner calculation as in the \((+,+,-)\)
case.  For \(|k_\alpha|>1\), one splits the corner integral at the stopping
point
\[
    \tau=|k_\alpha|^{-1}.
\]
The piece \(0<\tau<|k_\alpha|^{-1}\) is estimated using
\[
    \langle k_\beta\rangle^{-m}\lesssim1,
    \qquad
    \langle k_\gamma\rangle^{-m}
    \lesssim
    \varepsilon^{-m}|k_\alpha|^{-m},
\]
while the piece \(|k_\alpha|^{-1}\le\tau\le\delta_1(\varepsilon)\) is estimated
using
\[
    \langle k_\beta\rangle^{-m}
    \lesssim
    \varepsilon^{-m/2}|k_\alpha|^{-m}\tau^{-m},
    \qquad
    \langle k_\gamma\rangle^{-m}
    \lesssim
    \varepsilon^{-m}|k_\alpha|^{-m}.
\]
This gives, for \(m>4\),
\[
    \langle k_\alpha\rangle^m |I_1|
    \lesssim
    \varepsilon^{-\frac{3m}{2}-\frac72}
    \delta_1(\varepsilon)^{-m+4}
    \|n\|_{L^\infty_m}^2.
\]
If \(|k_\alpha|\le1\), the direct estimate on
\(0<\tau\le\delta_1(\varepsilon)\) is bounded by the same right-hand side.

On the interior piece \(I_2\), we are away from both corners.  Hence
\[
    \sqrt\varepsilon\,\delta_1(\varepsilon)
    \lesssim
    |\overline F_1|
    \le1.
\]
Together with Lemma~\ref{lem:cutoff-lower-+--}, this gives
\[
    \langle |k_\alpha|\varsigma_\beta\rangle^{-m}
    \lesssim
    (\sqrt\varepsilon\,\delta_1(\varepsilon))^{-m}
    \langle k_\alpha\rangle^{-m},
\]
and
\[
    \langle |k_\alpha|\varsigma_\gamma\rangle^{-m}
    \lesssim
    \varepsilon^{-m}
    \langle k_\alpha\rangle^{-m}.
\]
Therefore
\[
\begin{aligned}
\langle k_\alpha\rangle^m |I_2|
&\lesssim
\langle k_\alpha\rangle^m
|k_\alpha|^4
\varepsilon^{-7/2}
\langle k_\alpha\rangle^{-2m}
\\
&\quad\times
\left(
    \varepsilon^{-m}
    +
    (\sqrt\varepsilon\,\delta_1(\varepsilon))^{-m}
    +
    \varepsilon^{-3m/2}\delta_1(\varepsilon)^{-m}
\right)
\|n\|_{L^\infty_m}^2
\\
&\lesssim
\varepsilon^{-\frac{3m+7}{2}}
\delta_1(\varepsilon)^{-m}
\|n\|_{L^\infty_m}^2,
\end{aligned}
\]
provided that \(m\ge4\).

Finally, \(I_3\) is symmetric to \(I_1\).  Near \(\theta_\beta^1\),
Lemma~\ref{lem:left-zero} gives
\[
    \sqrt{\varepsilon}\,
    |\theta_\beta-\theta_\beta^1|
    \lesssim
    |\overline F_1|
    \lesssim
    |\theta_\beta-\theta_\beta^1|,
\]
and the same triangle-inequality cancellation as in the \((+,+,-)\) corner
estimate gives the required linear factor at the corner.  Thus the identical
argument yields
\[
    \langle k_\alpha\rangle^m |I_3|
    \lesssim
    \varepsilon^{-\frac{3m}{2}-\frac72}
    \delta_1(\varepsilon)^{-m+4}
    \|n\|_{L^\infty_m}^2,
    \qquad m>4.
\]

Combining the estimates for \(I_1,I_2,I_3\), we obtain
\[
    \langle k_\alpha\rangle^m
    \left|
        \mathcal C_{+--}(n)(k_\alpha)
    \right|
    \lesssim
    \varepsilon^{-\frac{3m+7}{2}}
    \left(
        \delta_1(\varepsilon)^{-m}
        +
        \delta_1(\varepsilon)^{-m+4}
    \right)
    \|n\|_{L^\infty_m}^2,
\]
for \(m>4\).  This proves the proposition.
\end{proof}


\section{Unboundedness of the Collision Operator Without a Cut-off Near
\texorpdfstring{$|\cos\theta|=\tfrac12$}{the Degenerate Angle}}
\label{Section: Unboundedness}

In this section we show that the angular localization away from
\(|\cos\theta|=\tfrac12\) is necessary for the weighted
\(L^\infty_m\) boundedness proved above.  We keep the zero-frequency cut-off
away from \(\cos\theta=0\), since the construction below stays uniformly away
from that set, but we remove the cut-off near \(|\cos\theta|=\tfrac12\).

More precisely, let \(\mathcal C^{(0)}\) denote the collision operator obtained
from \eqref{eq:WKE-cutoff-symmetric} by replacing
\(\chi_\varepsilon\) with \(\chi_{0,\varepsilon}\), i.e. with the cut-off away
from \(\cos\theta=0\) only.  Thus no angular localization is imposed near
\(|\cos\theta|=\tfrac12\).  Since all angles used in the counterexample below
are bounded away from \(\cos\theta=0\), the factor \(\chi_{0,\varepsilon}\) is
identically equal to one on the relevant region for all sufficiently large
\(n\).

\begin{lemma}[Unboundedness without the \(|\cos\theta|=\tfrac12\) cut-off]
\label{lem:unbounded-no-cutoff}
Let \(m>4\).  The operator \(\mathcal C^{(0)}\) is not a bounded map from
\(L^\infty_m\) to \(L^\infty_m\).  More precisely, there exists a sequence
\(\{\eta_n\}_{n\gg1}\subset L^\infty_m\) such that
\[
    \sup_n \|\eta_n\|_{L^\infty_m}<\infty,
\]
but
\[
    \sup_{\sigma=\pm1}\sup_{k\in\mathbb R^2}
    \langle k\rangle^m
    \left|\mathcal C^{(0)}(\eta_n)(\sigma,k)\right|
    \longrightarrow \infty
    \qquad\text{as } n\to\infty .
\]
\end{lemma}

\begin{proof}
Let \(m>4\).  We construct a uniformly bounded sequence in \(L^\infty_m\)
whose full collision output diverges.  The construction is localized near the
high--low--high configuration
\[
    (n,\pi/3),\qquad (1,\pi),\qquad (n,4\pi/3),
\]
which is resonant on the \((+,+,-)\) branch since the corresponding signed
frequencies are
\[
    \frac12,\qquad -1,\qquad \frac12.
\]
Thus the frequency resonance is compatible with a nonlocal triad consisting of
two high wave numbers and one order-one wave number.

Fix a small constant \(a_0>0\).  We shall choose \(b_0>0\) sufficiently small,
depending on \(a_0\), and then choose \(M>1\) sufficiently large; all these
constants are independent of \(n\).  Define
\[
A_n
:=
\left\{
    k:
    \bigl||k|-n\bigr|\le 1,
    \quad
    0<\frac\pi3-\theta<2a_0 n^{-1}
\right\},
\]
\[
B_n
:=
\left\{
    k:
    M^{-1}\le |k|\le M,
    \quad
    |\theta-\pi|<b_0 n^{-1/2}
\right\},
\]
and
\[
C_n
:=
\left\{
    k:
    \bigl||k|-n\bigr|\le M,
    \quad
    \left|\theta-\frac{4\pi}{3}\right|<M n^{-1}
\right\}.
\]
Set
\[
    \eta_{n,+}(k)
    :=
    n^{-m}\mathbf 1_{A_n}(k)+\mathbf 1_{B_n}(k),
    \qquad
    \eta_{n,-}(k)
    :=
    n^{-m}\mathbf 1_{C_n}(k),
\]
that is,
\[
    \eta_n(+,k)=\eta_{n,+}(k),
    \qquad
    \eta_n(-,k)=\eta_{n,-}(k).
\]
We use characteristic functions only for notational simplicity.  Replacing
these indicators by nonnegative smooth bumps supported in the same packets
gives the same lower bound.  Since the high-frequency packets carry the factor
\(n^{-m}\), while the low-frequency packet has size \(O(1)\), we have
\[
    \|\eta_n\|_{L^\infty_m}\lesssim 1 .
\]

Set
\[
    \theta_{\alpha,n}:=\frac\pi3-a_0n^{-1},
    \qquad
    k_{\alpha,n}:=
    n(\cos\theta_{\alpha,n},\sin\theta_{\alpha,n}),
    \qquad
    \varsigma_{\alpha,n}:=\frac{k_{\alpha,n}}{n}.
\]
We evaluate the output on the \(+\)-branch at \(k_{\alpha,n}\).  Then
\(k_{\alpha,n}\in A_n\), and all angular supports used below stay uniformly
away from \(\cos\theta=0\).  Hence the retained zero-frequency cut-off
\(\chi_{0,\varepsilon}\) is identically equal to one on the relevant region
for all sufficiently large \(n\).

\medskip
\noindent\textbf{Isolation of the resonant high--low--high channel.}
We first record why, at the fixed output point \(k_{\alpha,n}\in A_n\), the
only packet contribution which can produce the selected high--low--high corner
is the \((+,+,-)\) packet
\[
    k_\alpha\in A_n,\qquad
    k_\beta\in B_n,\qquad
    k_\gamma\in C_n,
\]
together with its \(\beta\leftrightarrow\gamma\) exchange.

On the \((+,+,-)\) branch, the branch-specific support gives
\[
    \eta_{n,+}\text{ supported in }A_n\cup B_n,
    \qquad
    \eta_{n,-}\text{ supported in }C_n.
\]
At the fixed output \(k_{\alpha,n}\in A_n\), the only high-frequency packet
which can almost cancel \(A_n\) under momentum conservation is \(C_n\); the
remaining vector is then of order one and points near \(\pi\).  Thus the
high--low--high packet assignment is
\[
    k_\beta\in B_n,\qquad k_\gamma\in C_n.
\]
The alternative \(k_\beta\in A_n\) is separated from the \((+,+,-)\) resonance
near this output direction and does not meet the selected corner.

The \((+,-,+)\) branch gives the exchanged packet assignment
\[
    k_\alpha\in A_n,\qquad
    k_\beta\in C_n,\qquad
    k_\gamma\in B_n.
\]
It is obtained from the previous one by interchanging the roles of
\(\beta\) and \(\gamma\).  By the \(\beta\leftrightarrow\gamma\) symmetry of
the integrand, it has the same leading sign on the exchanged packet and hence
cannot cancel the positive \((+,+,-)\) contribution.

We next exclude the other representative sign configurations.  On the
\((+,+,+)\) branch, both input factors are evaluated on the \(+\)-component
whenever they occur, and the \(+\)-component is supported only in
\(A_n\cup B_n\).  In particular there is no \(+\)-branch high packet near
\[
    (n,\tfrac{4\pi}{3}).
\]
Equivalently, the unit-scale polar parametrization of the \((+,+,+)\)
resonance manifold shows that a zero-mode endpoint can occur only when
\[
    \theta_\beta,\theta_\gamma
    \in
    \left\{\frac{\pi}{2},\frac{3\pi}{2}\right\},
\]
not in the low-packet direction \(\theta=\pi\).  Hence the \((+,+,+)\) branch
is disjoint from the selected high--low--high support.

Finally consider the \((+,-,-)\) branch.  Since the negative component of
\(\eta_n\) is supported only in \(C_n\), the term involving
\(\eta_{n,-}(k_\beta)\eta_{n,-}(k_\gamma)\) would require both
\[
    k_\beta,k_\gamma\in C_n,
\]
which is incompatible with momentum conservation at the fixed output
\(k_{\alpha,n}\in A_n\).  The mixed terms require only one negative mode to lie
in \(C_n\), but they would have to lie on the regular \((+,-,-)\) resonance
curve with
\[
    \theta_{\alpha,n}
    =
    \frac{\pi}{3}-a_0n^{-1}
    <
    \frac{\pi}{3}.
\]
By the unit-scale parametrization of the \((+,-,-)\) manifold, for
\[
    0\le \theta_\alpha\le \frac{\pi}{3}
\]
there is no nontrivial one-dimensional resonance curve.  Therefore the
\((+,-,-)\) branch has no regular co-area component through the selected output
direction and gives no contribution on the chosen supports.

Consequently, at the fixed output point \(k_{\alpha,n}\in A_n\), every
nonzero packet contribution to the full collision output is either supported
on the \((+,+,-)\) channel estimated below, supported on its
\(\beta\leftrightarrow\gamma\) symmetric counterpart, or is disjoint from the
chosen polar supports.

\medskip
\noindent\textbf{Growth of the isolated \((+,+,-)\) cusp contribution.}
Let \(I^n_{++-}\) denote the contribution of the \((+,+,-)\) branch restricted
to the component \(\sin\theta_\gamma\le0\) and to the interval
\[
    |\theta_\beta-\pi|<b_0n^{-1/2}.
\]
We estimate \(I^n_{++-}\) from below.  Write, at unit scale,
\[
    k_{\alpha,n}=n\varsigma_{\alpha,n},
    \qquad
    k_\beta=n\varsigma_\beta,
    \qquad
    k_\gamma=n\varsigma_\gamma,
    \qquad
    \varsigma_\gamma=-\varsigma_{\alpha,n}-\varsigma_\beta .
\]
On the \(\sin\theta_\gamma\le0\) branch of the \((+,+,-)\) resonance curve, use
the notation introduced earlier, evaluated at
\(\theta_\alpha=\theta_{\alpha,n}\):
\[
    p=\cos\theta_{\alpha,n}+\cos\theta_\beta,
    \qquad
    r=\sqrt{1-p^2},
\]
\[
    \widetilde f
    =
    -\sin\theta_\beta\,p-\cos\theta_\beta\,r,
    \qquad
    \widetilde F_1
    =
    \sin\theta_{\alpha,n}\,p+\cos\theta_{\alpha,n}\,r.
\]
Then
\[
    |\varsigma_\beta|
    =
    \frac{\widetilde F_1}{\widetilde f},
    \qquad
    |\varsigma_\gamma|
    =
    \frac{\sin(\theta_\beta-\theta_{\alpha,n})}{\widetilde f}.
\]
Let
\[
    \theta_\beta
    =
    \pi+\eta,
    \qquad
    |\eta|<b_0 n^{-1/2}.
\]
A Taylor expansion at
\[
    (\theta_\alpha,\theta_\beta)=\left(\frac\pi3,\pi\right)
\]
gives
\[
    \widetilde F_1
    =
    2a_0n^{-1}+O(\eta^2)+O(n^{-2})+O(|\eta|^3),
\]
while
\[
    \widetilde f\sim1,
    \qquad
    r\sim1,
    \qquad
    \sin(\theta_\beta-\theta_{\alpha,n})\sim1.
\]
Since \(|\eta|<b_0n^{-1/2}\), the error \(O(\eta^2)\) is
\(O(b_0^2n^{-1})\), and \(O(|\eta|^3)=O(b_0^3n^{-3/2})\).  We choose \(b_0\)
sufficiently small, depending on \(a_0\), so that these errors are dominated by
the leading term \(2a_0n^{-1}\).  Then, uniformly for
\[
    |\eta|<b_0n^{-1/2},
\]
we have
\begin{equation}
\label{eq:unbd-basic-asymptotics}
    \widetilde F_1\sim n^{-1},
    \qquad
    |\varsigma_\beta|\sim n^{-1},
    \qquad
    |\varsigma_\gamma|\sim1.
\end{equation}
In particular,
\[
    |k_\beta|=n|\varsigma_\beta|\sim1,
    \qquad
    |k_\gamma|=n|\varsigma_\gamma|=n+O(1),
\]
and the same expansion gives
\[
    \theta_\gamma=\frac{4\pi}{3}+O(n^{-1}).
\]
We now choose \(M\) sufficiently large, depending only on the constants in the
preceding estimates, so that
\[
    k_\beta(\theta_\beta)\in B_n,
    \qquad
    k_\gamma(\theta_\beta)\in C_n
\]
for every \(|\theta_\beta-\pi|<b_0n^{-1/2}\) and all sufficiently large \(n\).
This is the only use of the large constant \(M\).  Hence, on this whole
interval,
\begin{equation}
\label{eq:unbd-indicators-active}
    \mathbf 1_{A_n}(k_{\alpha,n})
    \mathbf 1_{B_n}(k_\beta(\theta_\beta))
    \mathbf 1_{C_n}(k_\gamma(\theta_\beta))
    =1.
\end{equation}
Consequently,
\[
    \eta_{n,+}(k_{\alpha,n})=n^{-m},
    \qquad
    \eta_{n,+}(k_\beta)=1,
    \qquad
    \eta_{n,-}(k_\gamma)=n^{-m}.
\]

We next estimate the radial cancellation factor in the interaction coefficient.
On the \((+,+,-)\) branch,
\[
\Gamma_{\alpha\beta\gamma}
    =
    \pi^2
    (\sin\theta_{\alpha,n}+
     \sin\theta_\beta-
     \sin\theta_\gamma)
    (|k_{\alpha,n}|+|k_\beta|-|k_\gamma|),
\]
so the corresponding unit-scale radial factor is
\[
    1+|\varsigma_\beta|-|\varsigma_\gamma|.
\]
Set
\[
    \rho:=|\varsigma_\beta|,
    \qquad
    d:=\theta_\beta-\theta_{\alpha,n}.
\]
Since
\[
    \varsigma_\gamma=-\varsigma_{\alpha,n}-\varsigma_\beta,
    \qquad
    |\varsigma_{\alpha,n}|=1,
\]
we have
\[
    |\varsigma_\gamma|^2
    =
    1+\rho^2+2\rho\cos d .
\]
On the selected packet,
\[
    d=\frac{2\pi}{3}+O(n^{-1/2}),
\]
and hence \(\cos d\le -c_0<0\) for all sufficiently large \(n\).  Since
\(\rho\sim n^{-1}\),
\[
\begin{aligned}
1-|\varsigma_\gamma|
&=
\frac{1-|\varsigma_\gamma|^2}{1+|\varsigma_\gamma|}
=
\frac{-2\rho\cos d-\rho^2}{1+|\varsigma_\gamma|}
\ge c\rho>0 .
\end{aligned}
\]
Therefore
\[
    1+\rho-|\varsigma_\gamma|
    =
    \rho+(1-|\varsigma_\gamma|)
    \ge c\rho .
\]
The reverse inequality follows from
\[
    \bigl||\varsigma_\gamma|-1\bigr|
    =
    \bigl||\varsigma_{\alpha,n}+\varsigma_\beta|-|\varsigma_{\alpha,n}|\bigr|
    \le |\varsigma_\beta|=\rho,
\]
and hence
\[
    1+\rho-|\varsigma_\gamma|\le 2\rho.
\]
Thus
\begin{equation}
\label{eq:unbd-cancellation-factor}
    1+|\varsigma_\beta|-|\varsigma_\gamma|
    \sim
    |\varsigma_\beta|
    \sim n^{-1}.
\end{equation}
The remaining angular factor in \(\Gamma\) is harmless:
\[
    (\sin\theta_{\alpha,n}+
     \sin\theta_\beta-
     \sin\theta_\gamma)^2
    \sim1.
\]

The nonlinear factor has a fixed positive sign on the selected packet.  Indeed,
\[
    \omega_\alpha=\cos\theta_{\alpha,n}=\frac12+o(1),
    \qquad
    \omega_\beta=\cos\theta_\beta=-1+o(1),
\]
and, since \(\sigma_\gamma=-1\),
\[
    \omega_\gamma=-\cos\theta_\gamma=\frac12+o(1).
\]
Using \eqref{eq:unbd-indicators-active}, the expression inside the parentheses
in the collision integrand satisfies
\[
\begin{aligned}
&\omega_\alpha\eta_{n,+}(k_\beta)\eta_{n,-}(k_\gamma)
+
\omega_\gamma\eta_{n,+}(k_{\alpha,n})\eta_{n,+}(k_\beta)
+
\omega_\beta\eta_{n,+}(k_{\alpha,n})\eta_{n,-}(k_\gamma)
\\
&\qquad
=
\left(\frac12+o(1)\right)n^{-m}
+
\left(\frac12+o(1)\right)n^{-m}
-
(1+o(1))n^{-2m}
\ge c n^{-m}.
\end{aligned}
\]
Since \(\omega_\alpha=\frac12+o(1)>0\), the full nonlinear factor is also
bounded below by \(c n^{-m}\).

Finally, by the \((+,+,-)\) co-area ratio computed earlier,
\[
    \left|
    \frac{
        \widetilde F_1
        \bigl(
            \cos\theta_{\alpha,n}\sin\theta_\beta
            -
            \sin\theta_{\alpha,n}\cos\theta_\beta
        \bigr)
    }{
        r\,\widetilde f^{3}
    }
    \right|
    \sim n^{-1}.
\]
Combining this with \eqref{eq:unbd-cancellation-factor}, the interaction
coefficient and the co-area factor together contribute
\[
    (1+|\varsigma_\beta|-|\varsigma_\gamma|)^2
    \left|
    \frac{
        \widetilde F_1
        \bigl(
            \cos\theta_{\alpha,n}\sin\theta_\beta
            -
            \sin\theta_{\alpha,n}\cos\theta_\beta
        \bigr)
    }{
        r\,\widetilde f^{3}
    }
    \right|
    \sim n^{-3}.
\]
The scaled co-area representation carries the prefactor
\[
    |k_{\alpha,n}|^4=n^4.
\]
Therefore
\[
\begin{aligned}
\langle k_{\alpha,n}\rangle^m I^n_{++-}
&\ge
c\,n^m n^4
\int_{|\theta_\beta-\pi|<b_0n^{-1/2}}
    n^{-3}n^{-m}
\,d\theta_\beta
\\
&\ge c n^{1/2}.
\end{aligned}
\]

By the isolation step above, every nonzero monomial in the full integrand at
the fixed output point \((+,k_{\alpha,n})\) is either supported on the
\((+,+,-)\) channel just estimated, supported on its
\(\beta\leftrightarrow\gamma\) symmetric counterpart, or is disjoint from the
chosen polar supports.  The symmetric counterpart has the same leading sign.
Hence the full collision output inherits the lower bound for \(I^n_{++-}\):
\[
\begin{aligned}
\langle k_{\alpha,n}\rangle^m
\left|
    \mathcal C^{(0)}(\eta_n)(+,k_{\alpha,n})
\right|
&\ge
\langle k_{\alpha,n}\rangle^m I^n_{++-}
\\
&\ge c n^{1/2}.
\end{aligned}
\]
Since
\[
    \|\eta_n\|_{L^\infty_m}\lesssim1,
\]
this proves that \(\mathcal C^{(0)}\) cannot be bounded from
\(L^\infty_m\) to \(L^\infty_m\).
\end{proof}

\begin{remark}[Geometric meaning of the \(|\cos\theta|=\tfrac12\) cut-off]
The obstruction above is not a coordinate singularity.  It is a cusp-type loss
of transversality of the \((+,+,-)\) resonance curve at the low-frequency
corner.

To see this, work at unit scale and write
\[
    \theta_\alpha=a,
    \qquad
    \theta_\beta=\pi+\eta,
    \qquad
    |\varsigma_\beta|=\rho,
    \qquad
    \varsigma_\gamma=-\varsigma_\alpha-\varsigma_\beta .
\]

On the \((+,+,-)\) branch, the frequency resonance
\[
    \Xi_1(\alpha,\beta,\gamma)=0
\]
reads
\[
    \cos a+\cos\theta_\beta-\cos\theta_\gamma=0.
\]
After imposing the momentum constraint through
\(\varsigma_\gamma=-\varsigma_\alpha-\varsigma_\beta\), a Taylor expansion near
\[
    (a,\eta,\rho)=\left(\frac\pi3,0,0\right)
\]
gives
\[
    \Xi_1(\alpha,\beta,\gamma)
    =
    (2\cos a-1)
    +
    \frac{\eta^2}{2}
    -
    \sin^2 a\,\rho
    +
    O(\rho|\eta|+\rho^2+|\eta|^4).
\]
Thus, when \(\cos a=1/2\),
\[
    \rho\sim\eta^2.
\]
At physical scale \(|k_\alpha|=n\), the condition \(|k_\beta|\sim1\) becomes
\(\rho\sim n^{-1}\), and hence
\[
    |\theta_\beta-\pi|\sim n^{-1/2}.
\]
This is larger than the generic transverse window \(n^{-1}\), and the extra
phase space is precisely the \(n^{1/2}\) growth in the counterexample.
Figure~\ref{fig:high-low-high-cusp} illustrates this high--low--high cusp.
\end{remark}

\begin{figure}[t]
    \centering
    \includegraphics[
        height=0.22\textheight,
        keepaspectratio
    ]{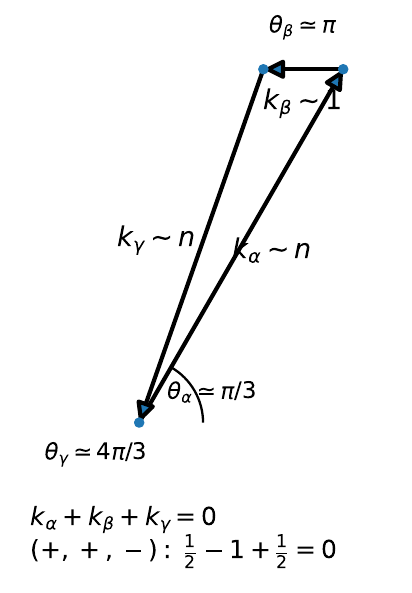}
    \hspace{0.07\textwidth}
    \raisebox{0.02\textheight}{
    \includegraphics[
        height=0.18\textheight,
        keepaspectratio
    ]{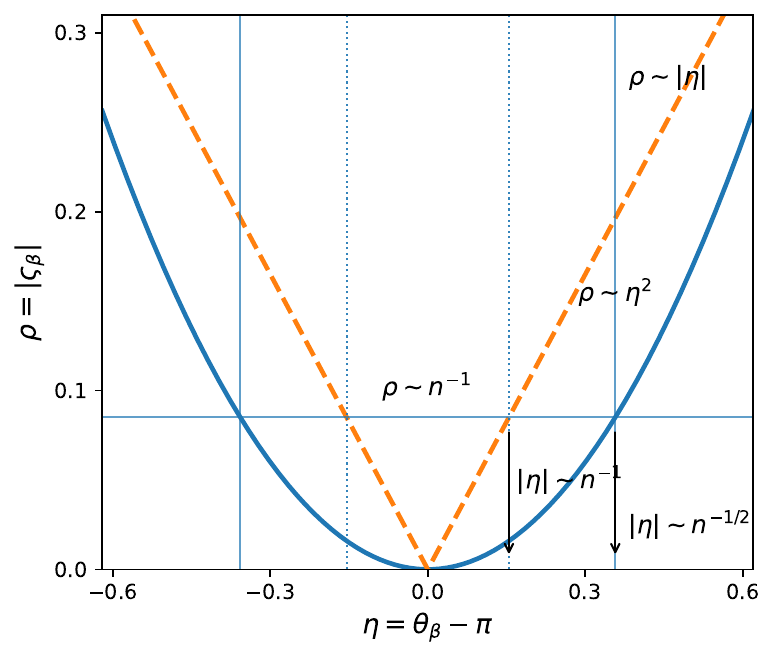}
    }

    \caption{
    The high--low--high cusp at \(|\cos\theta|=\tfrac12\).
    Left: the \((+,+,-)\) triad with directions
    \(\pi/3,\pi,4\pi/3\).  Right: the quadratic contact
    \(|\varsigma_\beta|\sim|\theta_\beta-\pi|^2\), which gives the angular
    window \(n^{-1/2}\) when \(|k_\alpha|=n\) and \(|k_\beta|\sim1\).
    }
    \label{fig:high-low-high-cusp}
\end{figure}
\begin{remark}
The particular counterexample constructed above would also be ruled out by an
additional symmetric cut-off near
\[
    |\cos\theta|=1,
\]
since the low mode in the \((+,+,-)\) packet satisfies
\[
    \theta_\beta\to\pi,
    \qquad
    |\cos\theta_\beta|\to1.
\]
By the symmetric placement of the cut-off in the kernel, excluding this low
mode would remove the whole triad.  This, however, is a different and stronger
angular truncation.  It removes the bad interaction indirectly through the
low horizontal mode.

The obstruction exhibited in the proof is the finite-frequency cusp of the
high--frequency pair:
\[
    |\cos\theta_\alpha|=\frac12,
    \qquad
    |\cos\theta_\gamma|=\frac12
\]
on the \((+,+,-)\) packet.  On the exchanged \((+,-,+)\) packet, the second
high-frequency mode is \(\beta\) instead of \(\gamma\).  Thus the natural
symmetric localization associated with this obstruction is the cut-off near
\[
    |\cos\theta|=\frac12
\]
for each triad member.
\end{remark}

\section{On the Zero-Frequency Cut-off}\label{sec:zero-frequency-cutoff}
\subsection{The zero-frequency cut-off and the singular exchange structure}
\label{subsec:zero-frequency-cutoff}

We close this section by recording the role of the angular localization away
from the zero-frequency set
\[
    \cos\theta=0.
\]
This localization is different in nature from the localization away from
\(|\cos\theta|=\tfrac12\).  The latter removes the high--low--high positive
unbounded contribution exhibited in Section~\ref{Section: Unboundedness}.
The zero-frequency set, by contrast, corresponds to modes for which
\[
    \omega_{\sigma,k}=\sigma\frac{k_1}{|k|}=0.
\]
Such modes do not carry a fast oscillatory phase.  Thus, from the viewpoint of
formal wave-turbulence derivations, they should not automatically be treated
as ordinary propagating waves.  A derivation which includes this set would
likely have to treat the zero-frequency component as a slow or non-oscillatory
mode, rather than as part of the purely propagating two-branch WKE considered
here.

There is also a distinct analytic issue.  If the cut-off near
\(\cos\theta=0\) is removed, the collision operator should not be expected to
remain a bounded vector field on the weighted space \(L^\infty_m\).  Near
zero-frequency endpoints of the resonance curves, the collision operator
contains unbounded high--high exchange rates.  These leading exchange terms do
not have the same sign-coherent positive-gain structure as the
\(|\cos\theta|=\tfrac12\) obstruction.  Instead, at the level of an
unweighted \(L^2\) energy calculation, their leading high-frequency pair
matrix has a dissipative--antisymmetric--remainder structure.  We record this
structure below only as a local structural observation.  It is not used in the
proof of the main well-posedness theorem, and it does not by itself remove the
zero-frequency cut-off in the \(L^\infty_m\) framework of this paper.

Throughout the following discussion we use the same scaled variables as before:
\[
    k_j=|k_\alpha|\varsigma_j,
    \qquad j\in\{\alpha,\beta,\gamma\},
    \qquad |\varsigma_\alpha|=1.
\]
The local co-area density in the angular parametrization is denoted by
\[
    d\nu_{\rm loc}
    :=
    \frac{
        \left|\dfrac{d\varsigma_\beta^1}{d\theta_\beta}\right|
    }{
        \left|\partial_{\varsigma_\beta^2}\widetilde\Xi_1\right|
    }
    d\theta_\beta .
\]
For instance, at a collapsed zero-mode endpoint
\(\rho:=|\varsigma_{\rm low}|\to0\), one has the local behavior
\[
    d\nu_{\rm loc}\sim \rho\,d\rho .
\]
This is the same co-area density used throughout the preceding parametrized
estimates; no new co-area convention is introduced here.

\paragraph{The \((+,+,+)\) branch.}

On the upper half of the unit-scale \((+,+,+)\) resonance curve,
\[
    \theta_\beta\in
    \left[\frac{\pi}{2},\,\pi+\theta_\alpha\right],
\]
the two limiting endpoints correspond to
\[
    k_\beta\to0
    \qquad\text{and}\qquad
    k_\gamma\to0,
\]
respectively.  More precisely,
\[
    \theta_\beta\to\frac{\pi}{2},
    \qquad
    |\varsigma_\beta|\to0,
\]
at the first endpoint, while
\[
    \theta_\beta\to\pi+\theta_\alpha,
    \qquad
    |\varsigma_\gamma|\to0,
\]
at the second.  These angular values are limiting directions along the
punctured resonance curve; the angle of the collapsed zero vector itself is
not defined.
Near \(k_\beta=0\), we freeze the low-frequency factor as a nonnegative
coefficient \(g_\beta\) and retain the variation only on the high-frequency
pair:
\[
    n_\beta\leadsto g_\beta,
    \qquad
    n_\alpha\leadsto h_\alpha,
    \qquad
    n_\gamma\leadsto h_\gamma.
\]
This defines the frozen-low-mode high-pair block considered below.

Keeping the terms that contain the frozen factor \(g_\beta\), the
\(\alpha\)-equation contributes
\[
\begin{aligned}
\omega_\alpha
\left(
    \omega_\alpha g_\beta h_\gamma
    +
    \omega_\gamma h_\alpha g_\beta
\right)
&=
g_\beta
\left(
    \omega_\alpha\omega_\gamma h_\alpha
    +
    \omega_\alpha^2h_\gamma
\right).
\end{aligned}
\]
By the symmetry between the two high legs, the corresponding contribution to
the \(\gamma\)-equation is
\[
    g_\beta
    \left(
        \omega_\gamma^2h_\alpha
        +
        \omega_\alpha\omega_\gamma h_\gamma
    \right).
\]
Thus, after suppressing the common nonnegative factor
\(\Gamma_{\alpha\beta\gamma}^2g_\beta\) and the local co--area density, the
principal high-pair block is
\[
    \mathsf M_{\alpha\gamma}
    :=
    \begin{pmatrix}
        \omega_\alpha\omega_\gamma & \omega_\alpha^2\\
        \omega_\gamma^2 & \omega_\alpha\omega_\gamma
    \end{pmatrix}.
\]
Equivalently, at the level of this frozen endpoint model,
\[
    \binom{(\partial_t h)_\alpha}{(\partial_t h)_\gamma}
    \sim
    \Gamma_{\alpha\beta\gamma}^2g_\beta\,
    \mathsf M_{\alpha\gamma}
    \binom{h_\alpha}{h_\gamma}.
\]
Here \(h_\gamma\) remains a function of the low-mode parameter:
\[
    h_\gamma
    =
    h\bigl(+,|k_\alpha|
       \varsigma_\gamma(\alpha,\rho)\bigr),
    \qquad
    \rho:=|\varsigma_\beta|.
\]

If \(g_\beta\equiv1\), this is the high-pair block associated with
\(M(h_\alpha,1,h_\gamma)\).  The relevant local energy object is therefore the
two-component pairing
\[
    \operatorname{Re}
    \left\langle
        \binom{h_\alpha}{h_\gamma},
        \mathsf M_{\alpha\gamma}
        \binom{h_\alpha}{h_\gamma}
    \right\rangle_{\mathbb R^2},
\]
rather than the scalar pairing with the \(\alpha\)-equation alone.

Although
\[
    \operatorname{spec}(\mathsf M_{\alpha\gamma})
    =
    \{0,\,2\omega_\alpha\omega_\gamma\}
\]
and the nonzero eigenvalue is negative near the collapsed endpoint,
\(\mathsf M_{\alpha\gamma}\) is generally not self-adjoint.  Hence the spectral
sign alone does not determine the sign of the standard Euclidean energy.
Using
\[
    \omega_\alpha+\omega_\beta+\omega_\gamma=0,
\]
one obtains
\[
\begin{aligned}
\operatorname{Re}
\left\langle
    \binom{h_\alpha}{h_\gamma},
    \mathsf M_{\alpha\gamma}
    \binom{h_\alpha}{h_\gamma}
\right\rangle_{\mathbb R^2}
&=
-\frac{(\omega_\alpha-\omega_\gamma)^2}{4}
    |h_\alpha-h_\gamma|^2
\\
&\quad+
\frac{\omega_\beta^2}{4}
    |h_\alpha+h_\gamma|^2 .
\end{aligned}
\]
Thus the standard high-pair energy is dissipative up to an
\(O(\omega_\beta^2)\) defect.  The decomposition below isolates precisely the
dissipative, antisymmetric, and lower-order parts of this identity.
\[
    \mathsf M_{\alpha\gamma}
    =
    \mathsf D_{\alpha\gamma}
    +
    \mathsf A_{\alpha\gamma}
    +
    \mathsf R_{\alpha\gamma},
\]
where
\[
    \mathsf D_{\alpha\gamma}
    :=
    (-\omega_\alpha\omega_\gamma)
    \begin{pmatrix}
        -1&1\\
        1&-1
    \end{pmatrix},
\]
\[
    \mathsf A_{\alpha\gamma}
    :=
    -\frac{\omega_\beta(\omega_\alpha-\omega_\gamma)}{2}
    \begin{pmatrix}
        0&1\\
        -1&0
    \end{pmatrix},
\]
and
\[
    \mathsf R_{\alpha\gamma}
    :=
    \frac{\omega_\beta^2}{2}
    \begin{pmatrix}
        0&1\\
        1&0
    \end{pmatrix}.
\]
Indeed,
\[
\begin{aligned}
\mathsf M_{\alpha\gamma}-\mathsf D_{\alpha\gamma}
&=
\begin{pmatrix}
0&-\omega_\alpha\omega_\beta\\
-\omega_\gamma\omega_\beta&0
\end{pmatrix}
\\
&=
\mathsf A_{\alpha\gamma}
+
\mathsf R_{\alpha\gamma}.
\end{aligned}
\]
The first term is dissipative:
\[
\operatorname{Re}
\left\langle
    \binom{h_\alpha}{h_\gamma},
    \mathsf D_{\alpha\gamma}
    \binom{h_\alpha}{h_\gamma}
\right\rangle
=
-(-\omega_\alpha\omega_\gamma)
|h_\alpha-h_\gamma|^2
\le0.
\]
The second term is skew-symmetric and therefore invisible in the unweighted
energy pairing:
\[
\operatorname{Re}
\left\langle
    \binom{h_\alpha}{h_\gamma},
    \mathsf A_{\alpha\gamma}
    \binom{h_\alpha}{h_\gamma}
\right\rangle
=0.
\]
The third term is the \(O(\omega_\beta^2)\) symmetric defect appearing in the
quadratic-form identity above.

The nonlinear term
\[
    \omega_\alpha\omega_\beta n_\alpha n_\gamma
\]
is not included in this principal high-pair paralinearization.  It is
quadratic in the two high variables and carries the additional small factor
\(\omega_\beta\).  Likewise, the full Fr\'echet linearization would contain
terms in which the variation is placed on the low leg.  Both belong to the
terms omitted from the frozen high-pair endpoint model.

Near \(k_\beta=0\), one has
\[
    |\omega_\beta|\lesssim\rho,
    \qquad
    d\nu_{\mathrm{loc}}\sim\rho\,d\rho,
    \qquad
    \rho=|\varsigma_\beta|.
\]
Moreover, the nonvanishing part of the unit-scale interaction coefficient
satisfies
\[
\pi^4
(\sin\theta_\alpha+\sin\theta_\beta+\sin\theta_\gamma)^2
(1+|\varsigma_\beta|+|\varsigma_\gamma|)^2
\longrightarrow
4\pi^4 .
\]
Consequently, the principal high-pair exchange has the weighted
vector-field size
\[
    |k_\alpha|^4
    \int_0^{\rho_0}
        \rho\,g_\beta(|k_\alpha|\rho)\,d\rho,
\]
where the limiting angular direction is suppressed in the notation.  If the
low-mode background is \(O(1)\) near the origin, this quantity may grow like
\[
    O(|k_\alpha|^2).
\]
This is the sense in which the zero-frequency exchange does not define a
bounded vector field in the weighted \(L^\infty_m\) framework used here.

The matrix remainder has an additional factor
\[
    \omega_\beta^2=O(\rho^2).
\]
If
\[
    g_\beta(|k_\alpha|\rho)
    \lesssim
    \langle|k_\alpha|\rho\rangle^{-m},
    \qquad
    m>4,
\]
then its model contribution is bounded:
\[
\begin{aligned}
|k_\alpha|^4
\int_0^{\rho_0}
    \rho\,\omega_\beta^2
    g_\beta(|k_\alpha|\rho)\,d\rho
&\lesssim
|k_\alpha|^4
\int_0^{\rho_0}
    \rho^3
    \langle|k_\alpha|\rho\rangle^{-m}\,d\rho
\\
&=
\int_0^{|k_\alpha|\rho_0}
    y^3\langle y\rangle^{-m}\,dy
\lesssim1.
\end{aligned}
\]
This estimate concerns only the matrix remainder in the frozen leading
endpoint model.  It does not establish boundedness of all remaining terms in
the full non-cutoff collision operator.  In particular, the low-leg
variations, the quadratic high-pair term carrying \(\omega_\beta\), and the
errors in the endpoint expansions of the coefficient and co--area density are
not analyzed here.

The endpoint \(k_\gamma=0\) is obtained by exchanging \(\beta\) and
\(\gamma\).  The low mode is then \(\gamma\), the high pair is
\((\alpha,\beta)\), and the same construction applies with low frequency
\(\omega_\gamma\).
%

\paragraph{The remaining resonant branches.}

On the \((+,+,-)\) branch, the collapsed endpoint
\[
    \varsigma_\gamma=0
\]
is a zero-frequency endpoint analogous to those of the \((+,+,+)\) branch.
The low mode is now \(\gamma\), while the high pair is
\((\alpha,\beta)\); hence the same frozen high-pair structure applies after
exchanging \(\beta\) and \(\gamma\).  This endpoint is excluded by the
localization near \(\cos\theta=0\).

The other collapsed point on the \((+,+,-)\) branch is
\[
    \varsigma_\beta=0,
\]
which is the finite-frequency high--low--high cusp.  The cut-off near
\[
    |\cos\theta|=\frac12
\]
removes its critical quadratic degeneration.  On the remaining noncritical
part of the cusp, the interaction coefficient satisfies
\[
    \bigl|1+|\varsigma_\beta|-|\varsigma_\gamma|\bigr|
    \lesssim
    |\varsigma_\beta|.
\]
Together with the local co--area behavior
\[
    d\nu_{\rm loc}
    \lesssim
    |\varsigma_\beta|\,d|\varsigma_\beta|,
\]
this gives the integrable endpoint factor
\[
    |\varsigma_\beta|^3\,d|\varsigma_\beta|.
\]
Thus the noncritical cusp is controlled after the critical
\(|\cos\theta|=\tfrac12\) region has been removed.  The exchanged
\((+,-,+)\) branch is identical after interchanging \(\beta\) and
\(\gamma\).

The collapsed endpoints of the \((+,-,-)\) branch are controlled in the same
way.  Near, for example, \(\varsigma_\beta=0\),
\[
    \bigl|1-|\varsigma_\beta|-|\varsigma_\gamma|\bigr|
    \lesssim
    |\varsigma_\beta|,
\]
and the same co--area estimate again produces an integrable factor of order
\[
    |\varsigma_\beta|^3\,d|\varsigma_\beta|.
\]
The endpoint with \(\varsigma_\gamma=0\) follows by central symmetry.

We therefore retain both angular localizations in the present paper.  The
cut-off near \(\cos\theta=0\) removes the zero-frequency exchange endpoints,
whereas the cut-off near \(|\cos\theta|=\frac12\) removes the critical
quadratic high--low--high cusp.  The remaining collapsed endpoints are
controlled by the cancellations above.

The discussion above explains why the zero-frequency cut-off is retained in
the present work.  The frozen high-pair block has a favorable energy-level
structure, but the leading zero-frequency exchange does not define a bounded
vector field on \(L^\infty_m\), and the remaining terms of the full
non-cutoff operator are not controlled here.  Removing this cut-off would
likely require a different framework, combining an energy-level treatment of
the leading exchange with estimates for the low-leg terms, endpoint errors,
and weight commutators.

This would not remove the distinct cut-off near
\[
    |\cos\theta|=\frac12,
\]
which remains necessary in the present \(L^\infty_m\) framework because of
the critical quadratic high--low--high cusp.

%


\section{Related Hamiltonian Boussinesq models}
\label{sec:related-Hamiltonian-models}

The full non-rotating Boussinesq model studied above has two distinct angular
obstructions: the zero-frequency set
\[
    \cos\theta=0
\]
and the critical quadratic high--low--high cusp at
\[
    |\cos\theta|=\frac12,
\]
illustrated in Figure~\ref{fig:high-low-high-cusp}.  We now compare this
geometry with two related Hamiltonian Boussinesq models.

For the rotating--stratified system, rotation opens a frequency gap and
removes the zero-frequency wave degeneration.  We determine the range in
which the three-wave resonant manifold is nonempty, classify the exceptional
critical point of the reduced resonance function, and identify the rotating
analogue of the finite-frequency high--low--high cusp.  A complete
function-space theory for the rotating collision operator is left for future
work.

For the hydrostatic Euler--Boussinesq system, we obtain a complete cut-off
estimate.  The hydrostatic dispersion retains the zero-frequency set
\(\cos\theta=0\) and has an additional pole at \(\sin\theta=0\).  After
localizing away from these two axes, we prove that the collision operator is
bounded on \(L^\infty_m\) for \(m>4\), and hence obtain local well-posedness.
In contrast with the full Boussinesq model, the hydrostatic collapsed
high--low--high geometry is linear rather than quadratic, so no localization
near \(|\cos\theta|=\frac12\) is required.

\subsection{The rotating--stratified Boussinesq model}
\label{subsec:rotating-model}

With the buoyancy frequency normalized to one, the two-dimensional
rotating--stratified Boussinesq system may be written in terms of the
stream function \(A\), the transverse velocity \(v\), and the vertical
displacement \(\zeta\) as
\begin{equation}
\label{eq:rotating-Boussinesq-Hamiltonian}
\begin{cases}
\partial_t(-\Delta A)
+\{A,-\Delta A\}
-\partial_1\zeta
-f\partial_2v
=0,
\\[0.4em]
\partial_tv+\{A,v\}-f\partial_2A=0,
\\[0.4em]
\partial_t\zeta+\{A,\zeta\}-\partial_1A=0.
\end{cases}
\end{equation}
This is a non-canonical Hamiltonian system with Hamiltonian
\begin{equation}
\label{eq:rotating-Boussinesq-energy}
    H_f
    =
    \frac12
    \int
    \left(
        |\nabla A|^2+v^2+\zeta^2
    \right)\,dx .
\end{equation}

The linearized system has a zero-frequency balanced branch and two
inertia--gravity wave branches.  The linear potential vorticity is
\[
    q_{\mathrm{lin}}
    =
    \partial_1v-f\partial_2\zeta .
\]
Following Shavit--B\"uhler--Shatah, one removes the balanced component by
restricting to the wave manifold
\[
    q_{\mathrm{lin}}=0
\]
and retaining only the two wave branches
\cite{shavit2026rotating}.  Their dispersion relation is
\[
    \omega_{\sigma,k}^{f}
    =
    \sigma\Omega_f(k),
    \qquad
    \Omega_f(k)
    =
    \sqrt{
        \frac{(k_1)^2+f^2(k_2)^2}
             {(k_1)^2+(k_2)^2}
    }
    =
    \sqrt{\cos^2\theta+f^2\sin^2\theta}.
\]
Consequently,
\[
    f\le\Omega_f(\theta)\le1.
\]
In particular, for every \(f>0\),
\[
    \Omega_f\left(\frac{\pi}{2}\right)=f>0,
\]
so the vertical direction is no longer a zero-frequency wave direction.

The corresponding rotating WKE has the same symmetric three-wave structure as
the non-rotating equation, with the frequency and interaction coefficient
replaced by their rotating counterparts.  Since only the resonance geometry
is used below, we do not record the full collision operator.  The rotating
resonant set is determined by
\[
    k_\alpha+k_\beta+k_\gamma=0
\]
and
\[
    \sigma_\alpha\Omega_f(k_\alpha)
    +
    \sigma_\beta\Omega_f(k_\beta)
    +
    \sigma_\gamma\Omega_f(k_\gamma)
    =
    0.
\]

\begin{lemma}[Disappearance of rotating three-wave resonances]
\label{lem:no-rotating-resonance-large-f}
Let
\[
    \frac12\le f<1.
\]
Then there is no resonant triad satisfying
\[
    k_\alpha,k_\beta,k_\gamma\neq0.
\]
Equivalently, the nontrivial three-wave resonant manifold is empty.
\end{lemma}

\begin{proof}
A resonant triad cannot have all three frequencies with the same sign.
After relabeling, any mixed-sign resonance may therefore be written as
\[
    \Omega_f(\theta_\alpha)
    =
    \Omega_f(\theta_\beta)
    +
    \Omega_f(\theta_\gamma).
\]
If \(f>\frac12\), then
\[
    \Omega_f(\theta_\alpha)
    \le1
    <
    2f
    \le
    \Omega_f(\theta_\beta)
    +
    \Omega_f(\theta_\gamma),
\]
which is impossible.

It remains to consider \(f=\frac12\).  Equality would require
\[
    \Omega_f(\theta_\alpha)=1,
    \qquad
    \Omega_f(\theta_\beta)
    =
    \Omega_f(\theta_\gamma)
    =
    \frac12.
\]
The first condition forces \(k_\alpha\) to be horizontal, while the latter two
force \(k_\beta\) and \(k_\gamma\) to be vertical.  Hence
\(k_\beta+k_\gamma\) is vertical and cannot equal the nonzero horizontal
vector \(-k_\alpha\).  This contradicts
\[
    k_\alpha+k_\beta+k_\gamma=0.
\]
Therefore the nontrivial resonant manifold is empty also at
\(f=\frac12\).
\end{proof}

Thus the nontrivial rotating resonance manifold is confined to
\[
    0<f<\frac12.
\]

For completeness, after eliminating
\[
    k_\gamma=-k_\alpha-k_\beta,
\]
define the reduced rotating resonance function
\[
    \widetilde\Xi_{1,f}(k_\beta)
    :=
    \sigma_\alpha\Omega_f(k_\alpha)
    +
    \sigma_\beta\Omega_f(k_\beta)
    +
    \sigma_\gamma\Omega_f(-k_\alpha-k_\beta).
\]
The corresponding co--area Jacobian is
\[
    J_f(k_\beta)
    :=
    \left|
        \nabla_{k_\beta}
        \widetilde\Xi_{1,f}(k_\beta)
    \right|.
\]
Since
\[
    k_\gamma=-k_\alpha-k_\beta,
\]
one has
\[
    \nabla_{k_\beta}\widetilde\Xi_{1,f}
    =
    \sigma_\beta\nabla\Omega_f(k_\beta)
    -
    \sigma_\gamma\nabla\Omega_f(k_\gamma),
\]
where
\[
    \nabla\Omega_f(k)
    =
    \frac{1-f^2}{\Omega_f(k)|k|^4}
    \begin{pmatrix}
        k_1(k_2)^2\\
        -(k_1)^2k_2
    \end{pmatrix}.
\]

\begin{lemma}[Critical points of the rotating reduced resonance function]
\label{lem:rotating-coarea-critical-points}
Assume
\[
    0<f<\frac12,
\]
and fix \(k_\alpha\neq0\).  Let
\[
    k_\gamma=-k_\alpha-k_\beta,
\]
and suppose that
\[
    \widetilde\Xi_{1,f}(k_\beta)=0,
    \qquad
    k_\beta,k_\gamma\neq0.
\]
If
\[
    k_{\beta,1}k_{\beta,2}\neq0,
    \qquad
    k_{\gamma,1}k_{\gamma,2}\neq0,
\]
then
\[
    J_f(k_\beta)>0.
\]

If \(J_f(k_\beta)=0\), then both group velocities vanish, so
\(k_\beta\) and \(k_\gamma\) lie on the coordinate axes.  Up to permutation,
coordinate reflections, and simultaneous reversal of all branch signs, the
only possible finite-frequency critical configuration has
\[
    (\omega_\alpha,\omega_\beta,\omega_\gamma)
    =
    (1-f,f,-1).
\]
For fixed \(k_\alpha\), this is an isolated nondegenerate critical point of
\(\widetilde\Xi_{1,f}\), rather than a one-dimensional co--area singularity.
\end{lemma}

\begin{proof}
By the formula preceding the lemma,
\[
    J_f(k_\beta)=0
\]
is equivalent to
\[
    \sigma_\beta\nabla\Omega_f(k_\beta)
    =
    \sigma_\gamma\nabla\Omega_f(k_\gamma).
\]
The explicit expression for \(\nabla\Omega_f\) shows that, for \(k\neq0\),
\[
    \nabla\Omega_f(k)=0
    \quad\Longleftrightarrow\quad
    k_1k_2=0.
\]
Moreover,
\[
    k\cdot\nabla\Omega_f(k)=0,
\]
as expected from the degree-zero homogeneity of \(\Omega_f\).

Suppose first that both group velocities are nonzero.  Their signed equality
implies that they are parallel.  Since
\(\nabla\Omega_f(k_\beta)\perp k_\beta\) and
\(\nabla\Omega_f(k_\gamma)\perp k_\gamma\), it follows that
\[
    k_\beta\parallel k_\gamma.
\]
The momentum constraint then implies that
\(k_\alpha,k_\beta,k_\gamma\) are all collinear.  Since \(\Omega_f\) is even
and homogeneous of degree zero, the three frequency magnitudes are equal:
\[
    \Omega_f(k_\alpha)
    =
    \Omega_f(k_\beta)
    =
    \Omega_f(k_\gamma).
\]
The frequency resonance would therefore require
\[
    \sigma_\alpha+\sigma_\beta+\sigma_\gamma=0,
\]
which is impossible for three signs in \(\{\pm1\}\).  Hence no critical point
with nonzero group velocities belongs to the nontrivial resonant manifold.

It follows that \(J_f(k_\beta)=0\) can occur only if
\[
    \nabla\Omega_f(k_\beta)
    =
    \nabla\Omega_f(k_\gamma)
    =
    0.
\]
Thus both \(k_\beta\) and \(k_\gamma\) lie on coordinate axes.  If they lie on
the same axis, then the triad is collinear, which has already been excluded.
Therefore one vector must be vertical and the other horizontal.  Their
frequency magnitudes are respectively
\[
    f
    \qquad\text{and}\qquad
    1.
\]
Since \(1\) is the largest frequency, the only possible resonance relation is
\[
    1=f+\Omega_f(k_\alpha),
\]
and hence
\[
    \Omega_f(k_\alpha)=1-f.
\]
Up to the stated symmetries, we may take
\[
    (\sigma_\alpha,\sigma_\beta,\sigma_\gamma)=(+,+,-)
\]
and
\[
    k_\alpha=(k_{\alpha,1},k_{\alpha,2}),
    \qquad
    k_{\beta,0}=(0,-k_{\alpha,2}),
    \qquad
    k_{\gamma,0}=(-k_{\alpha,1},0).
\]
This gives
\[
    (\omega_\alpha,\omega_\beta,\omega_\gamma)
    =
    (1-f,f,-1).
\]
Since \(0<f<1/2\), the relation
\(\Omega_f(k_\alpha)=1-f\) implies
\[
    \cos^2\theta_\alpha
    =
    \frac{1-2f}{1-f^2},
\]
so both \(k_{\alpha,1}\) and \(k_{\alpha,2}\) are nonzero.

It remains to verify that this critical point is isolated.  For the
representative sign configuration,
\[
    \widetilde\Xi_{1,f}(k_\beta)
    =
    \Omega_f(k_\alpha)
    +
    \Omega_f(k_\beta)
    -
    \Omega_f(-k_\alpha-k_\beta).
\]
At \(k_{\beta,0}\),
\[
    \nabla_{k_\beta}
    \widetilde\Xi_{1,f}(k_{\beta,0})=0.
\]
A direct calculation gives
\[
    D^2\Omega_f(0,-k_{\alpha,2})
    =
    (1-f^2)
    \begin{pmatrix}
        \dfrac{1}{f\,k_{\alpha,2}^2} & 0\\[1.1ex]
        0 & 0
    \end{pmatrix},
\]
and
\[
    D^2\Omega_f(-k_{\alpha,1},0)
    =
    (1-f^2)
    \begin{pmatrix}
        0 & 0\\[0.6ex]
        0 & -\dfrac{1}{k_{\alpha,1}^2}
    \end{pmatrix}.
\]
Therefore
\[
\begin{aligned}
D_{k_\beta}^2
\widetilde\Xi_{1,f}(k_{\beta,0})
&=
D^2\Omega_f(k_{\beta,0})
-
D^2\Omega_f(k_{\gamma,0})
\\
&=
(1-f^2)
\begin{pmatrix}
    \dfrac{1}{f\,k_{\alpha,2}^2} & 0\\[1.1ex]
    0 & \dfrac{1}{k_{\alpha,1}^2}
\end{pmatrix}.
\end{aligned}
\]
This matrix is positive definite.  Hence \(k_{\beta,0}\) is an isolated
nondegenerate local minimum of the reduced resonance function.
\end{proof}

\paragraph{Comparison with the non-rotating zero-frequency geometry.}

The frequency gap
\[
    \Omega_f(\theta)\ge f>0
\]
removes the zero-frequency wave degeneration of the non-rotating model.
At the exceptional configuration identified in
Lemma~\ref{lem:rotating-coarea-critical-points}, the co--area Jacobian
vanishes, but the Hessian of the reduced resonance function is positive
definite.  More precisely, if \(k_{\beta,0}\) denotes the critical point and
\[
    \mathsf H_f
    :=
    D_{k_\beta}^2
    \widetilde\Xi_{1,f}(k_{\beta,0}),
\]
then
\[
    \widetilde\Xi_{1,f}(k_{\beta,0}+q)
    =
    \frac12 q^{\mathsf T}\mathsf H_f q
    +
    O(|q|^3),
    \qquad
    \mathsf H_f>0.
\]
Consequently, for \(q\) sufficiently small,
\[
    \widetilde\Xi_{1,f}(k_{\beta,0}+q)
    \ge c_f|q|^2,
\]
and the zero level set consists locally only of \(k_{\beta,0}\).  In
particular, no one-dimensional resonant branch passes through this critical
point.

For sufficiently small regular values \(\lambda>0\), the portion of the
level set \(\{\widetilde\Xi_{1,f}=\lambda\}\) lying in a sufficiently small
neighborhood of \(k_{\beta,0}\) has length comparable to
\(\sqrt{\lambda}\), while
\[
    J_f\sim_f\sqrt{\lambda}
\]
there.  Hence its local co--area mass remains uniformly bounded as
\(\lambda\downarrow0\).

Consequently, there exist \(\delta_f,\lambda_f>0\) such that, for every
regular value \(0<\lambda<\lambda_f\),
\[
\begin{aligned}
\int_{
    \{\widetilde\Xi_{1,f}=\lambda\}
    \cap B_{\delta_f}(k_{\beta,0})
}
\frac{d\mathcal H^1}{J_f}
&\lesssim_f
\lambda^{-1/2}
\mathcal H^1\!\left(
    \{\widetilde\Xi_{1,f}=\lambda\}
    \cap B_{\delta_f}(k_{\beta,0})
\right)
\\
&\lesssim_f 1 .
\end{aligned}
\]
Thus the integrated local co--area contribution remains uniformly bounded as
\(\lambda\downarrow0\).

This is qualitatively different from the non-collapsed zero-frequency
critical geometry of the non-rotating Boussinesq model.  Fix
\[
    k_\alpha=(0,1),
    \qquad
    k_{\beta,0}=k_{\gamma,0}
    =
    \left(0,-\frac12\right),
\]
and write
\[
    u:=k_{\beta,1},
    \qquad
    v:=k_{\beta,2}+\frac12.
\]
Then
\[
    k_\beta
    =
    \left(u,v-\frac12\right),
    \qquad
    k_\gamma
    =
    \left(-u,-v-\frac12\right),
\]
and the reduced non-rotating resonance function is
\[
    \widetilde\Xi_1(u,v)
    =
    \frac{u}{
        \sqrt{u^2+(v-\frac12)^2}
    }
    -
    \frac{u}{
        \sqrt{u^2+(v+\frac12)^2}
    }.
\]
On the region where \(k_\beta,k_\gamma\neq0\), its zero level set is
\[
    u=0
    \qquad\text{or}\qquad
    v=0.
\]
Moreover,
\[
    \widetilde\Xi_1(u,v)
    =
    8uv
    +
    O\!\left(|(u,v)|^3\right).
\]
Thus the Hessian is indefinite, and the critical point is a saddle at which
two one-dimensional resonant branches cross.

Along either branch, if
\[
    \varrho:=\sqrt{u^2+v^2},
\]
then
\[
    J
    =
    \left|
        \nabla\widetilde\Xi_1
    \right|
    \sim
    \varrho,
    \qquad
    d\mathcal H^1
    \sim
    d\varrho.
\]
Hence the co--area factor degenerates logarithmically:
\[
    \frac{d\mathcal H^1}{J}
    \sim
    \frac{d\varrho}{\varrho}.
\]
Thus the non-rotating critical point lies on two resonant branches carrying a
logarithmically degenerate co--area factor, whereas the rotating exceptional
triad is an isolated finite-frequency critical configuration.

\paragraph{The rotating high--low--high cusp.}

For
\[
    0<f<\frac12,
\]
rotation does not remove the distinct finite-frequency high--low--high cusp.
As shown in
Lemma~\ref{lem:rotating-high-low-high-cusp}, its critical location is
determined by
\[
    \Omega_f(\theta_\alpha)=\frac12.
\]
Equivalently,
\[
    |\cos\theta_\alpha|
    =
    \sqrt{
        \frac{\frac14-f^2}{1-f^2}
    }.
\]
At this angle, the two limiting directions of the collapsed low mode
coalesce, and the resonance curve reaches the collapsed point quadratically.

Consequently, rotation removes the analogue of the non-rotating
zero-frequency localization near
\[
    \cos\theta=0,
\]
but it does not remove the separate finite-frequency high--low--high cusp.

\paragraph{Outlook for the rotating WKE.}

For each fixed
\[
    0<f<\frac12,
\]
the preceding analysis suggests that the natural rotating localization is a
finite-frequency cut-off near
\[
    \Omega_f(\theta)=\frac12.
\]
More precisely, one may choose a symmetric angular function
\(\chi_{\varepsilon,f}\) satisfying
\[
    \chi_{\varepsilon,f}(\theta)=0
    \qquad\text{when}\qquad
    \left|
        \Omega_f(\theta)-\frac12
    \right|
    <\varepsilon,
\]
and place one factor on each member of the resonant triad.  Denote the
corresponding rotating collision operator by
\(\mathcal C_\varepsilon^f\).

The frequency gap removes the zero-frequency obstruction, while the
exceptional critical point in
Lemma~\ref{lem:rotating-coarea-critical-points} is isolated and does not lie
on a singular one-dimensional resonant branch.  Motivated by the
non-rotating estimates, we therefore expect that, for every fixed
\(f\in(0,\frac12)\) and \(m>4\),
\[
    \left\|
        \mathcal C_\varepsilon^f(n)
    \right\|_{L^\infty_m}
    \lesssim_{f,\varepsilon,m}
    \|n\|_{L^\infty_m}^2.
\]
Such an estimate would yield local well-posedness of the rotating cut-off WKE
by the same Banach-space argument as in the non-rotating case.

A proof would require a complete parametrization of the rotating resonant
branches and uniform estimates for the rotating interaction coefficient and
co--area factor in all endpoint and large-wave-number regimes.  These
calculations are beyond the scope of the present paper and are left for
future work.  No uniformity as \(f\downarrow0\) is expected, since the
frequency gap closes and the non-rotating zero-frequency geometry is recovered
in this limit.

\subsection{The hydrostatic Euler--Boussinesq model}
\label{subsec:hydrostatic-model}

The hydrostatic model originates from a different Hamiltonian PDE.  In terms
of the stream function \(\varphi\) and density perturbation \(\vartheta\), it
takes the form
\begin{equation}
\label{eq:hydrostatic-Hamiltonian-main-text}
\begin{cases}
\partial_t\partial_2^2\varphi-\partial_1\vartheta
=
\{-\partial_2^2\varphi,\varphi\},
\\[0.4em]
\partial_t\vartheta+\partial_1\varphi
=
\{-\vartheta,\varphi\}.
\end{cases}
\end{equation}
Its Hamiltonian is
\[
    H^h
    =
    \frac12
    \int
    \left(
        |\partial_2\varphi|^2+\vartheta^2
    \right)\,dx,
\]
and its two linear wave branches have dispersion relation
\[
    \omega_{\sigma,k}^h
    =
    \sigma\frac{k_1}{|k_2|},
    \qquad
    k_2\neq0.
\]

The hydrostatic WKE considered here is derived independently from
\eqref{eq:hydrostatic-Hamiltonian-main-text}; it is not obtained by replacing
the dispersion relation in an already closed Boussinesq WKE.  The formal
PDE-to-WKE derivation is given in
Appendix~\ref{app:hydrostatic-formal-derivation}.

The hydrostatic dispersion has two distinct angular degeneracies:
\[
    \cos\theta=0
    \quad\Longrightarrow\quad
    \omega_{\sigma,k}^h=0,
\]
whereas
\[
    \sin\theta=0
    \quad\Longrightarrow\quad
    |\omega_{\sigma,k}^h|=\infty.
\]
We therefore choose
\[
    0<\varepsilon_0,\varepsilon_h<1,
    \qquad
    \varepsilon_0^2+\varepsilon_h^2<1,
\]
and introduce a smooth symmetric angular cut-off
\[
    \chi_{\varepsilon_0,\varepsilon_h}^h(\theta)
\]
supported in
\[
    |\cos\theta|\ge\varepsilon_0,
    \qquad
    |\sin\theta|\ge\varepsilon_h.
\]
One factor is placed on each member of the resonant triad.  We denote the
corresponding collision operator by
\(\mathcal C_{\varepsilon_0,\varepsilon_h}^h\).
No localization near
\[
    |\cos\theta|=\frac12
\]
is imposed.

\begin{proposition}[Hydrostatic resonance geometry and cut-off estimate]
\label{prop:hydrostatic-cutoff-estimate}
Let \(m>4\).  Then
\begin{equation}
\label{eq:hydrostatic-cutoff-estimate}
    \left\|
        \mathcal C_{\varepsilon_0,\varepsilon_h}^h(n^h)
    \right\|_{L^\infty_m}
    \lesssim_{\varepsilon_0,\varepsilon_h,m}
    \|n^h\|_{L^\infty_m}^2.
\end{equation}
Consequently, the cut-off hydrostatic WKE is locally well posed in
\(L^\infty_m\).
\end{proposition}

\begin{proof}
We record only the geometric features that differ from the full Boussinesq
analysis.  By symmetry, it is enough to take
\[
    \sigma_\alpha=+1,
    \qquad
    0<\theta_\alpha<\frac{\pi}{2},
    \qquad
    \varsigma_\alpha
    =
    (\cos\theta_\alpha,\sin\theta_\alpha).
\]
The unit-scale resonant set is parametrized by
\[
    \varsigma_\beta(z)
    =
    \bigl(
        \cos\theta_\alpha X(z),
        \sin\theta_\alpha z
    \bigr),
\]
and
\[
    \varsigma_\gamma(z)
    =
    \bigl(
        -\cos\theta_\alpha(1+X(z)),
        -\sin\theta_\alpha(1+z)
    \bigr).
\]
The hydrostatic frequency resonance reduces to
\[
    1
    +
    \sigma_\beta\frac{X(z)}{|z|}
    -
    \sigma_\gamma\frac{1+X(z)}{|1+z|}
    =
    0.
\]
Solving this equation on
\[
    z>0,
    \qquad
    -1<z<0,
    \qquad
    z<-1,
\]
gives
\begin{equation}
\label{eq:hydrostatic-X-table}
\begin{array}{c|c|c|c}
&
z>0
&
-1<z<0
&
z<-1
\\[0.3em]\hline
(+,+,+)
&
-z^2
&
\displaystyle\frac{z^2}{1+2z}
&
z(z+2)
\\[0.9em]
(+,+,-)
&
\displaystyle-\frac{z(z+2)}{1+2z}
&
z(z+2)
&
\displaystyle\frac{z^2}{1+2z}
\\[1em]
(+,-,-)
&
z(z+2)
&
\displaystyle-\frac{z(z+2)}{1+2z}
&
-z^2 .
\end{array}
\end{equation}
For the same-sign configurations, the middle branch is understood away from
\(z=-\frac12\).

\begin{figure}[t]
    \centering
    \includegraphics[
        width=\textwidth,
        height=0.58\textheight,
        keepaspectratio
    ]{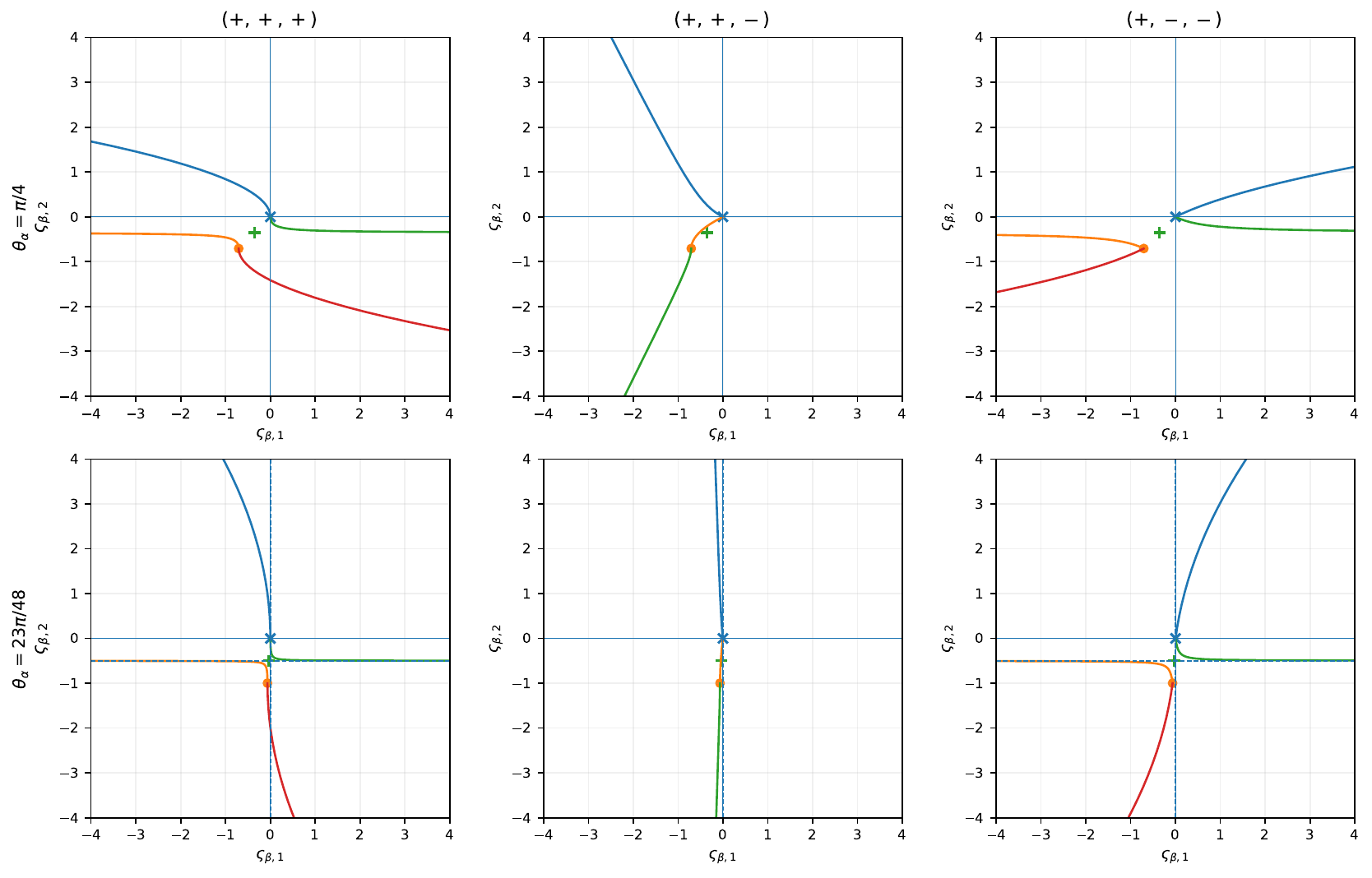}
    \caption{
    Unit-scale hydrostatic resonance curves in the
    \(\varsigma_\beta\)-plane.  The columns correspond to the three
    representative sign configurations \((+,+,+)\), \((+,+,-)\), and
    \((+,-,-)\), while the rows correspond to
    \(\theta_\alpha=\pi/4\) and \(\theta_\alpha=23\pi/48\).
    The cross marks \(\varsigma_\beta=0\), the dot marks
    \(\varsigma_\beta=-\varsigma_\alpha\), and the plus sign marks
    \(-\varsigma_\alpha/2\).  The dashed lines in the second row indicate the
    limiting sets as \(\theta_\alpha\to\pi/2\).
    }
    \label{fig:hydrostatic-resonance-curves}
\end{figure}

After eliminating \(\varsigma_\gamma\), let
\[
    J_h
    :=
    \left|
        \nabla_{\varsigma_\beta}
        \widetilde\Xi_1^h
    \right|.
\]
Since
\[
\partial_{\varsigma_{\beta,1}}
\widetilde\Xi_1^h
=
\frac{\sigma_\beta}{|\varsigma_{\beta,2}|}
-
\frac{\sigma_\gamma}{|\varsigma_{\gamma,2}|},
\]
the co--area factor is
\begin{equation}
\label{eq:hydrostatic-short-coarea}
\frac{d\mathcal H^1}{J_h}
=
\sin^2\theta_\alpha
\begin{cases}
\displaystyle
\frac{|z||1+z|}
     {\bigl||1+z|-|z|\bigr|}
\,dz,
&
\sigma_\beta=\sigma_\gamma,
\\[1.2em]
\displaystyle
\frac{|z||1+z|}
     {|1+z|+|z|}
\,dz,
&
\sigma_\beta=-\sigma_\gamma.
\end{cases}
\end{equation}

The resonant factorization
\[
    \widetilde{\mathcal V}_j
    =
    \omega_j^h\Gamma_{\alpha\beta\gamma}^h
\]
and the explicit hydrostatic interaction coefficient give
\begin{equation}
\label{eq:hydrostatic-short-Gamma-table}
\frac{\Gamma_{\alpha\beta\gamma}^h(z)}
     {2\pi\sin\theta_\alpha}
=
\begin{array}{c|c|c|c}
&
z>0
&
-1<z<0
&
z<-1
\\[0.3em]\hline
(+,+,+)
&
1+z
&
-1
&
-z
\\
(+,+,-)
&
0
&
-z
&
-1
\\
(+,-,-)
&
-z
&
0
&
1+z .
\end{array}
\end{equation}
At physical scale \(R=|k_\alpha|\),
\[
    \Gamma_{\alpha\beta\gamma}^h
    (R\varsigma_\alpha,R\varsigma_\beta,R\varsigma_\gamma)
    =
    R
    \Gamma_{\alpha\beta\gamma}^h
    (\varsigma_\alpha,\varsigma_\beta,\varsigma_\gamma).
\]
Together with the scaling of the co--area measure, this produces the overall
factor \(R^4\) in the collision integral.

On the \((+,+,+)\) branch, the two angular localizations bound
\[
    \left|\frac{k_1}{k_2}\right|
\]
above and below for each member of the triad.  Directly from
\eqref{eq:hydrostatic-X-table}, the admissible values of \(z\) therefore lie
in a finite union of compact intervals, uniformly separated from
\[
    0,\quad 1,\quad -1,\quad -2,\quad -\frac12,
    \quad\text{and}\quad |z|=\infty.
\]
Hence
\[
    |\varsigma_\beta(z)|
    \sim_{\varepsilon_0,\varepsilon_h}
    1,
    \qquad
    |\varsigma_\gamma(z)|
    \sim_{\varepsilon_0,\varepsilon_h}
    1,
\]
and this contribution is bounded by
\[
    R^4\langle R\rangle^{-m}
    \|n^h\|_{L^\infty_m}^2.
\]

On the \((+,+,-)\) branch, the part \(z>0\) is dynamically inactive:
\[
    \Gamma_{++-}^h(z)=0.
\]
On the active middle branch, as \(z\to0^-\),
\[
    |\varsigma_\beta(z)|
    \sim_{\varepsilon_0,\varepsilon_h}
    |z|,
    \qquad
    |\varsigma_\gamma(z)|
    \sim_{\varepsilon_0,\varepsilon_h}
    1.
\]
Moreover,
\[
    |\Gamma_{++-}^h(z)|
    \lesssim_{\varepsilon_0,\varepsilon_h}
    |z|,
    \qquad
    \frac{d\mathcal H^1}{J_h}
    \lesssim_{\varepsilon_0,\varepsilon_h}
    |z|\,dz.
\]
Thus the local kernel carries the factor
\[
    |z|^3\,dz.
\]
Restoring the scale \(R\), the term containing the collapsed low-mode weight is
controlled by
\[
\begin{aligned}
R^4
\int_0^\delta
    \rho^3
    \langle R\rho\rangle^{-m}
\,d\rho
&=
\int_0^{R\delta}
    y^3\langle y\rangle^{-m}
\,dy
\\
&\lesssim_m1,
\end{aligned}
\]
because \(m>4\).  The terms containing only high-mode weights are bounded by
\[
    R^4\langle R\rangle^{-m}.
\]

On the non-compact part \(z<-1\),
\[
    |\varsigma_\beta(z)|
    +
    |\varsigma_\gamma(z)|
    \sim_{\varepsilon_0,\varepsilon_h}
    |z|,
\]
while the remaining coefficient and co--area factors have at most linear
growth.  Hence the weighted tail is controlled by
\[
    R^4
    \int_2^\infty
        t\langle Rt\rangle^{-m}
    \,dt
    \lesssim_m1.
\]

Finally, on the \((+,-,-)\) branch,
\[
    \Gamma_{+--}^h(z)=0,
    \qquad
    -1<z<0.
\]
The two active branches have collapsed endpoints at \(z=0^+\) and
\(z=-1^-\).  Near either endpoint, with \(\rho\) denoting the distance to the
endpoint,
\[
    |\Gamma_{+--}^h|
    \lesssim_{\varepsilon_0,\varepsilon_h}
    \rho,
    \qquad
    \frac{d\mathcal H^1}{J_h}
    \lesssim_{\varepsilon_0,\varepsilon_h}
    \rho\,d\rho,
\]
and hence the same factor
\[
    \rho^3\,d\rho
\]
appears.  The remaining regular portions are handled as in the
\((+,+,+)\) case.  The other sign configurations follow by symmetry.

Combining the three representative branches proves
\eqref{eq:hydrostatic-cutoff-estimate}.  Local well-posedness follows from the
standard Banach-space iteration.
\end{proof}

In particular, the hydrostatic resonance geometry has no analogue of the
finite-frequency critical quadratic cusp at
\[
    |\cos\theta_\alpha|=\frac12.
\]
All active collapsed hydrostatic branches approach the low mode linearly, and
the interaction coefficient vanishes to first order.  Thus no localization
near \(|\cos\theta|=\frac12\) is needed.

\begin{remark}[The horizontal hydrostatic cut-off]
\label{rmk:hydrostatic-horizontal-cutoff}
The localization near
\[
    \sin\theta=0
\]
cannot be removed within the present radial \(L^\infty_m\) framework.
Although the explicit powers of \(1/|k_2|\) in the interaction coefficient
partially cancel on the exact resonant set, this cancellation is not uniform
in the joint near-horizontal and high-frequency limit.

Indeed, consider the \((+,+,+)\) branch and let
\[
    R\to\infty,
    \qquad
    \sin\theta_{\alpha,R}=R^{-1/2},
    \qquad
    z=R^{-1/2}\lambda,
\]
where
\[
    \lambda\in[\lambda_0,\lambda_1]
    \Subset(0,\infty).
\]
Since \(X_{+++}(z)=-z^2\),
\[
    k_{\beta,R}
    =
    R\varsigma_\beta(z)
    =
    \left(
        -\cos\theta_{\alpha,R}\lambda^2,\,
        \lambda
    \right),
\]
and hence
\[
    |k_{\beta,R}|\sim1,
    \qquad
    |k_{\alpha,R}|
    \sim
    |k_{\gamma,R}|
    \sim R.
\]
Along this family,
\[
    \Gamma_{\alpha\beta\gamma}^h
    \sim R^{1/2},
    \qquad
    |\omega_\alpha^h|
    \sim R^{1/2},
    \qquad
    \frac{d\mathcal H^1}{J_h}
    \sim d\lambda.
\]
Therefore the positive gain kernel has size
\[
    (\Gamma_{\alpha\beta\gamma}^h)^2
    (\omega_\alpha^h)^2
    \frac{d\mathcal H^1}{J_h}
    \sim
    R^2\,d\lambda.
\]

Choosing a nonnegative spectrum with
\[
    n_R(k_{\alpha,R})=0,
    \qquad
    n_R(k_{\beta,R})\sim1,
    \qquad
    n_R(k_{\gamma,R})\sim R^{-m},
    \qquad
    \|n_R\|_{L^\infty_m}\lesssim1,
\]
one obtains
\[
    \langle k_{\alpha,R}\rangle^m
    \left|
        \mathcal C^h(n_R)(k_{\alpha,R})
    \right|
    \gtrsim
    R^2.
\]
Thus the positive gain term alone prevents a bounded-vector-field estimate
without a fixed localization away from \(\sin\theta=0\).

This obstruction is specific to the present isotropic radial metric.
A treatment allowing arbitrarily small \(|\sin\theta|\) would require a
different framework, possibly using anisotropic weights adapted to the
hydrostatic slope
\[
    \left|\frac{k_1}{k_2}\right|,
\]
or incorporating the pole-generated high--low exchange into an energy-level
evolution.  Such a theory lies beyond the scope of the present paper.
\end{remark}


\section*{Acknowledgements}
\phantomsection
\addcontentsline{toc}{section}{Acknowledgements}

The author is a PhD student under the supervision of Pierre
Germain and Michele Coti Zelati.  This paper grew out of a problem introduced
to her by Pierre Germain, to whom she is deeply grateful.  She would also like
to thank Pierre Germain and Michele Coti Zelati for many inspiring discussions,
for their careful reading of the manuscript, and for their comments and
suggestions, which greatly improved the paper.  The author is also grateful to
Michal Shavit for kind discussions.
\appendix
\section{Formal Derivation of the WKE}
\label{Appendix: Formal derivation}

This appendix records a self-contained formal derivation of the kinetic
model used in the main text.  We first pass from the two-dimensional
Boussinesq Hamiltonian system to its normal-mode equations.  We then describe
the weakly turbulent regime in the large periodic box, perform the Duhamel
expansion to second order, use the parity-pairing identity for independent
uniform phases, and take successively the large-box and large-time limits.
The Hamiltonian energy and pseudo-momentum are used once, at the modal level,
to obtain the resonant factorization of the interaction coefficients.  This
factorization is also the algebraic origin of the corresponding collision
invariants of the WKE, so no separate repetition of the same argument is
needed later in the appendix.

The final part gives an independent derivation of the hydrostatic WKE from the
hydrostatic Euler--Boussinesq Hamiltonian system.  In particular, the
hydrostatic kinetic equation is not obtained by inserting the hydrostatic
dispersion relation into an already closed non-hydrostatic WKE.

The discussion is formal.  We do not prove propagation of the random-phase
structure, control of the Duhamel remainder up to kinetic time, or the
large-box equidistribution of quasi-resonant lattice points.  Earlier work of
Shavit--B\"uhler--Shatah derives the non-hydrostatic two-dimensional
Boussinesq kinetic equation by a normal-mode closure~\cite{shavit2024}.  Here
we use the normalization of the present paper, make the parity pairings and
kinetic scaling explicit, and derive the hydrostatic WKE directly from the
hydrostatic Hamiltonian PDE.

\subsection{The Boussinesq Hamiltonian system and its normal modes}
\label{app:boussinesq-modes}

Let
\[
    x=(x_1,x_2)\in\mathbb T_L^2
    :=
    (\mathbb R/L\mathbb Z)^2,
\]
where \(x_2\) is the vertical coordinate.  We work with \(L\)-periodic
perturbations of a motionless background with constant buoyancy gradient
\(N^2e_2\).  The affine background profile is understood on the universal
cover \(\mathbb R^2\); only the perturbation variables are required to be
periodic.  The corresponding inviscid perturbation equations are
\begin{equation}
\label{eq:appendix-boussinesq-primitive}
\begin{cases}
(\partial_t+u\cdot\nabla)u=-\nabla p+b\,e_2,\\[0.3em]
(\partial_t+u\cdot\nabla)b+N^2u_2=0,\\[0.3em]
\nabla\cdot u=0,
\end{cases}
\end{equation}
where \(e_2=(0,1)\).

We restrict to the zero-mean wave sector.  Since \(u\) is divergence free,
there is then a unique mean-zero stream function \(A\) such that
\[
    u=\nabla^\perp A
    :=
    (-\partial_2A,\partial_1A).
\]
Indeed, define the signed scalar vorticity
\[
    q_\mathrm{sign}
    :=
    \partial_2u_1-\partial_1u_2
    =
    -\Delta A,
\]
then the periodic Biot--Savart law reads
\[
    A=(-\Delta)^{-1}q_\mathrm{sign},
    \qquad
    u=\nabla^\perp(-\Delta)^{-1}q_\mathrm{sign},
\]
where \((-\Delta)^{-1}\) is taken on mean-zero functions.  Setting
\[
    \zeta:=-\frac{b}{N^2},
\]
system~\eqref{eq:appendix-boussinesq-primitive} becomes
\begin{equation}
\label{eq:appendix-boussinesq-stream}
\begin{cases}
\partial_t(-\Delta A)
+\{A,-\Delta A\}
-N^2\partial_1\zeta=0,\\[0.3em]
\partial_t\zeta
+\{A,\zeta\}
-\partial_1A=0,
\end{cases}
\end{equation}
where
\[
    \{F,G\}
    :=
    \partial_1F\,\partial_2G
    -
    \partial_2F\,\partial_1G .
\]

This is a non-canonical Hamiltonian system with Hamiltonian energy
\begin{equation}
\label{eq:appendix-boussinesq-energy}
    E
    =\frac12\int_{\mathbb T_L^2}
       \bigl(|\nabla A|^2+N^2\zeta^2\bigr)\,dx,
\end{equation}
and horizontal pseudo-momentum
\begin{equation}
\label{eq:appendix-boussinesq-pseudomomentum}
    P=\int_{\mathbb T_L^2}\zeta\,\Delta A\,dx.
\end{equation}

We use the Fourier convention
\[
    F(x)=\frac1{L^2}\sum_{k\in\mathbb Z_L^2}
       \widehat F(k)e^{2\pi i k\cdot x},
    \qquad
    \mathbb Z_L^2:=\frac1L\mathbb Z^2,
\]
so that
\[
    \int_{\mathbb T_L^2}|F(x)|^2\,dx
    =\frac1{L^2}\sum_{k\in\mathbb Z_L^2}|\widehat F(k)|^2.
\]
For $k\neq0$ and $\sigma\in\{\pm1\}$, define
\begin{equation}
\label{eq:appendix-normal-mode-transform}
    Z_{\sigma,k}
    :=-\sigma\pi|k|\,\widehat A(k)+\frac N2\widehat\zeta(k).
\end{equation}
Equivalently,
\[
    \widehat A(k)=\frac{Z_{-,k}-Z_{+,k}}{2\pi|k|},
    \qquad
    \widehat\zeta(k)=\frac{Z_{+,k}+Z_{-,k}}{N}.
\]
The linearized system diagonalizes as
\begin{equation}
\label{eq:appendix-linear-wave-equation}
    \partial_t Z_{\sigma,k}+i\omega_{\sigma,k}Z_{\sigma,k}=0,
    \qquad
    \omega_{\sigma,k}:=\sigma N\frac{k^1}{|k|}.
\end{equation}
The horizontal slowness is
\[
    s_{\sigma,k}:=\frac{k^1}{\omega_{\sigma,k}}
    =\frac{\sigma|k|}{N}.
\]
With the above normalization, the two quadratic invariants diagonalize as
\begin{equation}
\label{eq:appendix-modal-invariants}
    E=\frac1{L^2}\sum_{\sigma=\pm1}\sum_k|Z_{\sigma,k}|^2,
    \qquad
    P=\frac{2\pi}{L^2}\sum_{\sigma=\pm1}\sum_k
       s_{\sigma,k}|Z_{\sigma,k}|^2,
\end{equation}
up to the harmless overall sign convention in the definition of $P$.

Introduce the multi-index
\[
    \alpha=(\sigma_\alpha,k_\alpha),
    \qquad
    \omega_\alpha:=\omega_{\sigma_\alpha,k_\alpha},
    \qquad
    s_\alpha:=s_{\sigma_\alpha,k_\alpha},
\]
and use the same convention for $\beta$ and $\gamma$.  A direct Fourier
calculation of the quadratic terms gives
\begin{equation}
\label{eq:appendix-modal-amplitude-equation}
    \partial_t Z_\alpha+i\omega_\alpha Z_\alpha
    =\frac{\varepsilon}{2}
      \sum_{\beta,\gamma}
      V_{\alpha,L}^{\beta\gamma}
      \overline{Z_\beta}\,\overline{Z_\gamma}\,
      \delta^K_{k_\alpha+k_\beta+k_\gamma,0}.
\end{equation}
Here $0<\varepsilon\ll1$ measures the weak nonlinearity,
$\delta^K$ is the Kronecker delta on $\mathbb Z_L^2$, and
\[
    V_{\alpha,L}^{\beta\gamma}
    =V_{\alpha,L}^{\gamma\beta}\in\mathbb R.
\]
Under the Fourier normalization above, one interaction vertex contains a
factor $L^{-2}$; we therefore write
\begin{equation}
\label{eq:appendix-L-scaling-V}
    V_{\alpha,L}^{\beta\gamma}
    =L^{-2}\mathcal V_\alpha^{\beta\gamma},
\end{equation}
where $\mathcal V$ has a nontrivial continuum limit.  For nonzero wave vectors
satisfying the momentum constraint, a direct computation gives
\begin{equation}
\label{eq:appendix-boussinesq-interaction-coefficient}
\mathcal V_\alpha^{\beta\gamma}
=
-\pi N^2
\frac{\sigma_\alpha\sigma_\beta\sigma_\gamma}
{|k_\alpha||k_\beta||k_\gamma|}
(k_\beta\times k_\gamma)
(s_\beta-s_\gamma)
(s_\alpha+s_\beta+s_\gamma).
\end{equation}
In particular, the signed-area factor $k_\beta\times k_\gamma$ makes every
nonzero collinear interaction vanish.

\subsection{Hamiltonian triad identities and resonant factorization}
\label{app:resonant-factorization}

The Hamiltonian invariants are used at this stage to identify the algebraic
structure of the interaction coefficients.  Conservation of the two diagonal
quadratic forms in \eqref{eq:appendix-modal-invariants} implies, for every
nondegenerate interacting triad,
\begin{equation}
\label{eq:appendix-V-energy-identity}
    \mathcal V_\alpha^{\beta\gamma}
    +\mathcal V_\beta^{\alpha\gamma}
    +\mathcal V_\gamma^{\alpha\beta}=0,
\end{equation}
and
\begin{equation}
\label{eq:appendix-V-pseudomomentum-identity}
    s_\alpha\mathcal V_\alpha^{\beta\gamma}
    +s_\beta\mathcal V_\beta^{\alpha\gamma}
    +s_\gamma\mathcal V_\gamma^{\alpha\beta}=0.
\end{equation}
These identities can also be verified directly from
\eqref{eq:appendix-boussinesq-interaction-coefficient}.

On the resonant set,
\[
    k_\alpha+k_\beta+k_\gamma=0,
    \qquad
    \omega_\alpha+\omega_\beta+\omega_\gamma=0.
\]
Since $\omega_js_j=k_j^1$, the momentum constraint also gives
\[
    \omega_\alpha s_\alpha
    +\omega_\beta s_\beta
    +\omega_\gamma s_\gamma=0.
\]
Thus the vectors
\[
    \bigl(
       \mathcal V_\alpha^{\beta\gamma},
       \mathcal V_\beta^{\alpha\gamma},
       \mathcal V_\gamma^{\alpha\beta}
    \bigr)
    \quad\text{and}\quad
    (\omega_\alpha,\omega_\beta,\omega_\gamma)
\]
obey the same two linear constraints.  On the nondegenerate resonant set the
common orthogonal complement is one-dimensional.  Hence there exists a scalar
$\Gamma_{\alpha\beta\gamma}$, symmetric under permutations of
$(\alpha,\beta,\gamma)$, such that
\begin{equation}
\label{eq:appendix-resonant-factorization}
    \mathcal V_\alpha^{\beta\gamma}
       =\omega_\alpha\Gamma_{\alpha\beta\gamma},
    \qquad
    \mathcal V_\beta^{\alpha\gamma}
       =\omega_\beta\Gamma_{\alpha\beta\gamma},
    \qquad
    \mathcal V_\gamma^{\alpha\beta}
       =\omega_\gamma\Gamma_{\alpha\beta\gamma}.
\end{equation}
After setting $N=1$ and passing to the continuum normalization of the main
text, this coefficient is
\begin{equation}
\label{eq:appendix-explicit-Gamma}
    \Gamma_{\alpha\beta\gamma}
    =\pi^2
      \bigl(
          \sigma_\alpha\sin\theta_\alpha
          +\sigma_\beta\sin\theta_\beta
          +\sigma_\gamma\sin\theta_\gamma
      \bigr)
      \bigl(
          \sigma_\alpha|k_\alpha|
          +\sigma_\beta|k_\beta|
          +\sigma_\gamma|k_\gamma|
      \bigr).
\end{equation}

The same two Hamiltonian identities will reappear at the kinetic level as the
collision invariants $1$ and $s_\alpha$.  The direct symmetrization for the
angularly cut-off WKE is carried out in
Section~\ref{Conservation Laws}; it is not repeated in a separate appendix
subsection.

\subsection{The weakly turbulent regime and random-phase data}
\label{app:weakly-turbulent-regime}

Following the standard terminology in mathematical wave turbulence, we work
formally in the \emph{weakly turbulent regime}.  For the present three-wave
system this combines
\begin{enumerate}[label=\textup{(\roman*)}]
    \item weak nonlinearity, $0<\varepsilon\ll1$;
    \item a large periodic box, $L\gg1$;
    \item independent random phases with deterministic amplitudes.
\end{enumerate}
The large-box and weak-nonlinearity parameters are linked by the requirement
that the quasi-resonant layer observed on kinetic time contain many lattice
points.  For the two-dimensional quadratic interaction considered here this
leads formally to
\begin{equation}
\label{eq:appendix-weak-turbulence-window}
    1\ll T_{\rm kin}\asymp\varepsilon^{-2}\ll L^2,
    \qquad\text{equivalently}\qquad
    L^{-1}\ll\varepsilon\ll1.
\end{equation}
The origin of this condition is explained in Step~3 below.

Let $\mathcal A_+$ contain one representative from each reality pair
$(\sigma,k)\leftrightarrow(\sigma,-k)$.  For
$\alpha\in\mathcal A_+$, let
\[
    \vartheta_\alpha\sim\operatorname{Unif}[0,1),
    \qquad
    z_\alpha:=e^{2\pi i\vartheta_\alpha},
\]
with the variables $\{z_\alpha\}_{\alpha\in\mathcal A_+}$ independent.  We
choose fixed-amplitude random-phase data
\begin{equation}
\label{eq:appendix-random-phase-data-parity}
    b_{\alpha,L}:=B_{\alpha,L}(0)
    =L\sqrt{n_\alpha^0}\,z_\alpha,
    \qquad \alpha\in\mathcal A_+,
\end{equation}
and extend them by the reality condition
\[
    b_{(\sigma,-k),L}=\overline{b_{(\sigma,k),L}}.
\]
The factor $L$ is the two-dimensional thermodynamic normalization.  With
\begin{equation}
\label{eq:appendix-normalized-spectrum-parity}
    n_{\alpha,L}(t)
    :=\frac1{L^2}\mathbb E|B_{\alpha,L}(t)|^2,
\end{equation}
one has $n_{\alpha,L}(0)=n_\alpha^0$, and the expected Hamiltonian energy per
unit area has the nontrivial large-box limit
\[
    \frac1{L^2}\mathbb E E_L(0)
    =\frac1{L^2}\sum_\alpha n_\alpha^0
    \longrightarrow
    \int n^0(\alpha)\,d\alpha.
\]

The averaging rule is the parity-pairing identity.  For indices reduced to the
independent half-lattice,
\begin{equation}
\label{eq:appendix-parity-pairing}
\mathbb E\!\left(
 z_{\alpha_1}\cdots z_{\alpha_m}
 \overline{z_{\beta_1}}\cdots\overline{z_{\beta_m}}
\right)
=
\begin{cases}
1,
&\begin{array}{l}
\text{if there exists }\pi\in S_m\text{ such that}\\[-0.2em]
\alpha_{\pi(j)}=\beta_j\text{ for every }j,
\end{array}
\\[1em]
0,&\text{otherwise.}
\end{cases}
\end{equation}
Equivalently, the expectation is nonzero exactly when the two lists of
indices agree as multisets.  This is an exact identity for independent uniform
phases; no Gaussian assumption and no Wick formula are used.

In extracting the leading large-box contribution, we retain the nondegenerate
parity pairings.  Fully repeated-index diagonals have one fewer free lattice
index and disappear after the normalized large-box limit.  The
reality-degenerate pairing
\[
    k_\beta=-k_\gamma
\]
forces $k_\alpha=0$ by momentum conservation.  The zero vector is not part of
the wave-mode system, and the corresponding coefficient in the original
Fourier equation vanishes.  The analogous feedback degeneracies in the second
Duhamel iterate force $k_\beta=0$ or $k_\gamma=0$ and are neglected for the
same reason.

\subsection{Four-step formal derivation of the kinetic equation}
\label{app:four-step-formal-derivation}

We now follow the standard four-step presentation: a weakly nonlinear Duhamel
expansion, parity-pairing cancellations, the large-box limit, and the
large-time resonant limit.

\paragraph{Step 1: weakly nonlinear Duhamel expansion.}
Define the interaction representation
\[
    B_{\alpha,L}(t):=e^{i\omega_\alpha t}Z_{\alpha,L}(t)
\]
and the frequency resonance function
\[
    \Xi_1(\alpha,\beta,\gamma)
    :=\omega_\alpha+\omega_\beta+\omega_\gamma.
\]
Using \eqref{eq:appendix-L-scaling-V}, the modal equation becomes
\begin{equation}
\label{eq:appendix-interaction-representation}
    \partial_t B_{\alpha,L}
    =\frac{\varepsilon}{2L^2}
      \sum_{\beta,\gamma}
      \mathcal V_\alpha^{\beta\gamma}
      e^{it\Xi_1(\alpha,\beta,\gamma)}
      \overline{B_{\beta,L}}\,\overline{B_{\gamma,L}}\,
      \delta^K_{k_\alpha+k_\beta+k_\gamma,0}.
\end{equation}
Set
\[
    G_t(\lambda):=\int_0^t e^{is\lambda}\,ds.
\]
We expand
\begin{equation}
\label{eq:appendix-Duhamel-series}
    B_{\alpha,L}(t)
    =b_{\alpha,L}
     +\varepsilon B_{\alpha,L}^{(1)}(t)
     +\varepsilon^2B_{\alpha,L}^{(2)}(t)
     +\cdots.
\end{equation}
The first iterate is
\begin{equation}
\label{eq:appendix-B1}
    B_{\alpha,L}^{(1)}(t)
    =\frac1{2L^2}\sum_{\beta,\gamma}
      \mathcal V_\alpha^{\beta\gamma}
      G_t\bigl(\Xi_1(\alpha,\beta,\gamma)\bigr)
      \overline{b_{\beta,L}}\,\overline{b_{\gamma,L}}\,
      \delta^K_{k_\alpha+k_\beta+k_\gamma,0},
\end{equation}
and the second iterate is determined by
\begin{equation}
\label{eq:appendix-B2-differential}
\begin{aligned}
    \partial_tB_{\alpha,L}^{(2)}
    =\frac1{2L^2}\sum_{\beta,\gamma}
      \mathcal V_\alpha^{\beta\gamma}
      e^{it\Xi_1(\alpha,\beta,\gamma)}
      \bigl(
          \overline{B_{\beta,L}^{(1)}}\,\overline{b_{\gamma,L}}
         +\overline{b_{\beta,L}}\,\overline{B_{\gamma,L}^{(1)}}
      \bigr)
      \delta^K_{k_\alpha+k_\beta+k_\gamma,0}.
\end{aligned}
\end{equation}

\paragraph{Step 2: parity pairings and the second moment.}
Expanding the normalized second moment
\eqref{eq:appendix-normalized-spectrum-parity},
\begin{equation}
\label{eq:appendix-spectrum-expansion-parity}
\begin{aligned}
 n_{\alpha,L}(t)
={}&n_\alpha^0
 +\frac{2\varepsilon}{L^2}\operatorname{Re}
    \mathbb E\bigl[B_{\alpha,L}^{(1)}\overline{b_{\alpha,L}}\bigr]\\
 &+\frac{\varepsilon^2}{L^2}
 \left(
    \mathbb E|B_{\alpha,L}^{(1)}|^2
    +2\operatorname{Re}
       \mathbb E\bigl[B_{\alpha,L}^{(2)}
       \overline{b_{\alpha,L}}\bigr]
 \right)
 +\text{higher-order terms}.
\end{aligned}
\end{equation}
The order-$\varepsilon$ term contains three unmatched phases and therefore
vanishes.

For $\mathbb E|B_{\alpha,L}^{(1)}|^2$, the parity identity requires
\[
    \{\beta,\gamma\}=\{\beta',\gamma'\}.
\]
Away from the degenerate sets, the two possibilities are
\[
    (\beta',\gamma')=(\beta,\gamma)
    \qquad\text{or}\qquad
    (\beta',\gamma')=(\gamma,\beta).
\]
Using
$\mathcal V_\alpha^{\beta\gamma}
 =\mathcal V_\alpha^{\gamma\beta}$, one obtains
\begin{equation}
\label{eq:appendix-B1-square-parity}
\mathbb E|B_{\alpha,L}^{(1)}(t)|^2
=\frac12\sum_{\beta,\gamma}
\bigl(\mathcal V_\alpha^{\beta\gamma}\bigr)^2
n_\beta^0n_\gamma^0
\bigl|G_t(\Xi_{\alpha\beta\gamma})\bigr|^2
\delta^K_{k_\alpha+k_\beta+k_\gamma,0},
\end{equation}
up to the neglected repeated-index and zero-wavevector configurations.

For the part of $B_{\alpha,L}^{(2)}$ obtained by inserting the first iterate
on the $\gamma$-leg, the parity identity is nonzero precisely when the two
inner indices agree, up to permutation, with $\alpha$ and $\beta$.  The two
nondegenerate permutations give the same contribution.  The alternative
reality pairing forces $k_\gamma=0$ and is omitted.  Using
\begin{equation}
\label{eq:appendix-time-identity-parity}
    2\operatorname{Re}
    \int_0^t e^{is\lambda}\overline{G_s(\lambda)}\,ds
    =|G_t(\lambda)|^2,
\end{equation}
we obtain
\begin{equation}
\label{eq:appendix-B2-gamma-parity}
\begin{aligned}
2\operatorname{Re}
\mathbb E\bigl[B_{\alpha;\gamma,L}^{(2)}(t)
\overline{b_{\alpha,L}}\bigr]
={}&\frac12\sum_{\beta,\gamma}
\mathcal V_\alpha^{\beta\gamma}
\mathcal V_\gamma^{\alpha\beta}
 n_\alpha^0n_\beta^0
 |G_t(\Xi_{\alpha\beta\gamma})|^2
 \delta^K_{k_\alpha+k_\beta+k_\gamma,0}.
\end{aligned}
\end{equation}
The symmetric insertion on the $\beta$-leg gives
\begin{equation}
\label{eq:appendix-B2-beta-parity}
\begin{aligned}
2\operatorname{Re}
\mathbb E\bigl[B_{\alpha;\beta,L}^{(2)}(t)
\overline{b_{\alpha,L}}\bigr]
={}&\frac12\sum_{\beta,\gamma}
\mathcal V_\alpha^{\beta\gamma}
\mathcal V_\beta^{\alpha\gamma}
 n_\alpha^0n_\gamma^0
 |G_t(\Xi_{\alpha\beta\gamma})|^2
 \delta^K_{k_\alpha+k_\beta+k_\gamma,0}.
\end{aligned}
\end{equation}
Combining the three nondegenerate parity pairings yields
\begin{equation}
\label{eq:appendix-second-order-spectrum-parity}
\begin{aligned}
 n_{\alpha,L}(t)-n_\alpha^0
 ={}&\frac{\varepsilon^2}{2L^2}
 \sum_{\beta,\gamma}
 \mathcal Q_{\alpha\beta\gamma}(n^0)
 |G_t(\Xi_{\alpha\beta\gamma})|^2
 \delta^K_{k_\alpha+k_\beta+k_\gamma,0}
 +\text{higher-order terms},
\end{aligned}
\end{equation}
where
\begin{equation}
\label{eq:appendix-prekinetic-Q-parity}
\begin{aligned}
\mathcal Q_{\alpha\beta\gamma}(n)
={}&
\bigl(\mathcal V_\alpha^{\beta\gamma}\bigr)^2n_\beta n_\gamma
+\mathcal V_\alpha^{\beta\gamma}\mathcal V_\gamma^{\alpha\beta}
 n_\alpha n_\beta\\
&+\mathcal V_\alpha^{\beta\gamma}\mathcal V_\beta^{\alpha\gamma}
 n_\alpha n_\gamma.
\end{aligned}
\end{equation}

\paragraph{Step 3: large-box limit.}
For fixed $\alpha$, the Kronecker momentum constraint sets
$k_\gamma=-k_\alpha-k_\beta$.  Hence the sum in
\eqref{eq:appendix-second-order-spectrum-parity} contains one
 two-dimensional lattice sum, and formally
\begin{equation}
\label{eq:appendix-Riemann-sum}
    \frac1{L^2}\sum_{k\in\mathbb Z_L^2}F(k)
    \longrightarrow
    \int_{\mathbb R^2}F(k)\,dk
    \qquad\text{as }L\to\infty.
\end{equation}
Equivalently, the Kronecker momentum constraint converges to the Dirac factor
$\delta(k_\alpha+k_\beta+k_\gamma)$.  At fixed observation time,
\eqref{eq:appendix-second-order-spectrum-parity} therefore becomes
\begin{equation}
\label{eq:appendix-large-box-prekinetic}
\begin{aligned}
 n_\alpha(t)-n_\alpha^0
 \simeq{}&\frac{\varepsilon^2}{2}
 \int
 \mathcal Q_{\alpha\beta\gamma}(n^0)
 |G_t(\Xi_{\alpha\beta\gamma})|^2
 \delta(k_\alpha+k_\beta+k_\gamma)
 \,d\beta\,d\gamma.
\end{aligned}
\end{equation}

For $t\gg1$, the factor $|G_t(\Xi_1)|^2$ is concentrated in the
quasi-resonant layer
\[
    |\Xi_1|\lesssim t^{-1}.
\]
After imposing the momentum constraint, and away from critical points of the
resonance function, this layer is a strip of area $O(t^{-1})$ around a
one-dimensional resonance curve.  Since the lattice $\mathbb Z_L^2$ has
density $L^2$, the expected number of lattice points in this strip is
\[
    O\!\left(\frac{L^2}{t}\right).
\]
The Riemann-sum passage therefore requires an equidistribution, or lattice
equipartition, hypothesis for the quasi-resonant points and the sampling
condition
\begin{equation}
\label{eq:appendix-large-box-sampling}
    1\ll t\ll L^2.
\end{equation}
At kinetic time $t\sim\varepsilon^{-2}$, this becomes
\[
    \varepsilon^{-2}\ll L^2,
    \qquad\text{equivalently}\qquad
    L^{-1}\ll\varepsilon,
\]
which is the large-box component of
\eqref{eq:appendix-weak-turbulence-window}.

\paragraph{Step 4: large-time limit and kinetic time.}
In the sense of distributions,
\begin{equation}
\label{eq:appendix-secular-delta-limit}
    \frac1t|G_t(\lambda)|^2
    =\frac4t\frac{\sin^2(t\lambda/2)}{\lambda^2}
    \rightharpoonup 2\pi\delta(\lambda)
    \qquad\text{as }t\to\infty.
\end{equation}
Applying this to \eqref{eq:appendix-large-box-prekinetic} gives
\begin{equation}
\label{eq:appendix-large-time-increment}
\begin{aligned}
 n_\alpha(t)-n_\alpha^0
 \simeq{}&\pi\varepsilon^2t
 \int
 \mathcal Q_{\alpha\beta\gamma}(n^0)
 \delta(k_\alpha+k_\beta+k_\gamma)
 \delta(\omega_\alpha+\omega_\beta+\omega_\gamma)
 \,d\beta\,d\gamma.
\end{aligned}
\end{equation}
Thus the characteristic kinetic time is
\begin{equation}
\label{eq:appendix-kinetic-time}
    T_{\rm kin}=\frac1{\pi\varepsilon^2},
    \qquad
    \tau:=\frac{t}{T_{\rm kin}}=\pi\varepsilon^2t.
\end{equation}
The factor $\pi$ depends on the Fourier and kinetic-time normalization and may
be absorbed into the definition of $\tau$; the invariant scaling is
$T_{\rm kin}\asymp\varepsilon^{-2}$.

The formal order of limits is: first take the large-box limit on observation
times satisfying $t\ll L^2$, and then take the large-time limit.  Equivalently,
one may take the joint weak-turbulence limit
\begin{equation}
\label{eq:appendix-joint-kinetic-limit}
    \varepsilon\to0,
    \qquad
    L\to\infty,
    \qquad
    t=\frac{\tau}{\pi\varepsilon^2},
    \qquad
    L\varepsilon\to\infty,
\end{equation}
with $\tau=O(1)$.

Finally, the random-phase Markov closure replaces $n^0$ in the short kinetic
increment by the slowly evolving spectrum $n(\tau)$.  This yields
\begin{equation}
\label{eq:appendix-prekinetic-WKE}
    \partial_\tau n_\alpha
    =\int
    \mathcal Q_{\alpha\beta\gamma}(n)
    \delta(k_\alpha+k_\beta+k_\gamma)
    \delta(\omega_\alpha+\omega_\beta+\omega_\gamma)
    \,d\beta\,d\gamma.
\end{equation}
Using the resonant factorization
\eqref{eq:appendix-resonant-factorization},
\[
\mathcal Q_{\alpha\beta\gamma}(n)
=\Gamma_{\alpha\beta\gamma}^2\omega_\alpha
\bigl(
  \omega_\alpha n_\beta n_\gamma
 +\omega_\gamma n_\alpha n_\beta
 +\omega_\beta n_\alpha n_\gamma
\bigr),
\]
and therefore
\begin{equation}
\label{eq:appendix-formal-WKE}
\begin{aligned}
    \partial_\tau n_\alpha
    ={}&\int
    \Gamma_{\alpha\beta\gamma}^2\omega_\alpha
    \bigl(
        \omega_\alpha n_\beta n_\gamma
        +\omega_\gamma n_\alpha n_\beta
        +\omega_\beta n_\alpha n_\gamma
    \bigr)\\
    &\qquad\times
    \delta(k_\alpha+k_\beta+k_\gamma)
    \delta(\omega_\alpha+\omega_\beta+\omega_\gamma)
    \,d\beta\,d\gamma.
\end{aligned}
\end{equation}
This is the WKE studied in the main text.  We subsequently rename the slow
kinetic time $\tau$ as $t$.

\paragraph{Hamiltonian invariants at the kinetic level.}
\label{app:WKE-conservation-laws}
The resonant factorization in
\eqref{eq:appendix-resonant-factorization}, together with permutation
symmetry of the resonant measure, immediately yields the formal identities
\begin{equation}
\label{eq:appendix-WKE-energy-conservation}
    \frac{d}{d\tau}\int n_\alpha\,d\alpha=0,
\end{equation}
and
\begin{equation}
\label{eq:appendix-WKE-PM-conservation}
    \frac{d}{d\tau}\int s_\alpha n_\alpha\,d\alpha=0.
\end{equation}
These are the kinetic counterparts of
\eqref{eq:appendix-boussinesq-energy} and
\eqref{eq:appendix-boussinesq-pseudomomentum}.  The full permutation
calculation, including the symmetric angular cut-off and the $H$-theorem, is
given in Section~\ref{Conservation Laws}.

\begin{remark}[Formal nature of the four steps]
The argument above does not prove propagation of the random-phase structure,
control the Duhamel remainder up to $T_{\rm kin}$, or establish the
quasi-resonant lattice equidistribution required in Step~3.  The angular
cut-off and co-area estimates of the main text give a rigorous meaning to the
resulting collision operator; they do not by themselves justify the kinetic
limit from the Boussinesq PDE.
\end{remark}

\subsection{Independent derivation from the hydrostatic Hamiltonian system}
\label{app:hydrostatic-formal-derivation}

We now derive the hydrostatic kinetic equation from the hydrostatic
Euler--Boussinesq system itself.  The order of operations is
\[
\text{hydrostatic Hamiltonian PDE}
\longrightarrow
\text{hydrostatic normal modes}
\longrightarrow
\text{hydrostatic WKE}.
\]
We do not obtain the hydrostatic equation by replacing $|k|$ by $|k^2|$ in
\eqref{eq:appendix-formal-WKE}; a singular PDE limit and a kinetic closure need
not commute.

\subsubsection{Hydrostatic Hamiltonian, normal modes, and invariants}

Following the hydrostatic Euler--Boussinesq model
in~\cite{bianchini2025ill}, the horizontal velocity $u$, vertical velocity
$v$, density $\varrho$, and pressure $p$ satisfy
\begin{equation}
\label{eq:appendix-hydrostatic-primitive}
\begin{cases}
\partial_t\varrho+u\partial_1\varrho+v\partial_2\varrho=0,\\
\partial_tu+u\partial_1u+v\partial_2u+\partial_1p=0,\\
\partial_2p+\varrho=0,\\
\partial_1u+\partial_2v=0.
\end{cases}
\end{equation}
Write
\[
    u=\partial_2\varphi,
    \qquad
    v=-\partial_1\varphi,
\]
and perturb the stably stratified state $\varrho=-x_2$ by writing
$\varrho=-x_2+\vartheta$.  Then
\begin{equation}
\label{eq:appendix-hydrostatic-Hamiltonian-system}
\begin{cases}
\partial_t\partial_2^2\varphi-\partial_1\vartheta
   =\{-\partial_2^2\varphi,\varphi\},\\[0.3em]
\partial_t\vartheta+\partial_1\varphi
   =\{-\vartheta,\varphi\}.
\end{cases}
\end{equation}
This system has Hamiltonian energy
\begin{equation}
\label{eq:appendix-hydrostatic-energy}
    E^h
    =\frac12\int_{\mathbb T_L^2}
      \bigl(|\partial_2\varphi|^2+|\vartheta|^2\bigr)\,dx,
\end{equation}
and horizontal pseudo-momentum, up to an overall sign convention,
\begin{equation}
\label{eq:appendix-hydrostatic-PM}
    P^h=-\int_{\mathbb T_L^2}
      \vartheta\,\partial_2^2\varphi\,dx.
\end{equation}

For $k^2\neq0$, define
\begin{equation}
\label{eq:appendix-hydrostatic-normal-modes}
    Z^h_{\sigma,k}
    :=\sigma\pi|k^2|\widehat\varphi(k)
       +\frac12\widehat\vartheta(k).
\end{equation}
Then
\[
    \widehat\varphi(k)
    =\frac{Z^h_{+,k}-Z^h_{-,k}}{2\pi|k^2|},
    \qquad
    \widehat\vartheta(k)=Z^h_{+,k}+Z^h_{-,k},
\]
and
\begin{equation}
\label{eq:appendix-hydrostatic-dispersion}
    \partial_tZ^h_{\sigma,k}
    +i\omega^h_{\sigma,k}Z^h_{\sigma,k}=0,
    \qquad
    \omega^h_{\sigma,k}:=\sigma\frac{k^1}{|k^2|}.
\end{equation}
The hydrostatic horizontal slowness is
\[
    s^h_{\sigma,k}:=\frac{k^1}{\omega^h_{\sigma,k}}
    =\sigma|k^2|.
\]
The two invariants diagonalize as
\begin{equation}
\label{eq:appendix-hydrostatic-modal-invariants}
    E^h=\frac1{L^2}\sum_{\sigma,k}|Z^h_{\sigma,k}|^2,
    \qquad
    P^h=\frac{2\pi}{L^2}\sum_{\sigma,k}
       s^h_{\sigma,k}|Z^h_{\sigma,k}|^2.
\end{equation}

\subsubsection{Hydrostatic interaction coefficient and resonant factorization}

Let
\[
    I_1^h:=\{-\partial_2^2\varphi,\varphi\},
    \qquad
    I_2^h:=\{-\vartheta,\varphi\}.
\]
Their Fourier transforms are
\begin{equation}
\label{eq:appendix-hydrostatic-Fourier-nonlinearity}
\begin{aligned}
\widehat I_1^h(k)
={}&-\frac{16\pi^4}{L^2}
\sum_{p+q=k}(p^1q^2-p^2q^1)(p^2)^2
\widehat\varphi(p)\widehat\varphi(q),\\
\widehat I_2^h(k)
={}&\frac{4\pi^2}{L^2}
\sum_{p+q=k}(p^1q^2-p^2q^1)
\widehat\vartheta(p)\widehat\varphi(q).
\end{aligned}
\end{equation}
Substituting the inverse normal-mode transform, symmetrizing in $p,q$, and
using the reality condition gives
\begin{equation}
\label{eq:appendix-hydrostatic-amplitude-equation}
    \partial_tZ^h_\alpha+i\omega^h_\alpha Z^h_\alpha
    =\frac{\varepsilon}{2}
      \sum_{\beta,\gamma}
      \widetilde V_{\alpha,L}^{\beta\gamma}
      \overline{Z^h_\beta}\,\overline{Z^h_\gamma}\,
      \delta^K_{k_\alpha+k_\beta+k_\gamma,0},
\end{equation}
where
\begin{equation}
\label{eq:appendix-hydrostatic-interaction-coefficient}
\widetilde V_{\alpha,L}^{\beta\gamma}
=L^{-2}\widetilde{\mathcal V}_\alpha^{\beta\gamma},
\qquad
\widetilde{\mathcal V}_\alpha^{\beta\gamma}
=-\pi
\frac{\sigma_\alpha\sigma_\beta\sigma_\gamma}
{|k_\alpha^2||k_\beta^2||k_\gamma^2|}
(k_\beta\times k_\gamma)
(s_\beta^h-s_\gamma^h)
(s_\alpha^h+s_\beta^h+s_\gamma^h).
\end{equation}
The Hamiltonian energy and pseudo-momentum imply the hydrostatic triad
identities
\[
    \widetilde{\mathcal V}_\alpha^{\beta\gamma}
    +\widetilde{\mathcal V}_\beta^{\alpha\gamma}
    +\widetilde{\mathcal V}_\gamma^{\alpha\beta}=0,
\]
and
\[
    s_\alpha^h\widetilde{\mathcal V}_\alpha^{\beta\gamma}
    +s_\beta^h\widetilde{\mathcal V}_\beta^{\alpha\gamma}
    +s_\gamma^h\widetilde{\mathcal V}_\gamma^{\alpha\beta}=0.
\]
On the hydrostatic resonant set, the same one-dimensional nullspace argument
as above gives
\begin{equation}
\label{eq:appendix-hydrostatic-factorization}
    \widetilde{\mathcal V}_\alpha^{\beta\gamma}
       =\omega_\alpha^h\Gamma^h_{\alpha\beta\gamma},
    \qquad
    \widetilde{\mathcal V}_\beta^{\alpha\gamma}
       =\omega_\beta^h\Gamma^h_{\alpha\beta\gamma},
    \qquad
    \widetilde{\mathcal V}_\gamma^{\alpha\beta}
       =\omega_\gamma^h\Gamma^h_{\alpha\beta\gamma}.
\end{equation}

\subsubsection{Hydrostatic weak-turbulence closure}

We take fixed-amplitude random-phase data
\[
    b_{\alpha,L}^h
    =L\sqrt{(n_\alpha^h)^0}\,z_\alpha
\]
with the same independent uniform phases.  Repeating Steps~1--4 for
\eqref{eq:appendix-hydrostatic-amplitude-equation} gives the same three
nondegenerate parity-pairing contributions:
\begin{equation}
\label{eq:appendix-hydrostatic-WKE-V}
\begin{aligned}
\partial_\tau n_\alpha^h
={}&\int
\widetilde{\mathcal V}_\alpha^{\beta\gamma}
\Bigl(
  \widetilde{\mathcal V}_\alpha^{\beta\gamma}n_\beta^hn_\gamma^h
 +\widetilde{\mathcal V}_\gamma^{\alpha\beta}n_\alpha^hn_\beta^h
 +\widetilde{\mathcal V}_\beta^{\alpha\gamma}n_\alpha^hn_\gamma^h
\Bigr)\\
&\quad\times
\delta(k_\alpha+k_\beta+k_\gamma)
\delta(\omega_\alpha^h+\omega_\beta^h+\omega_\gamma^h)
\,d\beta\,d\gamma.
\end{aligned}
\end{equation}
The reality-degenerate gain pairing forces $k_\alpha=0$, while the feedback
degeneracies force $k_\beta=0$ or $k_\gamma=0$.  These exact zero modes are
not part of the hydrostatic wave system and are omitted.  Moreover, every
nonzero collinear interaction vanishes because of the factor
$k_\beta\times k_\gamma$ in
\eqref{eq:appendix-hydrostatic-interaction-coefficient}.

Using \eqref{eq:appendix-hydrostatic-factorization}, the hydrostatic WKE is
\begin{equation}
\label{eq:appendix-hydrostatic-WKE-Gamma}
\begin{aligned}
\partial_\tau n_\alpha^h
={}&\int
(\Gamma^h_{\alpha\beta\gamma})^2\omega_\alpha^h
\Bigl(
  \omega_\alpha^h n_\beta^hn_\gamma^h
 +\omega_\gamma^h n_\alpha^hn_\beta^h
 +\omega_\beta^h n_\alpha^hn_\gamma^h
\Bigr)\\
&\quad\times
\delta(k_\alpha+k_\beta+k_\gamma)
\delta(\omega_\alpha^h+\omega_\beta^h+\omega_\gamma^h)
\,d\beta\,d\gamma.
\end{aligned}
\end{equation}
This completes the independent hydrostatic PDE-to-WKE derivation.

\subsubsection{Hydrostatic energy and pseudo-momentum conservation}
\label{app:hydrostatic-WKE-conservation}

The hydrostatic WKE inherits the two invariants of the hydrostatic Hamiltonian
system.  Define
\[
\mathscr B^h_{\alpha\beta\gamma}(n^h)
:=
(\Gamma^h_{\alpha\beta\gamma})^2
\bigl(
  \omega_\alpha^h n_\beta^h n_\gamma^h
 +\omega_\beta^h n_\alpha^h n_\gamma^h
 +\omega_\gamma^h n_\alpha^h n_\beta^h
\bigr).
\]
Permutation symmetry gives, for every weight $q_\alpha$,
\begin{align}
\frac{d}{d\tau}\int q_\alpha n_\alpha^h\,d\alpha
={}&\frac13\iiint
\bigl(
q_\alpha\omega_\alpha^h
+q_\beta\omega_\beta^h
+q_\gamma\omega_\gamma^h
\bigr)
\mathscr B^h_{\alpha\beta\gamma}(n^h)
\notag\\
&\quad\times
\delta(k_\alpha+k_\beta+k_\gamma)
\delta(\omega_\alpha^h+\omega_\beta^h+\omega_\gamma^h)
\,d\alpha\,d\beta\,d\gamma.
\label{eq:appendix-hydrostatic-invariant-general}
\end{align}
Taking $q_\alpha=1$ gives
\begin{equation}
\label{eq:appendix-hydrostatic-WKE-energy}
    \frac{d}{d\tau}\int n_\alpha^h\,d\alpha=0,
\end{equation}
which is the kinetic counterpart of the hydrostatic Hamiltonian energy
\eqref{eq:appendix-hydrostatic-energy}.  Taking $q_\alpha=s_\alpha^h$ and
using
\[
    s_\alpha^h\omega_\alpha^h=k_\alpha^1
\]
gives
\begin{equation}
\label{eq:appendix-hydrostatic-WKE-PM}
    \frac{d}{d\tau}\int s_\alpha^h n_\alpha^h\,d\alpha=0,
\end{equation}
which, up to the factor $2\pi$ and the sign convention in
\eqref{eq:appendix-hydrostatic-PM}, is the kinetic counterpart of the
hydrostatic pseudo-momentum.  Thus the hydrostatic Hamiltonian energy and
pseudo-momentum appear consistently at the PDE, triad-coefficient, and WKE
levels.

\begin{remark}[Relation between the full and hydrostatic derivations]
Shavit--B\"uhler--Shatah derive the non-hydrostatic two-dimensional
Boussinesq WKE from the normal-mode dynamics~\cite{shavit2024}; we do not
claim that \eqref{eq:appendix-formal-WKE} is derived here for the first time.
The additional point is that the parity pairings, large-box sampling
condition, and kinetic time normalization are displayed explicitly, while
\eqref{eq:appendix-hydrostatic-WKE-Gamma} is obtained independently from the
hydrostatic Hamiltonian PDE.  To the best of our knowledge, this particular
two-branch, two-dimensional hydrostatic PDE-to-WKE derivation has not
previously been written in this form.
\end{remark}

\begin{remark}[The two closures are logically distinct]
The approximation
\[
    \omega_{\sigma,k}
    =\sigma N\frac{k^1}{\sqrt{(k^1)^2+(k^2)^2}}
    \approx \sigma N\frac{k^1}{|k^2|}
\]
relates the two linear dispersions in the anisotropic hydrostatic regime.  It
does not derive \eqref{eq:appendix-hydrostatic-WKE-Gamma} from
\eqref{eq:appendix-formal-WKE}.  The hydrostatic normal modes, interaction
coefficient, resonant measure, parity pairing, and kinetic closure above all
come directly from the hydrostatic PDE.
\end{remark}

\section{Additional Analytical Results}\label{appendix:additional-results}

\begin{lemma}\label{lem:psi_decreasing}
The function
\[
\psi(\theta_\alpha)
= \sin\theta_\alpha\left(\frac{\cos\theta_\alpha}{2}
      - \sqrt{\frac{1}{2}-\frac{\cos^2\theta_\alpha}{4}}\right)
  + \cos\theta_\alpha\sqrt{\frac{1}{2}
      + \cos\theta_\alpha\sqrt{\frac{1}{2}-\frac{\cos^2\theta_\alpha}{4}}},
\qquad \theta_\alpha\in(0,\tfrac{\pi}{3}),
\]
is strictly decreasing on the interval \((0,\pi/3)\).
\end{lemma}

\begin{proof}
Set
\[
    x=\cos\theta_\alpha\in\left(\frac12,1\right),
    \qquad
    A:=\sqrt{1-x^2},
    \qquad
    B:=\sqrt{2-x^2},
\]
and define
\[
    \Phi(x):=\psi(\arccos x).
\]
Since \(B^2=2-x^2\), we have
\[
    (x+B)^2
    =x^2+B^2+2xB
    =2(1+xB).
\]
As \(x+B>0\), it follows that
\[
    \sqrt{\frac12+x\sqrt{\frac12-\frac{x^2}{4}}}
    =
    \sqrt{\frac{1+xB}{2}}
    =
    \frac{x+B}{2}.
\]
Therefore
\[
    \Phi(x)
    =
    \frac12 A(x-B)
    +\frac12 x(x+B).
\]

Using
\[
    A'=-\frac{x}{A},
    \qquad
    B'=-\frac{x}{B},
\]
we obtain
\begin{align*}
    2\Phi'(x)
    &=
    -\frac{x}{A}(x-B)
    +A\left(1+\frac{x}{B}\right)
    +2x+B-\frac{x^2}{B} \\
    &=
    \frac{x(B-x)}{A}
    +A+2x
    +\frac{B^2+xA-x^2}{B} \\
    &=
    \frac{x(B-x)}{A}
    +A+2x
    +\frac{A(2A+x)}{B}.
\end{align*}
Now \(A>0\), \(B>0\), and \(x>0\). Moreover,
\[
    B^2-x^2
    =2(1-x^2)
    =2A^2>0,
\]
so \(B>x\). Hence every term in the last expression is strictly
positive, and therefore
\[
    \Phi'(x)>0,
    \qquad x\in\left(\frac12,1\right).
\]

Finally, since \(x=\cos\theta_\alpha\),
\[
    \psi'(\theta_\alpha)
    =
    -\sin\theta_\alpha\,
    \Phi'(\cos\theta_\alpha)<0
\]
for every \(\theta_\alpha\in(0,\pi/3)\). Thus
\(\psi\) is strictly decreasing on \((0,\pi/3)\).
\end{proof}
\begin{lemma}\label{lem:positivity_decreasing_term}
The function
\[
    \psi_1(\theta_\alpha)
    =
    -\frac{\sin\theta_\alpha\cos\theta_\alpha}{2}
    +
    \cos\theta_\alpha
    \sqrt{1-\frac{\cos^2\theta_\alpha}{4}},
    \qquad
    0\le\theta_\alpha\le\frac{\pi}{2},
\]
is nonnegative and strictly decreasing on \([0,\frac{\pi}{2}]\), and is
positive on \([0,\frac{\pi}{2})\).
\end{lemma}

\begin{proof}
Write
\[
    s:=\sin\theta_\alpha,
    \qquad
    c:=\cos\theta_\alpha.
\]
Since
\[
    1-\frac{c^2}{4}
    =
    \frac{3+s^2}{4},
\]
we have
\begin{align*}
    \psi_1(\theta_\alpha)
    &=
    \frac{c}{2}
    \left(
        \sqrt{3+s^2}-s
    \right) \\
    &=
    \frac{3c}{
        2\bigl(\sqrt{3+s^2}+s\bigr)
    }.
\end{align*}
On \([0,\frac{\pi}{2}]\), the numerator \(c\) is strictly decreasing,
whereas the positive denominator
\[
    \sqrt{3+s^2}+s
\]
is strictly increasing.  It follows that \(\psi_1\) is strictly decreasing.
Moreover,
\[
    \psi_1(\theta_\alpha)>0
    \quad\text{for}\quad
    0\le\theta_\alpha<\frac{\pi}{2},
    \qquad
    \psi_1\left(\frac{\pi}{2}\right)=0.
\]
This proves the claim.
\end{proof}

\begin{lemma}[Rotating high--low--high cusp]
\label{lem:rotating-high-low-high-cusp}
Let \(0<f<\frac12\), and define
\[
    \Omega_f(\theta)
    :=
    \sqrt{\cos^2\theta+f^2\sin^2\theta}.
\]
Equivalently, for \(k\neq0\),
\[
    \Omega_f(k)
    =
    \sqrt{
        \frac{(k_1)^2+f^2(k_2)^2}{|k|^2}
    }.
\]
Consider a high--low--high configuration in which the low wave vector \(q\)
collapses to zero.  Up to relabeling of the triad, the reduced frequency
resonance function has the local form
\[
    \widetilde\Xi_{1,f}^{\mathrm{HLH}}(q)
    :=
    \Omega_f(k_\alpha)
    +
    \Omega_f(-k_\alpha-q)
    -
    \Omega_f(q),
    \qquad
    k_\alpha=(\cos\theta_\alpha,\sin\theta_\alpha).
\]
Then the collapsed low-mode resonance condition is
\[
    \Omega_f(\theta_q)
    =
    2\Omega_f(\theta_\alpha),
\]
where \(\theta_q\) denotes the limiting direction of \(q\).  Hence the cusp
threshold is
\[
    \Omega_f(\theta_\alpha)=\frac12.
\]
Equivalently,
\[
    |\cos\theta_\alpha|
    =
    \sqrt{
        \frac{\frac14-f^2}{1-f^2}
    }.
\]
In the first quadrant, the critical angle is
\[
    \theta_\alpha^\ast(f)
    =
    \arccos
    \sqrt{
        \frac{\frac14-f^2}{1-f^2}
    }.
\]
At this critical angle, the resonance curve reaches the collapsed low mode
quadratically.  More precisely, if
\[
    q=\rho(\cos(\pi+\eta),\sin(\pi+\eta)),
    \qquad
    0<\rho\ll1,
    \qquad
    |\eta|\ll1,
\]
then
\[
    \widetilde\Xi_{1,f}^{\mathrm{HLH}}(q)=0
\]
has the local form
\[
    \rho=C_f\eta^2+O(|\eta|^3),
    \qquad
    C_f>0.
\]
Thus \(q=0\) is a cusp point of the high--low--high resonance curve.
\end{lemma}

\begin{proof}
As \(q\to0\), the magnitude of \(q\) collapses, but its direction remains as a
limiting angular parameter.  Since
\[
    -k_\alpha-q\to -k_\alpha
\]
and
\[
    \Omega_f(-k_\alpha)=\Omega_f(k_\alpha),
\]
the collapsed limit of the reduced resonance function is
\[
    \widetilde\Xi_{1,f}^{\mathrm{HLH}}(q)
    \to
    2\Omega_f(\theta_\alpha)-\Omega_f(\theta_q).
\]
Therefore a collapsed low-mode resonance requires
\[
    \Omega_f(\theta_q)=2\Omega_f(\theta_\alpha).
\]
Since
\[
    f\le \Omega_f(\theta)\le1,
\]
the limiting low-mode directions coalesce when
\[
    2\Omega_f(\theta_\alpha)=1.
\]
Thus the cusp threshold is
\[
    \Omega_f(\theta_\alpha)=\frac12.
\]
Using
\[
    \Omega_f^2(\theta)
    =
    f^2+(1-f^2)\cos^2\theta,
\]
we obtain
\[
    f^2+(1-f^2)\cos^2\theta_\alpha=\frac14.
\]
Hence
\[
    |\cos\theta_\alpha|
    =
    \sqrt{
        \frac{\frac14-f^2}{1-f^2}
    }.
\]

It remains to check the quadratic contact.  At the critical angle
\(\theta_\alpha=\theta_\alpha^\ast(f)\), one has
\[
    \Omega_f(\theta_\alpha)=\frac12.
\]
We expand near the horizontal limiting low-mode direction
\[
    \theta_q=\pi.
\]
Write
\[
    q=\rho(\cos(\pi+\eta),\sin(\pi+\eta)).
\]
Since \(\Omega_f\) has a nondegenerate maximum at the horizontal directions,
\[
    \Omega_f(\pi+\eta)
    =
    1-\frac{1-f^2}{2}\eta^2
    +
    O(\eta^4).
\]
On the other hand,
\[
    -k_\alpha-q
    =
    -k_\alpha+\rho(1,0)+O(\rho|\eta|),
\]
and therefore
\[
    \Omega_f(-k_\alpha-q)
    =
    \Omega_f(\theta_\alpha)
    +
    B_f(\theta_\alpha)\rho
    +
    O(\rho|\eta|+\rho^2),
\]
where
\[
    B_f(\theta_\alpha)
    =
    \nabla\Omega_f(-k_\alpha)\cdot(1,0)
    =
    -\frac{
        (1-f^2)\cos\theta_\alpha\sin^2\theta_\alpha
    }{
        \Omega_f(\theta_\alpha)
    }.
\]
At
\[
    \theta_\alpha=\theta_\alpha^\ast(f),
\]
this coefficient is strictly negative:
\[
    B_f(\theta_\alpha^\ast(f))<0.
\]
Substituting the expansions into
\(\widetilde\Xi_{1,f}^{\mathrm{HLH}}\) gives
\[
\begin{aligned}
    \widetilde\Xi_{1,f}^{\mathrm{HLH}}(\rho,\eta)
    &=
    \Omega_f(\theta_\alpha)
    +
    \Omega_f(-k_\alpha-q)
    -
    \Omega_f(\pi+\eta)
    \\
    &=
    B_f(\theta_\alpha)\rho
    +
    \frac{1-f^2}{2}\eta^2
    +
    O(\rho|\eta|+\rho^2+\eta^4).
\end{aligned}
\]
Since \(B_f(\theta_\alpha^\ast(f))\neq0\), solving
\[
    \widetilde\Xi_{1,f}^{\mathrm{HLH}}(\rho,\eta)=0
\]
gives
\[
    \rho
    =
    \frac{1-f^2}{-2B_f(\theta_\alpha^\ast(f))}
    \eta^2
    +
    O(|\eta|^3).
\]
The leading coefficient is positive because
\[
    B_f(\theta_\alpha^\ast(f))<0.
\]
Thus
\[
    \rho=C_f\eta^2+O(|\eta|^3),
    \qquad
    C_f>0.
\]
This proves that the resonance curve reaches the collapsed low mode
quadratically.  In Cartesian coordinates near \(q=0\), this is the usual cusp
behavior
\[
    (q_2)^2\sim C\,|q_1|^3.
\]
\end{proof}
\bibliographystyle{plain}
\bibliography{references}

@article{shavit2024,
  title={Sign-indefinite invariants shape turbulent cascades},
  author={Shavit, Michal and B{\"u}hler, Oliver and Shatah, Jalal},
  journal={Physical Review Letters},
  volume={133},
  number={1},
  pages={014001},
  year={2024},
  publisher={APS}
}

@article{shavit2025,
  title={Turbulent spectrum of 2D internal gravity waves},
  author={Shavit, Michal and B{\"u}hler, Oliver and Shatah, Jalal},
  journal={Physical Review Letters},
  volume={134},
  number={5},
  pages={054101},
  year={2025},
  publisher={APS}
}

@article{labarre2024kinetics,
  title={On the kinetics of internal gravity waves beyond the hydrostatic regime},
  author={Labarre, Vincent and Lanchon, Nicolas and Cortet, Pierre-Philippe and Krstulovic, Giorgio and Nazarenko, Sergey},
  journal={Journal of Fluid Mechanics},
  volume={998},
  pages={A17},
  year={2024},
  publisher={Cambridge University Press}
}

@article{caillol2000kinetic,
  title={Kinetic equations and stationary energy spectra of weakly nonlinear internal gravity waves},
  author={Caillol, Ph and Zeitlin, V},
  journal={Dynamics of atmospheres and oceans},
  volume={32},
  number={2},
  pages={81--112},
  year={2000},
  publisher={Elsevier}
}

@article{LvovTabak2001,
  title={Hamiltonian formalism and the Garrett-Munk spectrum of internal waves in the ocean},
  author={Lvov, Yuri V and Tabak, Esteban G},
  journal={Physical review letters},
  volume={87},
  number={16},
  pages={168501},
  year={2001},
  publisher={APS}
}

@article{lvov2010oceanic,
  title={Oceanic internal-wave field: Theory of scale-invariant spectra},
  author={Lvov, Yuri V and Polzin, Kurt L and Tabak, Esteban G and Yokoyama, Naoto},
  journal={Journal of Physical Oceanography},
  volume={40},
  number={12},
  pages={2605--2623},
  year={2010}
}

@article{lvov2004hamiltonian,
  title={A Hamiltonian formulation for long internal waves},
  author={Lvov, Yuri and Tabak, Esteban G},
  journal={Physica D: Nonlinear Phenomena},
  volume={195},
  number={1-2},
  pages={106--122},
  year={2004},
  publisher={Elsevier}
}

@article{lvov2012resonant,
  title={Resonant and near-resonant internal wave interactions},
  author={Lvov, Yuri V and Polzin, Kurt L and Yokoyama, Naoto},
  journal={Journal of Physical Oceanography},
  volume={42},
  number={5},
  pages={669--691},
  year={2012}
}

@article{hasselmann1962non,
  title={On the non-linear energy transfer in a gravity-wave spectrum Part 1. General theory},
  author={Hasselmann, Klaus},
  journal={Journal of Fluid Mechanics},
  volume={12},
  number={4},
  pages={481--500},
  year={1962},
  publisher={Cambridge University Press}
}

@article{muller1975dynamics,
  title={On the dynamics of internal waves in the deep ocean},
  author={M{\"u}ller, Peter and Olbers, Dirk J},
  journal={Journal of geophysical research},
  volume={80},
  number={27},
  pages={3848--3860},
  year={1975},
  publisher={Wiley Online Library}
}

@article{olbers1976nonlinear,
  title={Nonlinear energy transfer and the energy balance of the internal wave field in the deep ocean},
  author={Olbers, Dirk J},
  journal={Journal of Fluid mechanics},
  volume={74},
  number={2},
  pages={375--399},
  year={1976},
  publisher={Cambridge University Press}
}

@article{mccomas1977resonant,
  title={Resonant interaction of oceanic internal waves},
  author={McComas, C Henry and Bretherton, Francis P},
  journal={Journal of Geophysical Research},
  volume={82},
  number={9},
  pages={1397--1412},
  year={1977},
  publisher={Wiley Online Library}
}

@article{mccomas1981dynamic,
  title={The dynamic balance of internal waves},
  author={McComas, C Henry and M{\"u}ller, Peter},
  journal={Journal of Physical Oceanography},
  volume={11},
  number={7},
  pages={970--986},
  year={1981}
}

@article{mccomas1981time,
  title={Time scales of resonant interactions among oceanic internal waves},
  author={McComas, C Henry and M{\"u}ller, Peter},
  journal={Journal of physical oceanography},
  volume={11},
  number={2},
  pages={139--147},
  year={1981}
}

@book{nazarenko2011wave,
  title={Wave turbulence},
  author={Nazarenko, Sergey},
  volume={825},
  year={2011},
  publisher={Springer Science \& Business Media}
}

@article{ampatzoglou2025derivation,
  title={Derivation of the kinetic wave equation for quadratic dispersive problems in the inhomogeneous setting},
  author={Ampatzoglou, Ioakeim and Collot, Charles and Germain, Pierre},
  journal={American Journal of Mathematics},
  volume={147},
  number={4},
  pages={1053--1158},
  year={2025},
  publisher={Johns Hopkins University Press}
}

@article{ampatzoglou2022global,
  title={Global well-posedness and stability of the inhomogeneous kinetic wave equation near vacuum},
  author={Ampatzoglou, Ioakeim},
  journal={Kinetic and Related Models},
  year={2024},
  publisher={American Institute of Mathematical Sciences}
}

@article{AmpatzoglouMillerPavlovicTaskovic2025,
  title={Inhomogeneous wave kinetic equation and its hierarchy in polynomially weighted \(L^\infty\) spaces},
  author={Ampatzoglou, Ioakeim and Miller, Joseph K and Pavlovi{\'c}, Nata{\v{s}}a and Taskovi{\'c}, Maja},
  journal={Communications in Partial Differential Equations},
  volume={50},
  number={6},
  pages={723--765},
  year={2025},
  publisher={Taylor \& Francis}
}

@article{ampatzoglou2024scattering,
  title={Global existence of strong solutions to the inhomogeneous kinetic wave equation},
  author={Ampatzoglou, Ioakeim and L{\'e}ger, Tristan},
  journal={Communications in Mathematical Physics},
  volume={407},
  number={7},
  pages={144},
  year={2026},
  publisher={Springer}
}

@article{AmpatzoglouLegerIllposed2025,
  title={On the ill-posedness of kinetic wave equations},
  author={Ampatzoglou, Ioakeim and L{\'e}ger, Tristan},
  journal={Nonlinearity},
  volume={38},
  number={11},
  pages={115004},
  year={2025},
  publisher={IOP Publishing}
}

@article{Menegaki2024,
  title={$ L^{2}$-stability near equilibrium for the 4 waves kinetic equation},
  author={Menegaki, Angeliki},
  journal={Kinetic and Related Models},
  volume={17},
  number={4},
  pages={514--532},
  year={2024},
  publisher={Kinetic and Related Models}
}

@article{EscobedoMenegaki2024,
  title={Instability of singular equilibria of a wave kinetic equation},
  author={Escobedo, Miguel and Menegaki, Angeliki},
  journal={arXiv preprint arXiv:2406.05280},
  year={2024}
}

@article{CollotDietertGermain2024,
  title={Stability and Cascades for the Kolmogorov--Zakharov Spectrum of Wave Turbulence: C. Collot et al.},
  author={Collot, Charles and Dietert, Helge and Germain, Pierre},
  journal={Archive for Rational Mechanics and Analysis},
  volume={248},
  number={1},
  pages={7},
  year={2024},
  publisher={Springer}
}

@inproceedings{deng2021derivation,
  title={On the derivation of the wave kinetic equation for NLS},
  author={Deng, Yu and Hani, Zaher},
  booktitle={Forum of Mathematics, Pi},
  volume={9},
  pages={e6},
  year={2021},
  organization={Cambridge University Press}
}

@article{deng2021propagation,
  title={Propagation of chaos and the higher order statistics in the wave kinetic theory},
  author={Deng, Yu and Hani, Zaher},
  journal={arXiv preprint arXiv:2110.04565},
  year={2021}
}

@article{deng2023full,
  title={Full derivation of the wave kinetic equation},
  author={Deng, Yu and Hani, Zaher},
  journal={Inventiones mathematicae},
  volume={233},
  number={2},
  pages={543--724},
  year={2023},
  publisher={Springer}
}

@article{deng2023long,
  title={Long time justification of wave turbulence theory},
  author={Deng, Yu and Hani, Zaher},
  journal={arXiv preprint arXiv:2311.10082},
  year={2023}
}

@article{deng2024long,
  title={Long time derivation of the Boltzmann equation from hard sphere dynamics},
  author={Deng, Yu and Hani, Zaher and Ma, Xiao},
  journal={arXiv preprint arXiv:2408.07818},
  year={2024}
}

@article{deng2025hilbert,
  title={Hilbert's sixth problem: derivation of fluid equations via Boltzmann's kinetic theory},
  author={Deng, Yu and Hani, Zaher and Ma, Xiao},
  journal={arXiv preprint arXiv:2503.01800},
  year={2025}
}

@article{BuckmasterGermainHaniShatah2021,
  title={Onset of the wave turbulence description of the longtime behavior of the nonlinear Schr{\"o}dinger equation},
  author={Buckmaster, Tristan and Germain, Pierre and Hani, Zaher and Shatah, Jalal},
  journal={Inventiones mathematicae},
  volume={225},
  number={3},
  pages={787--855},
  year={2021},
  publisher={Springer}
}

@article{GermainIonescuTran2020,
  title={Optimal local well-posedness theory for the kinetic wave equation},
  author={Germain, Pierre and Ionescu, Alexandru D and Tran, Minh-Binh},
  journal={Journal of Functional Analysis},
  volume={279},
  number={4},
  pages={108570},
  year={2020},
  publisher={Elsevier}
}

@article{GermainLaZhang2025,
  title={Local well-posedness for the kinetic mmt model},
  author={Germain, Pierre and La, Joonhyun and Zhang, Katherine Zhiyuan},
  journal={Communications in Mathematical Physics},
  volume={406},
  number={1},
  pages={18},
  year={2025},
  publisher={Springer}
}

@incollection{grad1958,
  title={Principles of the kinetic theory of gases},
  author={Grad, Harold},
  booktitle={Thermodynamik der Gase/Thermodynamics of Gases},
  pages={205--294},
  year={1958},
  publisher={Springer}
}

@article{garrett1972space,
  title={Space-time scales of internal waves},
  author={Garrett, Christopher and Munk, Walter},
  journal={Geophysical Fluid Dynamics},
  volume={3},
  number={3},
  pages={225--264},
  year={1972},
  publisher={Taylor \& Francis}
}

@article{garrett1975space,
  title={Space-time scales of internal waves: A progress report},
  author={Garrett, Christopher and Munk, Walter},
  journal={Journal of Geophysical Research},
  volume={80},
  number={3},
  pages={291--297},
  year={1975},
  publisher={Wiley Online Library}
}

@article{munk1981internal,
  title={Internal waves and small-scale processes},
  author={Munk, WH},
  journal={Evolution of physical oceanography},
  year={1981},
  publisher={MIT press}
}

@article{collot2025longer,
  title={Derivation of the homogeneous kinetic wave equation: longer time scales},
  author={Collot, Charles and Germain, Pierre},
  journal={Journal of Functional Analysis},
  pages={111179},
  year={2025},
  publisher={Elsevier}
}

@article{collot2025homogeneous,
  title={On the derivation of the homogeneous kinetic wave equation},
  author={Collot, Charles and Germain, Pierre},
  journal={Communications on Pure and Applied Mathematics},
  volume={78},
  number={4},
  pages={856--909},
  year={2025},
  publisher={Wiley Online Library}
}

@article{vassilev2025,
  title={One-Dimensional Wave Kinetic Theory: KD Vassilev},
  author={Vassilev, Katja D},
  journal={Communications in Mathematical Physics},
  volume={406},
  number={12},
  pages={293},
  year={2025},
  publisher={Springer}
}

@article{wu2025,
  title={Rigorous Derivation of the Wave Kinetic Equation for the $\beta$-FPUT System},
  author={Wu, Boyang},
  journal={arXiv preprint arXiv:2506.02948},
  year={2025}
}

@article{vassilevwu2026,
  title={Rigorous Derivation of the Wave Kinetic Equation for full $\beta $-FPUT System},
  author={Vassilev, Katja and Wu, Boyang},
  journal={arXiv preprint arXiv:2605.19308},
  year={2026}
}

@incollection{cercignani1988,
  title={The boltzmann equation},
  author={Cercignani, Carlo},
  booktitle={The Boltzmann equation and its applications},
  pages={40--103},
  year={1988},
  publisher={Springer}
}

@article{DiPernaLions1989,
  title={On the Cauchy problem for Boltzmann equations: global existence and weak stability},
  author={DiPerna, Ronald J and Lions, Pierre-Louis},
  journal={Annals of Mathematics},
  pages={321--366},
  year={1989},
  publisher={JSTOR}
}

@article{villani2002,
  title={A review of mathematical topics in collisional kinetic theory},
  author={Villani, C{\'e}dric},
  journal={Handbook of mathematical fluid dynamics},
  volume={1},
  pages={71--74},
  year={2002},
  publisher={Elsevier}
}

@article{AlexandreDesvillettesVillaniWennberg2000,
  title={Entropy dissipation and long-range interactions},
  author={Alexandre, Radjesvarane and Desvillettes, Laurent and Villani, C{\'e}dric and Wennberg, Bernt},
  journal={Archive for rational mechanics and analysis},
  volume={152},
  number={4},
  pages={327--355},
  year={2000},
  publisher={Springer}
}

@article{bianchini2025ill,
  title={Ill-Posedness of the Hydrostatic Euler--Boussinesq Equations and Failure of Hydrostatic Limit: R. Bianchini, M. Coti Zelati, L. Ertzbischoff},
  author={Bianchini, Roberta and Coti Zelati, Michele and Ertzbischoff, Lucas},
  journal={Communications in Mathematical Physics},
  volume={406},
  number={10},
  pages={254},
  year={2025},
  publisher={Springer}
}

@article{shavit2026rotating,
  title={Wave turbulence of inertia--gravity waves: a theory for the oceanic spectrum},
  author={Shavit, Michal and B{\"u}hler, Oliver and Shatah, Jalal},
  journal={arXiv preprint arXiv:2601.01476},
  year={2026}
}

@article{PanWu2026,
  title={Local-in-Time Existence of $ L^1$ solutions to the Gravity Water Wave Kinetic Equation},
  author={Pan, Yulin and Wu, Xiaoxu},
  journal={arXiv preprint arXiv:2603.10882},
  year={2026}
}

@article{DengIonescuPusateriGravityI,
  title={On the wave turbulence theory of 2D gravity waves, I: deterministic energy estimates},
  author={Deng, Yu and Ionescu, Alexandru D and Pusateri, Fabio},
  journal={Communications on Pure and Applied Mathematics},
  volume={78},
  number={2},
  pages={211--322},
  year={2025},
  publisher={Wiley Online Library}
}

@article{DengIonescuPusateriGravityII,
  title={On the wave turbulence theory of 2D gravity waves, II: propagation of randomness},
  author={Deng, Yu and Ionescu, Alexandru and Pusateri, Fabio},
  journal={arXiv preprint arXiv:2504.14304},
  year={2025}
}

@article{germain2026stability,
  title={Stability of Rayleigh--Jeans Equilibria in the Kinetic FPU Equation},
  author={Germain, Pierre and La, Joonhyun and Menegaki, Angeliki},
  journal={Archive for Rational Mechanics and Analysis},
  volume={250},
  number={4},
  pages={65},
  year={2026},
  publisher={Springer}
}

@article{xiang2025long,
  title={Long-Time Existence and Behavior of Solutions to the Inhomogeneous Kinetic FPU Equation},
  author={Xiang, Haoling},
  journal={arXiv preprint arXiv:2512.21187},
  year={2025}
}

@article{escobedo2025entropy,
  title={Entropy maximizers for kinetic wave equations set on tori},
  author={Escobedo, Miguel and Germain, Pierre and La, Joonhyun and Menegaki, Angeliki},
  journal={Bulletin of the London Mathematical Society},
  volume={57},
  number={12},
  pages={3977--3990},
  year={2025},
  publisher={Wiley Online Library}
}

@article{lukkarinen2008anomalous,
  title={Anomalous energy transport in the FPU-$\beta$ chain},
  author={Lukkarinen, Jani and Spohn, Herbert},
  journal={Communications on Pure and Applied Mathematics: A Journal Issued by the Courant Institute of Mathematical Sciences},
  volume={61},
  number={12},
  pages={1753--1786},
  year={2008},
  publisher={Wiley Online Library}
}
\medskip
\noindent\textsc{Department of Mathematics,
Imperial College London,
London SW7 2AZ, United Kingdom}

\noindent\textit{Email address:} \texttt{h.xiang23@imperial.ac.uk}
\end{document}